\documentclass[12pt,a4paper]{article}
\usepackage{amsmath,bm}
\usepackage{amssymb}
\usepackage{cases}
\usepackage{bbm}
\usepackage{mathrsfs}
\usepackage{graphicx}
\usepackage{amsfonts}
\usepackage{enumerate}
\usepackage{theorem}
\usepackage{color}
\usepackage{cite}
\makeatletter
\def\tank#1{\protected@xdef\@thanks{\@thanks
 \protect\footnotetext[0]{#1}}}
\def\bigfoot{

 \@footnotetext}
\makeatother

\newcommand{\ea}{\end{array}}

\allowdisplaybreaks

\newtheorem{theorem}{Theorem}[section]
\newtheorem{lem}{Lemma}[section]
\newtheorem{prp}[theorem]{Proposition}
\newtheorem{thm}[theorem]{Theorem}

\newtheorem{remark}{Remark}[section]

\def\beq{\begin{equation}}
\def\nneq{\end{equation}}

\def\bthm{\begin{thm}}
\def\nthm{\end{thm}}

\def\Rd{\mathbb{R}^d}
\def\R{\mathbb{R}}
\def\Rdd{\mathbb{R}^{d+1}}
\def\a{\alpha}

\def\d{\delta}

\def\t{\theta}
\def\g{\gamma}
\def\G{\Gamma}
\def\c{\cdot}
\def\p{\partial}

\def\les{\lesssim}
\def\o{\omega}
\def\O{ \mathcal{O}}

\def\l{\lambda}

\def\e{\varepsilon}

{\theorembodyfont{\rmfamily}
}

\title{Strong Solutions to Reflecting Stochastic Differential Equations with Singular Drift}

\footnotesize{\author{
 Saisai Yang$^{1,}$\thanks{yangss@mail.ustc.edu.cn},\ \
 Tusheng Zhang$^{2,}$\thanks{Tusheng.Zhang@manchester.ac.uk}\\
 {\em $^1$ School of Mathematical Sciences,}\\
 {\em University of Science and Technology of China,}\\
 {\em Hefei, 230026, China}\\
 {\em $^2$ School of Mathematics, University of Manchester,}\\
 {\em Oxford Road, Manchester, M13 9PL, UK}\\
}
\date{}
\newenvironment{proof}{\par\noindent{\bf Proof:}}{\hspace*{\fill}$\blacksquare$\par}
\begin{document}
\maketitle
%\begin{minipage}{140mm}
%\begin{center}
\noindent \textbf{Abstract:}
In this paper, we prove that there exists a unique strong solution to  reflecting stochastic differential equations with merely measurable drift giving an affirmative answer to the longstanding problem. This is done
through Zvonkin transformation and a careful analysis of the transformed reflecting stochastic differential equations on non-smooth  time-dependent domains.
%\end{center}

%\end{minipage}

\vspace{4mm}

%\noindent \textbf{AMS Subject Classification}: Primary 60H15 Secondary 35R60, 37L55.

\vspace{3mm}
\noindent \textbf{Key Words:} strong solution; reflecting stochastic differential equations; singular drift; Zvonkin transformation; time-dependent domain.

\numberwithin{equation}{section}
\vskip 0.3cm
\noindent \textbf{AMS Mathematics Subject Classification:} Primary 60H10. Secondary 60J60.
\tableofcontents
\section{Introduction}
\indent
Let $D$ be a bounded domain in $\mathbb{R}^d$ with a $C^3$ boundary $\partial D$ and $b(t,x)$ a measurable $\Rd$-valued function bounded on $[0,T]\times D$ for every $T>0$. In this paper, we are concerned with the strong solutions to  reflecting stochastic differential equations (SDEs) on the domain $D$ with the singular drift $b$. The purpose  is to give an affirmative answer to the longstanding problem of the existence and uniqueness of strong  solutions.

More rigorously, given a probability space $(\Omega, \mathcal{F}, (\mathcal{F})_{t\ge 0}, P)$ satisfying the usual assumptions and a $d$-dimensional standard Bronwian motion $W_t, t\geq 0$ on the probability space. Denote by $n(x)$ the unit inward  normal to the boundary $\partial D$. We aim to  show that for any $x\in \bar D$, there exists a unique pair of continuous adapted processes $(X_t,L_t)$ solving the  reflecting stochastic differential equation below, namely,  $X_t \in \overline{D}$ for all $t \ge 0$, $P$-$a.e.$, $L_t$ is a continuous process of bounded variation  with values in $\R^d$, and the following equation holds:

\begin{align} \label{1.1}
\left\{
\begin{aligned}
& X_{t}=x+W_{t}+\int_{0}^{t} b (s, X_{s} ) d s+L_{t}, \\
& |L|_{t}=\int_{0}^{t} I_{ \{X_{s} \in \partial D \}} d |L|_{s}, \\
& L_{t}=\int_{0}^{t} n(X_{s}) d |L|_{s}, \\
\end{aligned}
\right.
\end{align}
where $|L|_t$ is the total variation of $L_t$.

\vskip 0.5cm

Reflecting SDEs have been investigated by many authors when the coefficients are smooth /lipschitz. H. Tanaka in \cite{Tanaka} obtained the strong solutions of the reflecting SDEs in a convex domain based on solving the corresponding Skorokhod problem. P.L. Lions and A.S. Sznitman in \cite{Lions} studied  the reflecting SDEs by a penalized  method in a  $C^3$-domain. P. Dupuis and H. Ishii in \cite{Dupuis3} obtained  the existence and uniqueness of the strong solutions to reflecting SDEs in a general domain, which only requires the directions of reflection to be $C^2$.
In one-dimensional case, T.S. Zhang in \cite{Zhang} obtained the strong solution to reflecting SDEs with
locally bounded drifts using crucially  the comparison theorem. In multidimensional case,
P. Mar\'{\i}n-Rubio and J. Real in \cite{Marin} obtained the strong solutions to
reflecting SDEs when the drifts satisfy a certain monotonicity condition.

\vskip 0.4cm

On the other hand, strong solutions have been studied by many people for  stochastic differential equations with singular drift. In the celebrated work \cite{Zvonkin},  Zvonkin introduced a quasi-isometric transformation of the phase space that can convert a stochastic differential equation with a non-zero singular drift into a SDE without drift. This method is now called Zvonkin transformation. There are many papers (particularly in recent years) devoted to extending the  Zvonkin transformation in various ways to obtain the strong solutions of stochastic differential equations with  singular coefficients. We mention  \cite{Gyongy}, \cite{Krylov2}, \cite{Veretennikov}, \cite{ZhangX2} and \cite{ZhangX1}.

\vskip 0.5cm

The purpose of this paper is to establish the existence and uniqueness of strong solutions of reflecting SDEs with drifts which are merely measurable. The existence of weak solutions of the reflecting SDE (\ref{1.1}) is clear by using the Girsanov transform. To get the pathwise strong solution, the key is to prove the pathwise uniqueness of the equation (\ref{1.1}). Because of the singularity of the drifts we could not rely on solving  the deterministic Skorohod problem, see for instance \cite{Dupuis3} and \cite{Lions}. We will use the Zvonkin transformation. However difficulties immediately arise. Zvonkin transformation maps the domain $D$ into a family of time-dependent domains which are not as regular as the original one. Thus, after the transformation we are bound to  establish the pathwise uniqueness of  reflecting SDEs in  time dependent, non-smooth domains. Moreover, the reflecting directions of the  transformed process are not as smooth as the original inward normal and also the coefficients of the transformed reflecting SDEs are not Lipschitz. The existing results on reflecting SDEs in time dependent domains can not be applied. A large part of our work is  to carry out a careful analysis of  the transformed, time dependent domains and the time dependent reflecting directions to establish the necessary regularities required. To get the pathwise uniqueness, eventually we also need to construct a family of auxiliary test functions. This is done  in a similar way as that in  \cite{Dupuis3} and \cite{Lundstrom}.

\vskip 0.5cm
Throughout this paper, we assume $b(t,x)$ is bounded on $[0,T] \times D$ for every $T>0$.

\vskip 0.5cm

 Now we describe the content and organization of the paper in more details. In Section 2, we consider the following parabolic partial differential equation (PDE) associated with the singular drift on the domain $(0, T) \times D$, equipped with the Neumann boundary condition:
 \begin{align}  \nonumber
 \left\{
\begin{aligned}
& \partial_{t} u^T(t,x)+\frac{1}{2} \Delta_{x} u^T(t,x)+b(t,x) \cdot \nabla_{x} u^T(t,x)=0,&&\forall (t,x) \in (0, T) \times D, \\
&  \frac{\partial u^T}{\partial n}(t, x)=n(x),&&\forall (t,x) \in (0, T) \times \partial D, \\
& u^T(T, x)=x, && \forall x \in D.\\
\end{aligned}
\right.
\end{align}
 We provide regularities of the solution $u^T(t,x)$, which will be used in subsequent sections. Especially, we show that there exists an open set $G  \supset \bar D$ such that the extension of $u^T(t,\c)$ on $G$ is a homeomorphism for $t \in [0,  T]$ and $\tilde u^T(t,x):=(t, u^T(t,x))$ is an open mapping on $(0,T)\times G$.
\vskip 0.4cm

 In Section 3,  we study the time dependent domains $u^T(t, D), \ t \in [0, T]$, the images of  domain $D$ under the solution mappings $u^T(t, \cdot)$, $t\in [0, T]$. Among other things, we showed that the domains $u^T(t, D)$ satisfy the exterior and interior cone conditions when $T$ is sufficiently small.  Regularities of the time dependent vector fields, $\gamma(t,x):=n((u^T)^{-1}(t,x))$,  of the reflecting directions are also established. Here $(u^T)^{-1}(t,x)$ denotes the inverse function of $u^T(t,x)$.
\vskip 0.4cm
 In Section 4, we consider the flows associated with the time dependent vector fields of reflecting directions:
  \begin{align} \nonumber
 \left\{
\begin{aligned}
& y(t, x, 0)=x,  \\
&  \partial_{r} y(t,x, r)=\g(t, y(t, x,r)), \ r \in \R.
\end{aligned}
  \right.
\end{align}
We will provide a number of regularity results of the hitting times $\Gamma(t,x)$ of the flows on certain hyperplane. These hitting times will be used to construct test functions for proving the pathwise uniqueness of the transformed reflecting SDEs. Roughly speaking, since $\g$ only belongs to some Sobolev space on $\tilde D:=\tilde u^T((0,T)\times D)$, to ensure the regularity of the hitting times $\Gamma(t,x)$, we need to prove that if $(t,x) \in \tilde D$, then $(t,y(t,x,r))$ lies in $\tilde D$ before $y(t,x,r)$ hits the hyperplane ($i.e.$ for $r \in (0,\G(t,x)]$). At the end of this section, we will establish some smooth approximations of $\G(t,x)$,  which will be used later to show that $y(\c,\c ,\Gamma(\c,\c))$ belongs to some Sobolev space on $\tilde D$.
\vskip 0.4cm

In Section 5, a family of auxiliary functions is constructed. We first construct the functions locally in some neighborhoods of the points on the boundary of the domain $\tilde D$ and then piece them together through a finite cover of the boundary. These test functions will be used to prove the pathwise uniqueness of the solutions of reflecting stochastic differential equations. We also introduce the stochastic Gronwall's inequality and Krylov's estimate.
\vskip 0.4cm

In Section 6, we will establish the existence and uniqueness of strong solutions to the reflecting SDEs (\ref{1.1}). The existence of a weak solution follows from the Girsanov theorem. The strong solution is obtained by proving the pathwise uniqueness of the solutions. To this end, we first establish a generalized It\^o's formula for the solution $X_t$ of the  reflecting SDEs using the Krylov's estimate. Then we will use the auxiliary functions to eliminate the local times of the transformed processes $u^T(t,X_t)$. Finally, with the help of the stochastic Gronwall's inequality, the pathwise uniqueness is proved for the transformed processes and hence the pathwise uniqueness of the solutions $X_t$ follows.
\vskip 0.4cm
The last part of the paper is  the appendix which provides the  proofs for some of the results in Section 4.

\vskip 0.5cm

We close this introduction by mentioning some conventions used throughout this
paper: $|\cdot|$ or $d(\cdot,\cdot)$ denotes the Euclidean norm in $\Rd$. $\c$ denotes the inner product in $\Rd$. Use $B(x,r)$ to denote the ball in $\Rd$ centered at $x$ with radius $r$, and use $n^i(x)$ to denote the $i$-th component of the unit inward  normal $n(x)$ for $1 \le i \le d$.
%For Borel set $A \subset \Rd$, set $A_T:=[0,T]\times A$. For Borel set $A \subset \Rdd$, set $A(t):=\{x:(t,x)\in A\}$.
For $d \times d$ matrix $A$, we use $|A|$ to denote the determinant of $A$ and define $\|A\|_2:=\sup_{x\in B(0,1)}|Ax|$,
$\|A\|:=\sup_{1\le i,j\le d}|a_{ij}|$.
% $A \subset \subset B$ stands for $\bar A \subset B$, for $A,B$ are open sets of $\Rd$ or $\Rdd$.
% We call $A$ is a open set in $[0,T]\times \Rd$ if there exists an open set $B \subset \Rdd$ such that $A=B\bigcap [0,T]\times \Rd$.
Let $D_xf(x)$ stand for the vector $(\p_{x_1}f(x),...,\p_{x_d}f(x))$ if $f:\R^d \to \R$, and $D_xf(x)$ stand for the Jacobian matrix of $f$ if $f:\R^d \to \R^d$. For $x \in \Rd$, a unite vector $\gamma$ in $\Rd$ and $\theta,r>0$, define $C(x, \gamma, \theta ,r):=\{y\in \Rd: 0<|y-x|<r \ \mbox{and} \ (y-x)\cdot \gamma > \cos \theta |y-x|\}$.
% $C^{0,1}_b(\tilde G)$.
%We will use $(a,b)$ stands for $(a\wedge b,a\vee b)$ and $\int_a^bf(t)dt:=\int_{(a,b)}f(t)dt$.
For $-\infty<b<a<\infty$, we stile use the symbol $(a,b]$ to stand for $\{t\in \Rd:b\le t<a\}$
%and use the symbol $(a,b)$ to stand for $\{t\in \Rd:b< t<a\}$  and use the symbol $\int_a^bf(t)dt$ to stand for $\int_b^af(t)dt$
when there is no danger of causing ambiguity. For an open set $O \subset  \Rd $ and measurable function $f: O \rightarrow \Rd$, denote $\| f\|_{L^{2d+2}(O)}:=\| |f| \|_{L^{2d+2}(O)}$.
For a bounded domain $\O \subset (0,T)\times \Rd$, $W_{2d+2}^{1,2}(\O)$ ($W_{2d+2}^{1,2}(\O; \Rd)$) is a Sobolev space of functions $f(t,x)$ (with value in $\Rd$) such that
$$\|f\|_{W_{2d+2}^{1,2}(\O)}:=\|f\|_{L^{2d+2}(\O)}+\| \p_tf(t,x)\|_{L^{2d+2}(\O)}+\sum_{1\le i,j \le d}\|\p_{x_i}\p_{x_j}f(t,x)\|_{L^{2d+2}(\O)}<\infty.$$
Where $\p_tf(t,x)$ stands for the first order weak derivative with respect to ($w.r.t.$) $t$ and $\p_{x_i}\p_{x_j}f(t,x)$ stands for the second order weak derivative $w.r.t.$ $x$.
In the sequel, we will also write $W_{2d+2}^{1,2}(\O)$ for $W_{2d+2}^{1,2}(\O; \Rd)$ with no danger of ambiguity.
$c$ will denote a generic positive constant which may be different from line to line, and $a \les b$ means $a \le cb$ for some unimportant $c>0$.
% Let the letter $c$ with or without subscripts stands for an unimportant positive constant, whose value may change in different places.

% $X_t$ has a continuous transition density function $p(t,x,y)$ which admits a two-sided Aronson's Gaussian type estimate. We know from \cite{BassChen} that for any $\pi \in \K_{d-1}$, there exists an (unique?) CAF $B_t$ associated with it.

\section{Parabolic PDEs associated with the singular drift}

In this section, we consider parabolic PDEs associated with the singular drift on the domain $D$, equipped with the Neumann boundary condition. We will provide some results on the regularity of the solutions, which will be used in subsequent sections.

\vskip 0.4cm

Since $\partial D \in C^3$ satisfies a uniform interior sphere condition and a uniform exterior sphere condition, we can find a positive constant $\d_0$ such that for each point $y \in \partial D$ there exist balls $B$ and $B'$, with the radii being  bounded from below by $\d_0$, satisfying $\overline{B} \cap D^c=\overline{B'} \cap \overline D=\{y\}$. Set $\G_{c}:=\{x\in \Rd: d(x, \p D)<c\}$ for $c>0$.
Following the argument of Lemma 14.16 in \cite{Gilbarg}, we see that for any $x \in \G_{\delta_0}$, there exists a unique $\varphi(x) \in \partial D$ such that $|x-\varphi(x)|=d(x, \p D)$, moreover $\varphi\in C^2(\G_{\delta_0})$. Thus we can extend $n(x)$ to the whole space $\Rd$ such that $n \in C_0^2(\Rd)$ with $|n(x)| \le 1$ on $\Rd$ and $|n(x)|= 1$ on $\G_{\frac{\delta_0}{2}}$ (for example, take $n(x):=n(\varphi(x))\phi(x)$ for some real function $\phi \in C_0^2(\G_{\delta_0})$ with $\phi(x)= 1$ on $\G_{\frac{\delta_0}{2}}$).

\vskip 0.4cm

From Chapter 4 (Section 9) in \cite{Lady}, it is known that for any $T>0$, there exists a unique weak solution $u^T\in W_{2d+2}^{1,2}((0, T) \times D)$ to the following boundary value problem:
\begin{align}\label{1}
 \left\{
\begin{aligned}
& \partial_{t} u^T(t,x)+\frac{1}{2} \Delta_{x} u^T(t,x)+b(t,x) \cdot \nabla_{x} u^T(t,x)=0,&&\forall (t,x) \in (0, T) \times D, \\
&  \frac{\partial u^T}{\partial n}(t, x)=n(x),&&\forall (t,x) \in (0, T) \times \partial D, \\
& u^T(T, x)=x, && \forall x \in D.\\
\end{aligned}
\right.
\end{align}
We now consider smooth approximations of the drift vector field $b(t,x)$. \vskip 0.3cm
Fix a nonnegative smooth function $\psi$ on $\R^{d+1}$ with compact support such that
$$\int_{\R^{d+1}} \psi(t,x)dtdx=1.$$
For any positive integer $n$, let $\psi_n(t,x):=2^{n(d+1)}\psi(2^nt,2^nx)$ and
\begin{eqnarray}\label{2.4} \nonumber
\begin{split}
b_n(t,x):=\int_{\R^{d+1}}b(s,y)\psi_n(t-s,x-y)dsdy.
\end{split}
\end{eqnarray}
Since $b_n(t,x)$ is smooth, according to Theorem 5.18 in \cite{Lieberman} there exists a unique \\
$u_n^T \in C^{1,2}_b([0, T]\times \bar{D})$,  that is the solution to the following boundary value problem:
\begin{align} \nonumber
 \left\{
\begin{aligned}
& \partial_{t} u_n^T(t,x)+\frac{1}{2} \Delta_{x} u_n^T(t,x)+b_n(t,x) \cdot \nabla_{x} u_n^T(t,x)=0,&&\forall (t,x) \in (0, T) \times D, \\
&  \frac{\partial u_n^T}{\partial n}(t, x)=n(x),&&\forall (t,x) \in (0, T) \times \partial D, \\
& u_n^T(T, x)=x, && \forall x \in D.\\
\end{aligned}
 \right.
\end{align}
Moreover, by Theorem 7.20 in \cite{Lieberman}
we have
\begin{eqnarray} \label{3.77}
\lim_{n\to \infty} \| u_n^{T}-  u^{T}\|_{W_{2d+2}^{1,2}((0,T)\times  D)}=0.
\end{eqnarray}

Set $G:=\{x \in \Rd, d(x,D)< \frac{\d_0}{2}\}$ and $G':=\{x \in \Rd, d(x,D)< \d_0\}$.
We have the following result.

%Now extend $u^T$ and $u_n^T$ to the domain $G'$ by setting $u^T(t,x):=2u^T(t,\varphi(x))-u^T(t,2\varphi(x)-x)$ and $u^T_n(t,x):=2u^T_n(t,\varphi(x))-u^T_n(t,2\varphi(x)-x)$ for $x \in G'\backslash D$. By (\ref{3.77}) and Sobolev inequality (see Lemma II.3.3 in \cite{Lady}), it is easy to see that $u^T(t,x)\in C^{0,1}((0,T)\times G')$ and $u^T_n(t,x)\in C^{1,1}((0,T)\times G')$. Moreover we have

\begin{lem} \label{L3.1.1}
There exist constants $M_0>0$ and $0< \a_0 <1$, such that for any $n\ge 1$, $0<T\le 1$, we can extend $u^T$ and $u^T_n$ to $[0,T] \times G'$, such that $u^T \in C^{0,1}([0,T]\times G')$, $u^T_n \in C^{1,1}([0,T]\times G')$, $u^T(T,x)=u^T_n(T,x)=x$ on $G'$ and
\begin{eqnarray}
\lim_{n\to \infty}\sup_{(t,x)\in [0,T]\times G'} (| u_n^{T}(t, x)- u^{T}(t, x)|+\|\nabla_{x}u_n^{T}(t, x)-\nabla_{x}u^{T}(t, x)\|)=0. \label{3.64}
\end{eqnarray}
Moreover, if $0\le s \le t \le T$ and $x \in G'$, then
\begin{eqnarray}
|u^{T}(t, x)-u^{T}(s, x)|+\|\nabla_{x}u^{T}(t, x)-\nabla_{x}u^{T}(s, x)\| \leqslant M_0|t-s|^{\a_0}, \label{3.1}\\
\|\nabla_{x}u_n^{T}(t, x)-\nabla_{x}u_n^{T}(s, x)\| \leqslant M_0|t-s|^{\a_0}. \label{2.38}
\end{eqnarray}
\end{lem}

\begin{proof}
By (\ref{3.77}) and Sobolev inequality (see Lemma II.3.3 in \cite{Lady}), we know that
$$u^T \in C^{0,1}([0,T] \times \bar D),$$
and (\ref{3.64})-(\ref{2.38}) hold if $G'$ is replaced by $\bar D$.

Now we define the extensions of $u^T$ and $u^T_n$ on $[0,T] \times G' \setminus \bar D$ by
\begin{eqnarray} \label{2.42}
\begin{split}
u^T(t,x):=2u^T(t,\varphi(x))-u^T(t,2\varphi(x)-x),\\
u^T_n(t,x):=2u^T_n(t,\varphi(x))-u^T_n(t,2\varphi(x)-x).
\end{split}
\end{eqnarray}
Since $\varphi\in C^2(\G_{\delta_0})$, it is easy to see that  $u^T(T,x)=u^T_n(T,x)=x$ on $G'$ and
$$u^T \in C^{0,1}([0,T]\times G' \setminus \bar D), \ u^T_n \in C^{1,1}([0,T]\times G' \setminus \bar D).$$
Now we show that for any $x_0 \in \p D$, $u^T(t,\cdot)$ and $u^T_n(t,\cdot)$ are differentiable in a neighborhood of $x_0$.

Since $\p D \in C^3$, for $x_0 \in \p D$ there exist a neighborhood $U$ of $x_0$ and a $C^3$-diffeomorphism $\Psi$ that maps $U$ onto $B(\Psi(x_0),r)$ for some $r>0$, such that $\Psi^{-1}(B^-(\Psi(x_0),r)=U \setminus D$, where $\Psi^{-1}$ is the inverses  of $\Psi$ and $B^-(\Psi(x_0),r):=\{x=(x_1,x_2,\cdots,x_d) \in B(\Psi(x_0),r) : x_d\le 0\}$.
Set $v(t,x):=u^T(t,\Psi^{-1}(x))$ and $v_n(t,x):=u^T_n(t,\Psi^{-1}(x))$. Then for $x \in B^-(\Psi(x_0),r)$,
\begin{eqnarray}
v(t,x)=2v(t,\Psi(\varphi(\Psi^{-1}(x)) ))-v(t,\Psi(2\varphi(\Psi^{-1}(x))-\Psi^{-1}(x))), \nonumber\\
v_n(t,x)=2v_n(t,\Psi(\varphi(\Psi^{-1}(x)) ))-v_n(t,\Psi(2\varphi(\Psi^{-1}(x))-\Psi^{-1}(x))). \nonumber
\end{eqnarray}
%Hence
%\begin{eqnarray}
%\begin{split}
%&\ \ \ \ \lim_{x\to x_0}D_x v(t,x)\\
%&=2\lim_{x\to x_0}\{D_xv(t,\Psi(\varphi(\Psi^{-1}(x))))D_x\Psi(\varphi(\Psi^{-1}(x)))D_x\varphi(\Psi^{-1}(x))D_x\Psi^{-1}(x)\\
%&\ \ \ \ -D_xv(t,\Psi(2\varphi(\Psi^{-1}(x))-\Psi^{-1}(x))) D_x \Psi(2\varphi(\Psi^{-1}(x))-\Psi^{-1}(x))(2D_x\varphi(\Psi^{-1}(x))D_x\Psi^{-1}(x)-D_x\Psi^{-1}(x))\}\\
%\end{split}
%\end{eqnarray}
One can verify that $v \in C^{0,1}([0,T]\times B(\Psi(x_0),r))$ and $v_n \in C^{1,1}([0,T] \times B(\Psi(x_0),r))$. Hence, we have
$$u^T(t,x)=v(t,\Psi(x)) \in C^{0,1}([0,T]\times U),$$
and
$$u_n^T(t,x)=v_n(t,\Psi(x)) \in C^{1,1}([0,T]\times U).$$

Note that $2\varphi(x)-x \in D$ for $x \in G' \setminus \bar D$, from the definition of $u^T$ and $u^T_n$ in (\ref{2.42}), clearly (\ref{3.64})-(\ref{2.38}) hold.
\end{proof}

\vskip 0.4cm

% By Claim 10, we know that $u(t,D)$ is a open domain and $u(t,\partial D)$ is the partial of $u(t,D)$.
Set $\tilde u^T(t,x):=(t,u^T(t,x))$ and $\tilde u_n^T(t,x):=(t,u_n^T(t,x))$. Then we have the following proposition:

\begin{prp} \label{C3.2}
% i) For any $n\ge 1$, $0 \le t \le T\le 1$ and $x \in G'$, $u^T(t, \cdot) \in C^1(G')$ and $\nabla_{x}u^{T}(\cdot,x) \in C^{\a}$. \\
There exists  a constant $T_0 \in (0, 1]$ such that for any $T\in (0,  T_0]$, $\tilde u^T $, $\tilde u_n^T $ are open mappings on $(0,T) \times G$. Moreover, there exist positive constants $M_1,M_2,M_3$, such that for any $n\ge 1$, $0 \le t\le T \le T_0$ and $x,y \in G$,
\begin{eqnarray}
\frac{1}{2} \le |D_x u^T(t,x)|\le 2, \ \frac{1}{2} \le |D_xu_n^T(t,x)|\le 2, \label{3.63}\\
%||D_xu_n(t,x)|| \wedge ||D_xu(t,x)||\ge \frac{1}{2}, \label{3.75} \\
M_1|x-y|\le |u^T(t,x)-u^T(t,y)|\le M_2|x-y|, \label{3.2} \\
M_1|x-y|\le |u_n^T(t,x)-u_n^T(t,y)|\le M_2|x-y|. \label{3.73}
\end{eqnarray}
Furthermore, if $|x-y|< \frac{\d_0}{2}$, then
\begin{eqnarray}\label{2.24}
|u^T(t,x)-u^T(t,y)| \wedge |u_n^T(t,x)-u_n^T(t,y)|\ge (1-M_3T^{\a_0})|x-y|,
\end{eqnarray}
where $\a_0$ is the constant defined in Lemma \ref{L3.1.1}.

\end{prp}

\begin{proof}
\textcolor[rgb]{1.00,0.00,0.00}{% We only show (\ref{2.24}), the others can be proved in the same way.
}
%The proof of this proposition is essentially the same as that of Lemma 3.4 in \cite{Gyongy}, we will give the proof for the sake of completeness.
We only give the proof of the properties of $u^T$ because the corresponding proof for $u_n^T$ is similar.

\vskip 0.3cm

Since $D_x u^T(T,x)$ is the identity matrix, by Lemma \ref{L3.1.1} one can see that (\ref{3.63})
 holds if $T_0$ is sufficiently small. Without loss of generality, we assume $T_0 < (\frac{1}{dM_0})^{\frac{1}{\a_0}} \wedge(\frac{\d_0}{8M_0})^{\frac{1}{\a_0}} $.
%\begin{eqnarray*}
%| u^T(t,x) |_2 \ge \|D_x u^T(T,x)\|_2-\|D_x u^T(T,x)-D_x u^T(t,x)\|_2\ge \frac{1}{2},
%\end{eqnarray*}
%and
%\begin{eqnarray*}
%|D_x u^T(t,x)|_2 \le \|D_x u^T(T,x)\|_2+\|D_x u^T(T,x)-D_x u^T(t,x)\|_2\le 2.
%\end{eqnarray*}

\vskip 0.3cm

To show that $\tilde u^T $ is an open mapping on $(0,T) \times G$, it is sufficient  to show that for any open set $A_1 \subset G$ and $(t_0,x_0) \in (0,T) \times A_1$, there exist constants $\eta,\d>0$ such that if $t \in (t_0-\eta, t_0+\eta)$, then
\begin{eqnarray}  \label{2.39}
 B(u^T(t_0,x_0), \d) \subset  u^T(t , A_1).
\end{eqnarray}

By (\ref{3.63}) and the implicit function theorem, $u^T(t_0,\cdot)$ is an open mapping on $ G$. Hence there exist a constant $\d>0$ and an open set $A_2 \subset A_1$ such that $u^T(t_0,A_2)=B(u^T(t_0,x_0), \d)$ and $B(u^T(t_0,x_0), 2\d) \subset  u^T(t_0, A_1)$. By (\ref{3.1}), there exists a constant $\eta>0$ such that for any $t \in (t_0-\eta, t_0+\eta)$,
\begin{eqnarray} \label{2.40}
|u^T(t,x)-u^T(t_0,x)| < \frac{\d}{4}, \ \forall x \in G'.
\end{eqnarray}

Suppose $B(u^T(t_0,x_0), \d) \nsubseteq u^T(t , A_1)$ for some $t\in (t_0-\eta, t_0+\eta)$. Then there exists a point $x \in A_2 $ such that $u^T(t_0,x) \in u^T(t,A_1)^c$. Since
 $u^T(t,x) \in u^T(t, A_1)$ and since $u^T(t,A_1)$ is an open set,
% there exists a $\lambda>0$ such that $\lambda u^T(t_0,x)+ (1-\lambda) u^T(t,x) \in u^T(t, \p A_1)$.
we have
\begin{eqnarray} \label{2.41}
\begin{split}
d(u^T(t_0,x),u^T(t,\p A_1)) &\le |u^T(t_0,x)-u^T(t,x)| <\frac{\d}{4}.
\end{split}
\end{eqnarray}
On the other hand, by (\ref{2.40}),
\begin{eqnarray} \nonumber
\begin{split}
d(u^T(t_0,x),u^T(t,\p A_1)) &\ge d(u^T(t_0,x),u^T(t_0,\p A_1))-d(u^T(t_0,\p A_1),u^T(t,\p A_1)) \\
&\ge d(B(u^T(t_0,x_0), \d),u^T(t_0,\p A_1))-d(u^T(t_0,\p A_1),u^T(t,\p A_1)) \\
& \ge \d-\frac{\d}{4}=\frac{3\d}{4},
\end{split}
\end{eqnarray}
which contradicts (\ref{2.41}). Hence we have (\ref{2.39}).
%Similarly, we see that $\tilde u_n^T $ is an open mapping on $(0,T) \times A_1$.

\vskip 0.3cm

Now we show (\ref{2.24}).
For $0 \le t\le T \le T_0$ and $x,y \in G$ with $|x-y|< \frac{\d_0}{2}$, we have $\lambda x+(1-\lambda)y \in G'$ for any $\lambda \in (0,1)$. Hence by (\ref{3.1}) and the fact that $\| \c \|_2 \le d\| \c \|$,
\begin{eqnarray*}
\begin{split}
&\ \ \ \   |u^T(t,x) -u^T(t,y)| \\
&\ge   |u^T(T,x) \! - \! u^T(T,y)|-|\int_0^1 (D_x u^T(T,\lambda x+(1-\lambda)y) \! - \! D_x u^T(t,\lambda x+(1-\lambda)y)) \c (x-y)d \lambda | \\
&\ge   |u^T(T,x) \! - \! u^T(T,y)|-\int_0^1 d \| D_x u^T(T,\lambda x+(1-\lambda)y) \! - \! D_x u^T(t,\lambda x+(1-\lambda)y)) \| |x-y| d \lambda \\
&\ge  |x-y|-dM_0|T-t|^{\a_0} |x-y| \ge (1-dM_0 T^{\a_0}) |x-y|.
\end{split}
\end{eqnarray*}

\vskip 0.3cm

Finally we show (\ref{3.2}). Let  $0 \le t\le T \le T_0$ and $x,y \in G$. When $|x-y|\ge \frac{\d_0}{2}$, notting that $T_0 <(\frac{\d_0}{8M_0})^{\frac{1}{\a_0}}$, by (\ref{3.1}) we have
\begin{eqnarray} \label{2.43}
\begin{split}
|u^T(t,x)-u^T(t,y)| \le 2 \sup_{(t,x)\in [0,T]\times G} |u^T(t,x) | \le \frac{4}{\d_0}\sup_{(t,x)\in [0,T]\times G} |u^T(T,x) | |x-y|,
\end{split}
\end{eqnarray}
and
\begin{eqnarray} \label{2.67}
\begin{split}
|u^T(t,x)-u^T(t,y)| &=|u^T(t,x)-u^T(T,x)-u^T(t,y) +u^T(T,y)+x-y| \\
& \ge |x-y|-|u^T(t,x)-u^T(T,x)|-|u^T(t,y) -u^T(T,y)|\\
& \ge |x-y|-2M_0T^{\a_0}\\
&\ge \frac{|x-y|}{2}-(\frac{\d_0}{4}- 2M_0T^{\a_0})\\
&> \frac{|x-y|}{2}.
\end{split}
\end{eqnarray}
When $|x-y|\le \frac{\d_0}{2}$, then $\lambda x+(1-\lambda)y \in G'$ for any $\lambda \in (0,1)$. Using the Lagrange mean value theorem and the boundness of $\|D_xu^T(t,x)\|$, we have
\begin{eqnarray} \label{2.68}
\begin{split}
|u^T(t,x)-u^T(t,y)| \le M_2 |x-y|.
\end{split}
\end{eqnarray}
Hence combining (\ref{2.24}) with (\ref{2.43})-(\ref{2.68}), we get (\ref{3.2}).
\end{proof}

\vskip 0.4cm

\section{Domain transformation and regularity of the reflecting directions }
In this section we study the time dependent domains which are the images of  domain $D$ under the solution mappings $u^T(t, \cdot)$, $t\in [0, T]$. Regularities of the time dependent vector field  of the reflecting directions will be established.

\vskip 0.3cm
We start by  showing the exterior and interior cone conditions for $u^T(t,D)$ for sufficiently small $T$.
Set
$$\d_1:= M_1  d( \p D ,   \p G)=\frac{M_1\d_0}{2},$$
$$\tilde G_1^T:=\{(t,x):0 \le t \le T , \ d(x,u^T(t,D))<\frac{\d_1}{2}\},$$
and
$$\tilde G_2^T:=\{(t,x):0 \le t \le T , \ d(x,u^T(t,D))<\frac{3\d_1}{4}\}.$$
Recall the constant $T_0$ defined in the statement of Proposition 2.1. We first have the following Lemma.
\begin{lem} \label{L3.1}
There exists an integer $N_0>0$ such that for $n\ge N_0$, $0<T\le T_0$,
$$\tilde u^T_n([0,{T}]\times \bar D) \bigcup \tilde u^T([0,{T}]\times \bar D) \subset \tilde G_1^T \subset \tilde G_2^T  \subset \tilde u^T_n([0,{T}]\times G) \bigcap \tilde u^T([0,{T}]\times G). $$
% such that for any $(t,x)\in [0,T_1]\times \bar D$, there exists a $r>0$ such that $\tilde B((t,u_n(t,x)),r) \subset \tilde G $.
\end{lem}

\begin{proof}
By (\ref{3.64}), there exists an integer $N_0>0$ such that for $n\ge N_0$, we have
\begin{eqnarray} \nonumber
\begin{split}
\sup_{(t,x)\in [0,T] \times G}|u^T_n(t,x)-u^T(t,x)|<\frac{\d_1}{8},
\end{split}
\end{eqnarray}
which implies $\tilde u^T_n([0,{T}]\times \bar D) \bigcup \tilde u^T([0,{T}]\times \bar D) \subset \tilde G_1^T \subset \tilde G_2^T$.
Noting that $u^T(t,G)$ is an open set by Proposition \ref{C3.2}  and the fact that $\inf_{0 \le  t \le T}d(u^T(t, D), u^T(t, \p G))\ge M_1  d( \p D ,   \p G)= \d_1$, we have $ \tilde G_2^T  \subset \tilde u^T([0,{T}]\times G)$.

Now we show that $\tilde G_2^T  \subset \tilde u^T_n([0,{T}]\times G)$ for $n \ge N_0$. Suppose there exists a point $(t,x)$ belonging to $\tilde G_2^T \setminus \tilde u^T_n([0,{T}]\times  G)$. Then $d(x,u^T_n(t,D))\le d(x,u^T(t,D))+\frac{\d_1}{8}<\frac{7\d_1}{8}$. Therefore there exists a $y_1 \in u^T_n(t,D)$ such that $|x-y_1|< \frac{7\d_1}{8}$. On the other hand, since $x\in u^T_n(t,G)^c$, $y_1 \in u^T_n(t,D)$ and  since $u_n^T(t,G)$ is an open set, there exists a $\lambda \in (0,1)$ such that  $y_2:=\lambda x+(1-\lambda)y_1 \in u^T_n(t, \p G)$. Hence $|y_1-y_2|<|x-y_1|< \frac{7\d_1}{8}$,
which contradicts the fact that $d(u^T_n(t,D),u^T_n(t, \p G))\ge M_1d( D , \p G)=\d_1$ by (\ref{3.73}).
% Since $u_n(t, G)$ is an open set in $\Rd$ by Proposition \ref{C3.2}, there exists a $r>0$ such that $B(x,r)\cap u_n(t, G)=\emptyset$ and $\bar B(x,r)\cap u_n(t, \bar G) \neq \emptyset$.
\end{proof}

\vskip 0.4cm

The next result shows that  $u^T(t,D)$ fulfils  the exterior and interior cone conditions for  sufficiently small $T$.
%Let $\d_0$ be the constant defined in Lemma \ref{L3.1.1}.

\begin{prp} \label{P2.2}
There exist constants $T_1 \in (0, T_0)$, $\theta_0 \in (0,\frac{\pi}{2})$ and $\d_2 \in (0, \frac{ \d_1}{2} ]$ such that for any $t \in [0,{T_1}]$ and $x\in \p D$,
\begin{eqnarray}
C(u^{T_1}(t,x),-n(x),\t_0,\d_2) \subset u^{T_1}(t,\bar D)^c, \label{3.59} \\
C(u^{T_1}(t,x),n(x),\t_0,\d_2) \subset u^{T_1}(t, D). \label{3.60}
\end{eqnarray}

\end{prp}

\begin{proof}
We only prove (\ref{3.59}). (\ref{3.60}) can be proved similarly.

Since $\p D$ is smooth, there exist constants $\t \in (0,\frac{\pi}{2})$ and $ r >  0$ such that for $x  \in   \p D$,   $C(x, -n(x), \t,  r) \!  \subset   \!   \bar D^c$.
Choose $T_1  \!  \in  \!  (0, T_0)$ to be sufficiently small so  that $\frac{\cos \t  \!  + \!  dM_0 T_1^{\a_0}}{1 \!  - \!  M_3 T_1^{\a_0}}  \!  \in \!   (0,1)$. \\
Set $\t_0:=\arccos \frac{\cos \t   + dM_0 T_1^{\a_0}}{1   -    M_3 T_1^{\a_0}}$ and $\d_2   :=   \frac{ \d_1}{2}   \wedge ( M_1r) $. Now we show that for $t    \in     [0,{T_1}]$ and $x    \in   \p D$, (\ref{3.59}) holds.

Take $y    \in    C(u^{T_1}(t,x),-n(x),\t_0,\d_2)$. Since $(t,y) \in \tilde G_1^T$, by (\ref{3.2}) and Lemma \ref{L3.1} there exists a $y'\in G$ such that $y=u^{T_1}(t,y')$ and
$$|x-y'|\le \frac{1}{M_1}|u^{T_1}(t,x)-u^{T_1}(t,y')|<\frac{\d_0}{4}\wedge r,$$
which implies $d(\l x+(1-\l)y',x)<\frac{\d_0}{4}$ for $\l \in (0,1)$. From the definition of $G$, we see that $\l x+(1-\l)y' \in G$. Together with (\ref{3.1}) and (\ref{2.24}) we have
\begin{eqnarray} \nonumber
\begin{split}
&\ \ \ \ (x-y') \c n(x)\\
&=(u^{T_1}(t,x)-u^{T_1}(t,y'))\c n(x)+(x-u^{T_1}(t,x)-(y'-u^{T_1}(t,y')))\c n(x)\\
&> \cos \t_0|u^{T_1}(t,x)-u^{T_1}(t,y')|-|u^{T_1}(T_1,x)-u^{T_1}(t,x)-(u^{T_1}(T_1,y')-u^{T_1}(t,y'))|\\
&\ge \cos \t_0|u^{T_1}(t,x)-u^{T_1}(t,y')|\\
&\quad \  -\int_0^1 \|\nabla_x u^{T_1}(T_1,\l x+(1-\l)y')-\nabla_x u^{T_1}(t,\l x+(1-\l)y') \|_2 |x-y'|d\l\\
&\ge (1-M_3 T_1^{\a_0}) \cos \t_0 |x-y'|-dM_0 T_1^{\a_0}|x-y'|= \cos \t |x-y'|, \\
\end{split}
\end{eqnarray}
which implies $y' \in C(x,-n(x),\t,r) \subset \bar D^c$. Hence $y=u^{T_1}(t,y')\in u^{T_1}(t,\bar D)^c$, which implies (\ref{3.59}).
\end{proof}

\vskip 0.4cm

Here and below, we fix $T_1>0$ as defined in Proposition \ref{P2.2}. Denote $u(t,x):=u^{T_1}(t,x)$, $u_n(t,x):=u^{T_1}_n(t,x)$, $\tilde u (t,x):=\tilde u^{T_1}(t,x)$, $\tilde u_n(t,x):=\tilde u_n^{T_1}(t,x)$, $\tilde G_1:=\tilde G_1^{T_1}$, $\tilde G_2:=\tilde G_2^{T_1}$ and $\tilde D :=\tilde u((0,{T_1})\times D)$.
\vskip 0.3cm

By Proposition \ref{C3.2} and Lemma \ref{L3.1}, the inverses  of $u$ and $u_n$ exist, denoted by $u^{-1}$ and $u_n^{-1}$. Moreover, it is easy to see that $u^{-1}, u_n^{-1}$ are continuous in $\tilde G_2$ w.r.t. $(t,x)$ for $n\ge N_0$. Take a smooth function $\phi(t,x) \in C_b^{\infty}([0,T_1]\times \Rd)$ such that $\phi(t,x)=1$ on $\tilde G_1$ and $\phi(t,x)=0$ on $[0,T_1]\times \Rd \diagdown \tilde G_2$. For $n\ge N_0$,
set
$$\gamma_n(t,x):=(\gamma_n^1(t,x),\gamma_n^2(t,x), \cdots,\gamma_n^d(t,x)):=n(u_n^{-1}(t,x))\phi(t,x),$$
$$\gamma(t,x):=(\gamma^1(t,x),\gamma^2(t,x),\cdots,\gamma^d(t,x)):=n(u^{-1}(t,x))\phi(t,x).$$
Then $\g_n$ and $\g$ are well defined in $[0,T_1]\times \Rd$.
$\g$ will be the directions of reflection of the transformed reflecting SDEs. To obtain the regularity of $\g$, we need to study the convergence of $u_n^{-1}$, which is the content of the next lemma.

\vskip 0.4cm

Set
\begin{eqnarray}
\begin{split}
D(t,c):=\{x:d(x,u(t,D)^c)> c \}, \label{2.19-1}
\end{split}
\end{eqnarray}
$\tilde {D}_{c}:=\{(t,x):t\in (0,T_1),\ x \in D(t, c)\}$ and $\tilde {D}_{c}':= \{(t,x):t\in (0,T_1),\ d(x,D^c)> c \}$ for $c>0$. Then we have the following result.

\begin{lem} \label{L3.2.2}
For $n\ge N_0$, we have $u^{-1}_n \in C_b^{1,1}(\tilde G_2)$, $u^{-1} \in  C^{0,1}_b(\tilde G_2)$ and
\begin{eqnarray}  \label{3.87}
\begin{split}
\lim_{n \to \infty}\|u^{-1}_n-u^{-1}\|_{C^{0,1}_b(\tilde G_2)}=0.
\end{split}
\end{eqnarray}
Moreover, for any constants $\e,p>0$ and functions $f \in L^p((0,T_1) \times D )$ and $g_k,g \in L^p( \tilde D)$ with $\lim_{k \to \infty} \|g_k-g\|_{L^p( \tilde D)}=0$,  $\tilde {D}_{\e}$ is an open set in $\Rdd$ and there exists an integer $N_0(\e) \ge N_0$ such that for any $n\ge N_0(\e)$, we have $\tilde u_n^{-1}(\tilde {D}_{\e}) \subset \tilde {D}_{\frac{\e}{2M_2}}'$, $u_n^{-1} \in C_b^{1,2}(\tilde {D}_{\e})$, $|D_x u^{-1}_n(t,x)| \ge \frac{1}{2}$ for $(t,x)\in \tilde {D}_{\e}$ and
\begin{eqnarray}
\lim_{n \to \infty}\|f(t,u_n^{-1}(t,x) )-f(t,u^{-1}(t,x)) \|_{L^p(\tilde {D}_{\e})}=0,  \label{3.88} \\
\lim_{n \to \infty}\|g(t,u_n(t,x) )-g(t,u(t,x)) \|_{L^p(\tilde {D}_{\e}')}=0, \label{3.89} \\
\lim_{k \to \infty}\|g_k(t,u_n(t,x) )-g(t,u_n(t,x)) \|_{L^p(\tilde {D}_{\e}')}=0,  \ \forall n \ge N_0(\e) ,  \label{3.91}
\end{eqnarray}
where $\tilde u_n^{-1}$ is the inverse of $\tilde u_n$ and $M_2$ was the constant defined in Proposition \ref{C3.2}.
\end{lem}

\begin{proof}
First we show (\ref{3.87}). For any $n\ge N_0$ and $(t,x)\in \tilde G_2$, we have $u_n^{-1}(t,x), u^{-1}(t,x) \in G$ by Lemma \ref{L3.1}. Hence by (\ref{3.64}) and (\ref{3.2}), for $n\ge N_0$ we have
\begin{eqnarray}\label{3.62}
\begin{split}
\sup_{(t,x)\in \tilde G_2}|u^{-1}_n(t,x)-u^{-1}(t,x)|&\le \sup_{(t,x)\in \tilde G_2} \frac{1}{M_1}|u(t,u_n^{-1}(t,x))-u(t,u^{-1}(t,x))|\\
&=\frac{1}{M_1}\sup_{(t,x)\in \tilde G_2}|u(t,u_n^{-1}(t,x))-x|\\
&=\frac{1}{M_1}\sup_{(t,x)\in \tilde G_2}|u(t,u_n^{-1}(t,x))-u_n(t,u_n^{-1}(t,x))|\\
&\le \frac{1}{M_1}\sup_{(t,y)\in [0,T_1]\times G}|u(t,y)-u_n(t,y)| \to 0,
\end{split}
\end{eqnarray}
as $n \to \infty$.

On the other hand, by (\ref{3.63}) and the implicit function theorem, it is easy to see that $u^{-1} \in C_b^{0,1}(\tilde G_2)$, $u^{-1}_n \in  C_b^{1,1}(\tilde G_2) \cap C_b^{1,2}(\tilde u_n((0,T_1)\times D))$ and for any $(t,x) \in \tilde G_2$,
\begin{eqnarray} \label{3.66}
\begin{split}
D_x u^{-1}(t,x)=[(D_x u) (t, u^{-1}(t,x))]^{-1}, \\
D_x u_n^{-1}(t,x)=[(D_x u_n)(t, u_n^{-1}(t,x))]^{-1}, \\
\p_t u_n^{-1}(t,x)=-D_x u_n^{-1}(t,x) \c (\p_t u_n)(t,u_n^{-1}(t,x)).\\
\end{split}
\end{eqnarray}
Combining this with (\ref{3.64}), (\ref{3.63}) and (\ref{3.62}), we obtain
$$\lim_{n \to \infty}\sup_{(t,x)\in \tilde G_2}\|D_xu^{-1}(t,x)-D_xu^{-1}_n(t,x)\|=0.$$
Hence we proved (\ref{3.87}).

\vskip 0.3cm

Given $\e>0$, now we show that $\tilde {D}_{\e}$ is an open set.

By Proposition \ref{C3.2}, we know that $\tilde D$ is open in $\Rdd$. Take $(t_0,x_0) \in \tilde {D}_{\e} \subset \tilde D$, then there exist constants $\eta,\d>0$ such that $(t_0-\eta,t_0+\eta)  \! \times  \! B(x_0, 2\d) \! \subset  \! \tilde D$ and $d(   B(x_0, 2\d),  \! u(t_0,\p D)    ) \!  > \! \e$. By (\ref{3.1}), there exists a positive constant $\eta'<\eta$ such that for any $(t, y) \in (t_0-\eta',t_0+\eta') \times G'$,  $|u(t,y)-u(t_0,y)| < \frac{\d}{2}$. Hence for $t \in (t_0-\eta',t_0+\eta')$,
\begin{eqnarray}   \nonumber
\begin{split}
d(B(x_0, \d), u(t,\p D) ) \ge d(B(x_0, \d), u(t_0, \p D) )-d(u(t_0, \p D), u(t,\p D) ) > \d+\e -\frac{\d}{2} >\e.
\end{split}
\end{eqnarray}
Hence we have $(t_0-\eta',t_0+\eta')\times B(x_0, \d) \subset \tilde {D}_{\e}$, which proves that $\tilde {D}_{\e}$ is open.

\vskip 0.3cm

Noting that $\tilde u(\tilde {D}_{\e}') \subset \tilde {D}_{M_1\e} $ and $\tilde u^{-1}(\tilde {D}_{\e}) \subset \tilde {D}_{\frac{\e}{M_2}}'$ by (\ref{3.2}), together with (\ref{3.64}) and (\ref{3.62}), there exists an integer $N_0(\e) \ge N_0$ such that for any $n\ge N_0(\e)$, we have
\begin{eqnarray}  \label{3.92}
\begin{split}
\tilde u_n(\tilde {D}_{\e}') \subset \tilde {D}_{\frac{M_1\e}{2}} \subset \tilde {D}.
\end{split}
\end{eqnarray}
and
\begin{eqnarray}  \label{3.90}
\begin{split}
 \tilde u_n^{-1}(\tilde {D}_{\e}) \subset \tilde {D}_{\frac{\e}{2M_2}}' \subset (0,T_1)\times D.
\end{split}
\end{eqnarray}
Hence by (\ref{3.63}), (\ref{3.66}) and (\ref{3.90}) we have $u_n^{-1} \in C_b^{1,2}(\tilde {D}_{\e})$ and for any $(t,x)\in \tilde {D}_{\e}$,
\begin{eqnarray}  \label{3.93}
\begin{split}
|D_x u^{-1}_n(t,x)|=|[(D_x u_n)(t, u^{-1}_n(t,x))]^{-1}| \ge \frac{1}{2},  \\
|D_x u^{-1} (t,x)|=|[(D_x u)(t, u^{-1} (t,x))]^{-1}| \ge \frac{1}{2}.
\end{split}
\end{eqnarray}

\vskip 0.3cm

By (\ref{3.63}) and (\ref{3.92}), it is easy to see that (\ref{3.91}) holds.

Now we show (\ref{3.88}).
For any $\e_1>0$, there exists a  function $\tilde f \in C_b(\Rd)$  such that $\|f-\tilde f \|_{L^p(\tilde {D}_{\frac{\e}{2M_2}}')} < \e_1$.
 Combining this with (\ref{3.62}), (\ref{3.90}), (\ref{3.93}) and a change of variable, we see that
\begin{eqnarray}  \nonumber
\begin{split}
&\ \ \ \ \lim_{n \to \infty}\|f(t,u_n^{-1}(t,x) )-f(t,u^{-1}(t,x)) \|_{L^p(\tilde {D}_{\e})} \\
&\le \lim_{n \to \infty}\|f(t,u_n^{-1}(t,x) )-\tilde f(t,u_n^{-1}(t,x)) \|_{L^p(\tilde {D}_{\e})} +\lim_{n \to \infty}\|f(t,u^{-1}(t,x) )-\tilde f(t,u^{-1}(t,x)) \|_{L^p(\tilde {D}_{\e})}  \\
&\ \ \ \ + \lim_{n \to \infty}\|\tilde f(t,u^{-1}_n(t,x) )-\tilde f(t,u^{-1}(t,x)) \|_{L^p(\tilde {D}_{\e})}\\
&\le \lim_{n \to \infty} 2 \| |f(t,u^{-1}_n(t,x) )-\tilde f(t,u^{-1}_n(t,x))| |D_x u^{-1}_n(t,x) | \|_{L^p(\tilde {D}_{\e})}  \\
&\ \ \ \ +\lim_{n \to \infty} 2 \| |f(t,u^{-1}(t,x) )-\tilde f(t,u^{-1}(t,x))| |D_x u^{-1}(t,x) | \|_{L^p(\tilde {D}_{\e})} \\
&\le \lim_{n \to \infty} 4 \| |f(t,x)-\tilde f(t,x)| \|_{L^p(\tilde {D}_{\frac{\e}{2M_2}}')} <4\e_1.\\
\end{split}
\end{eqnarray}
Since $\e_1$ is arbitrary, (\ref{3.88}) follows. By a similar argument, we can show (\ref{3.89}).
\end{proof}

% \begin{lem}
% $\g^{-1}$ and $\g_n^{-1}$ are continuous in $\tilde G$ $w.r.t.$ $(t,x)$.
% \end{lem}
\vskip 0.4cm

The following result provides the regularities of the reflecting directions $\g_n$ and $\g$.

\begin{prp} \label{P3.3}
For $n\ge N_0$, we have
$$\g_n\in C_b^{1,1}([0,{T_1}] \times \Rd) , \ \g \in W_{2d+2}^{1,2}(\tilde D) \bigcap C^{0,1}_b([0,{T_1}] \times \Rd),$$
and
\begin{eqnarray} \label{2-17-2}
\begin{split}
\lim_{n \to \infty}\|\g_n-\g\|_{C^{0,1}_b([0,{T_1}] \times \Rd)}=0.
\end{split}
\end{eqnarray}
Moreover, for any $\e >0$ and $n\ge N_0(\e)$, we have $\g_n\in C_b^{1,2}(\tilde {D}_{\e})$ and
% in $[0,{T_1}]\times \Rd$ with $\tilde {D} \subset [0,T_1]\times D$, if $\inf_{0 \le t \le {T_1}}d(\tilde {D}(t),u(t,\partial D))> \e$, then
%$u_n^{-1}$ converges to $u^{-1}$ in $W_{2d+2}^{1,2}(\tilde {D})$ as $n$ tends to infinity.
\begin{eqnarray}
\lim_{n,m \to \infty}\|\g_n- \g_m\|_{W_{2d+2}^{1,2}(\tilde {D}_{\e})}=0,     \label{3.68}    \\
\sup_{\e>0}\sup_{n\ge N_0(\e)} \|\g_n\|_ {W_{2d+2}^{1,2}(\tilde {D}_{\e})}<\infty, \label{3.79}
\end{eqnarray}
where $N_0(\e)$ was defined in Lemma \ref{L3.2.2}.
\end{prp}

\begin{proof}
It is evident that $\g \in W_{2d+2}^{1,2}(\tilde D)$ if (\ref{2-17-2})-(\ref{3.79}) hold. Hence by Lemma \ref{L3.2.2} we only need to prove (\ref{3.68}) and (\ref{3.79}).

\vskip 0.3cm

First we show (\ref{3.68}). By (\ref{3.77}), Lemma \ref{L3.2.2} and a change of variable, we see that for any $1\le i,j \le d$,
\begin{eqnarray}  \label{3.71}  \nonumber
\begin{split}
& \ \ \ \ \lim_{n \to \infty}\|(\p_{x_j} \p_{x_i} u_n)(t,u_n^{-1}(t,x))-(\p_{x_j} \p_{x_i} u)(t,u^{-1}(t,x)) \|_{L^{2d+2}(\tilde {D}_{\e})} \\
&\le \lim_{n \to \infty}\| (\p_{x_j} \p_{x_i} u_n)(t,u_n^{-1}(t,x))-(\p_{x_j} \p_{x_i} u)(t,u_n^{-1}(t,x))\|_{L^{2d+2}(\tilde {D}_{\e})}\\
& \ \ \ \ +\lim_{n \to \infty}\| (\p_{x_j} \p_{x_i} u)(t,u_n^{-1}(t,x))-(\p_{x_j} \p_{x_i} u)(t,u^{-1}(t,x))\|_{L^{2d+2}(\tilde {D}_{\e})}\\
&= \lim_{n \to \infty}\| (\p_{x_j} \p_{x_i} u_n)(t,u_n^{-1}(t,x))-(\p_{x_j} \p_{x_i} u)(t,u_n^{-1}(t,x))\|_{L^{2d+2}(\tilde {D}_{\e})}\\
&\le 2\lim_{n \to \infty}\||(\p_{x_j} \p_{x_i} u_n)(t,u_n^{-1}(t,x))-(\p_{x_j} \p_{x_i} u)(t,u_n^{-1}(t,x))| |D_xu_n^{-1}(t,x)|\|_{L^{2d+2}(\tilde {D}_{\e})}\\
&\le 2 \lim_{n \to \infty}\| |\p_{x_j} \p_{x_i} u_n(t,x)-(\p_{x_j} \p_{x_i} u)(t,x)| \|_{L^{2d+2}( \tilde {D}_{\frac{\e}{2 M_2}}')}=0.
\end{split}
\end{eqnarray}
Combining this with (\ref{3.63}), (\ref{3.87}) and (\ref{3.66}), we obtain
\begin{eqnarray} \label{3.67}
\begin{split}
\lim_{n,m \to \infty}\|\p_{x_j} \p_{x_i}\p\g_n- \p_{x_j} \p_{x_i}\g_m\|_{L^{2d+2}(\tilde {D}_{\e})}=0.
\end{split}
\end{eqnarray}
By (\ref{3.66}), similar to the proof of (\ref{3.67}), we also have
\begin{eqnarray} \label{3.69}  \nonumber
\begin{split}
\lim_{n,m \to \infty}\|\p_t \g_n- \p_t \g_m\|_{L^{2d+2}(\tilde {D}_{\e})}=0.
\end{split}
\end{eqnarray}
Hence (\ref{3.68}) follows.

\vskip 0.3cm

Now we show (\ref{3.79}). Note that  by Lemma \ref{L3.2.2}  and a change of variable,
\begin{eqnarray} \nonumber
\begin{split}
& \ \ \ \ \sup_{ \e>0,n\ge N_0(\e)}(\|\p_{x_j} \p_{x_i} u_n(t,u_n^{-1}(t,x))\|_{L^{2d+2}(\tilde {D}_{\e})}+\|\p_t u_n (t,u_n^{-1}(t,x)) \|_{L^{2d+2}(\tilde {D}_{\e})})\\
&\le 2 \sup_{\e>0,n\ge N_0(\e)}(\| \p_{x_j} \p_{x_i} u_n(t,x) \|_{L^{2d+2}(  \tilde {D}_{\frac{\e}{2 M_2}}' )}+\|\p_t u_n (t,x)  \|_{L^{2d+2}( \tilde {D}_{\frac{\e}{2 M_2}}'  )})\\
&\le 2 \sup_{n \ge 1}\|u_n(t,x) \|_{W^{1,2}_{2d+2}( (0,T_1) \times D )}<\infty,
\end{split}
\end{eqnarray}
combining this with (\ref{3.63}), (\ref{3.66}) and the boundness of $\| \nabla _x u_n^{-1}(t,x)\|$, we obtain (\ref{3.79}).
\end{proof}

\begin{remark}
{As $u_n$ is not in $C^{1,2}((0,T_1) \times G)$, $\g_n$ does not belong to $C^{1,2}(\tilde D)$.}
\end{remark}

\vskip 0.4cm

By (\ref{3.89}), (\ref{3.91}) and Theorem 7.9 in \cite{Gilbarg}, the following lemma is immediate.

\begin{lem} \label{L3.5}
For any $F \in W_{2d+2}^{1,2}(\tilde D) \bigcap C^{0,1}_b(\tilde D)$, we have
$$F(t,u(t,x))\in W_{2d+2}^{1,2}((0,T_1)\times D),$$
and moreover the chain rule of weak differentiation holds for $F(t,u(t,x))$.
\end{lem}

%\begin{proof}
%Since $F(t,y)\in W_{2d+2}^{1,2}(\tilde D)$, there exists a sequence of functions $\{F_k\}_{k \ge 1} \subset C^{\infty}(\tilde D)$ such that $F_k$ converges to $F$ in $W_{2d+2}^{1,2}(\tilde D)$ as $k \to \infty$. For any $\e>0$, using (\ref{3.91}), we see that
%%$F_k(t,u_n(t,x)) \in C^{1,2}(\tilde D)$ and
%\begin{eqnarray}
%\begin{split}
%\lim_{k_1,k_2 \to \infty}\|F_{k_1}(t,u_n(t,x) )-F_{k_2}(t,u_n(t,x)) \|_{W^{1,2}_{2d+2}(\tilde {D}_{\e}')}=0,  \ \forall n \ge N_0(\e).
%\end{split}
%\end{eqnarray}
%Hence $F(t,u_n(t,x) ) \in W^{1,2}_{2d+2}(\tilde {D}_{\e}')$. Then by (\ref{3.89}) it is easy to see that $F(t,u(t,x) ) \in W^{1,2}_{2d+2}(\tilde {D}_{\e}')$ and the chain rule of weak differentiation holds for $F(t,u(t,x))$
%\end{proof}

\vskip 0.5cm

%Using (\ref{2-17-2}), it is easy to see that the following lemma holds.  The estimates listed will be used in later sections.
We close this section by showing the following Lemma. The estimates listed will be used in later sections.
\begin{lem} \label{L3.2}
Fix $\t_1  \in (0,\frac{\t_0}{2} \wedge \arctan \frac{1}{24} )$ satisfying
\begin{eqnarray}
\cos^2 \t_1 +(\frac{\cos \t_1  -(2-2\cos \t_1 )^{\frac{1}{2}} }{1+12 \tan \t_1 }-(2-2\cos \t_1 )^{\frac{1}{2}} )^2\ge 1, \label{2.5} \\
\frac{\cos \t_1  -(2-2\cos \t_1 )^{\frac{1}{2}} }{1+12 \tan \t_1 }-(2-2\cos \t_1 )^{\frac{1}{2}} \ge \cos \frac{\t_0}{2},\label{2.6} \\
\frac{ (1-4 \tan^2 \t_1 )^{\frac{1}{2}} \cos \t_1  - 2 \tan \t_1   -\frac{1}{2}}{(\frac{5}{4}+ 4  \tan^2 \t_1 +2 \tan \t_1 - (1-4 \tan^2 \t_1 )^{\frac{1}{2}} \cos \t_1 )^{\frac{1}{2}} } >\cos \t_0, \label{2.17} \\
\frac{ (1-4 \tan^2 \t_1 )^{\frac{1}{2}}  \cos \t_1  -2 \tan \t_1  +\frac{1}{2} \cos \t_1 }{(\frac{9}{4}+4  \tan^2 \t_1  +2 \tan \t_1 )^{\frac{1}{2}}} >\cos \t_0. \label{2-3}
\end{eqnarray}
Then there exist constants $0<\d_3<  \d_2,\eta_0>0$ and an integer $N_1\ge N_0$ such that for any $t_0\in [0,{T_1}]$, $z_0 \in u(t_0,\partial D)$ and $n\ge N_1$, if $t,t'\in [(t_0-\eta_0)\vee 0,(t_0+\eta_0)\wedge T_1]$ and $x,x'\in B(z_0,\d_3)$, then $|\g_n(t,x)|=|\g(t,x)|=1$ and
%\begin{eqnarray} \label{3.3}
%\begin{split}
%1-\t_1 \le |\g_n(t,x)|\le 1 , \ 1-\t_1 \le |\g(t,x)|\le 1 ,
%\end{split}
%\end{eqnarray}
%and
\begin{eqnarray} \label{3.4}
\begin{split}
\g(t,x)\cdot \g_n(t',x') \ge \cos \t_1 , \ \g(t,x)\cdot \g(t',x') \ge \cos \t_1 , \ \g_n(t,x)\cdot \g_n(t',x') \ge \cos \t_1 .
\end{split}
\end{eqnarray}
\end{lem}

\begin{proof}
Recall that $\G_{c}=\{x\in \Rd: d(x, \p D)<c\}$ for $c>0$ and $|n(x)|=1$ on $\G_{\frac{\delta_0}{2}}$.
By (\ref{3.2}) and (\ref{3.87}), there exists a constant $\d \in(0,  \d_2)$ such that for sufficiently large $n$ and for $t\in [ 0,  T_1]$, if $d(x,u(t,\p D))<\d$, then $(t,x) \in \tilde G_1$ and $u^{-1}_n(t,x) \in \G_{\frac{\delta_0}{2}}$. It implies that $|\g_n(t,x)|=|\g(t,x)|=1$ by the definition of $\g_n$ and $\g$. (\ref{3.4}) follows from (\ref{2-17-2}).
\end{proof}

\vskip 0.4cm
\section{Flows associated with the time dependent reflecting directions }
In this section, we consider the flows associated with the time dependent vector fields of reflecting directions. We will provide a number of regularity results of the hitting times of the flows on certain hyperplane. These hitting times  will be used to construct test functions
in next section for proving the pathwise uniqueness of the transformed reflecting SDEs.
\vskip 0.4cm
Let $N_0$ and $T_1$ be fixed as in Section 3. For $(t,x) \in [0,{T_1}] \times \Rd$ and $n \ge N_0$, let $y(t,x, \cdot)$ be the solution of the following ordinary differential equation:
\begin{align}
 \left\{
\begin{aligned}
& y(t, x, 0)=x,  \nonumber\\
&  \partial_{r} y(t,x, r)=\g(t, y(t, x,r)), \ r \in \R.  \nonumber
\end{aligned}
  \right.
\end{align}
and
$y_{n}(t,x, \cdot)$ the solution of the following ordinary differential equation:
\begin{align}
 \left\{
\begin{aligned}
& y_{n}(t, x, 0)=x,  \nonumber\\
&  \partial_{r} y_{n}(t,x, r)=\g_{n}(t, y_{n}(t, x,r)), \ r \in \R.  \nonumber\\
\end{aligned}
  \right.
\end{align}
Since $\g \in C^{0,1}_b([0,T_1] \times \Rd)$ and $\g_n \in C^{1,1}_b([0,T_1] \times \Rd)$ by Proposition \ref{P3.3},  we see that $y(\c,\c,\c)$ belongs to $C^{0,1,1}([0,T_1] \times \Rd \times \R)$ and $y_n(\c,\c,\c)$ belongs to $C^{1,1,1}([0,T_1] \times \Rd \times \R)$. Moreover, $\psi_i^j(t,x,r):=\partial_{x_i}y^j(t,x,r)$ is the solution to the following equation:
\begin{align}  \label{3.7}
 \left\{
\begin{aligned}
& \psi_i^j(t, x, 0)=\d_i(j),\\
&  \partial_{r} \psi_i^j(t,x, r)=\sum_{1\le k \le d}\partial_{y_k}\g^j(t, y(t, x,r))\psi_i^k(t,x, r), \ r \in \R,
\end{aligned}
  \right.
\end{align}
$\psi_{n,i}^j(t,x,r):=\partial_{x_i}y_n^j(t,x,r)$ is the solution to the following equation:
\begin{align}  \label{4.53}
 \left\{
\begin{aligned}
& \psi_{n,i}^j(t, x, 0)=\d_i(j),\\
&  \partial_{r} \psi_{n,i}^j(t,x, r)=\sum_{1\le k \le d}\partial_{y_k}\g_n^j(t, y_n(t, x,r))\psi_{n,i}^k(t,x, r), \ r \in \R,
\end{aligned}
  \right.
\end{align}
and $\Lambda_n(t,x,r):=\partial_{t}y_n(t,x,r)$ is the solution to the following equation:
\begin{align}  \label{3.10}
 \left\{
\begin{aligned}
& \Lambda_{n}(t, x, 0)=0,\\
&  \partial_{r} \Lambda_{n}(t, x, r)= (\partial_{t} \g_{n} ) (t, y_{n}(t, x, r) )+D_{y} \g_{n} (t, y_{n}(t, x, r) )\c \Lambda_{n}(t, x, r), \ r \in \R,
\end{aligned}
  \right.
\end{align}
where $y^j(t,x,r)$ and $y_n^j(t,x,r)$ are the $j$-th components of $y(t,x,r)$ and $y_n(t,x,r)$ respectively, $\d_i(j):=1$ if $i=j$ and $\d_i(j):=0$ otherwise. By (\ref{2-17-2}) and the Gronwall's inequality, it is easy to see that for any $c>0$ and $1 \le i,j \le d$,
\begin{eqnarray}
\sup _{(t, x)\in  [0,{T_1}] \times \Rd , r \in (-c,c), n \ge N_0} | \psi_{n,i}^j(t, x, r)| <\infty ,   \label{4.4}   \\
\lim_{n\to \infty}\sup_{(t,x)\in [0,{T_1}] \times \Rd , r \in (-c,c)}|y_n(t,x,r)-y(t,x,r)|=0, \label{2.29} \\
\lim_{n \to \infty}\sup _{(t, x)\in  [0,{T_1}] \times \Rd , r \in (-c,c)} | \psi_{n,i}^j(t, x, r)-\psi_{i}^j(t, x, r)| =0. \label{3.30}
\end{eqnarray}

First we have the following simple lemma.

\begin{lem} \label{L4.2}
There exists a constant $\rho_0 \in (0,1)$ such that for any $n\ge N_0$, $r \in (-\rho_0,\rho_0)$, bounded measurable function $f(x)$ and open set $A \subset \Rd$, we have
\begin{eqnarray}\label{3.16}
\begin{split}
%\int_{A}f(y(t,x,r))dx \le 2 \int_{y(t,A,r)}f(x)dx , \\
\int_{A}f(y_n(t, x ,r))dx \le 2 \int_{y_n(t,A,r)}f(x)dx.
\end{split}
\end{eqnarray}
%where $\tilde D :=\tilde u((0,{T_1})\times D)$  is defined as in Section 3.
\end{lem}

\begin{proof}
By (\ref{4.53}), it is easy to see that for  $n\ge N_0$ and $r \in (-1,1)$,
\begin{eqnarray}\label{3.16-1}
\begin{split}
\sup_{(t, x)\in  [0,{T_1}] \times \Rd}|\psi_{n,i}^j(t,x,r)-\psi_{n,i}^j(t,x,0)|\le d |r|  \|\g_n\|_{C^{0,1}_b( [0,{T_1}] \times \Rd)}  \sup_{(t, x)\in  [0,{T_1}] \times \Rd,\atop |\tau|< 1,1\le k \le d} |\psi^{k}_{n,i}(t, x, \tau)|.
\end{split}
\end{eqnarray}
 Since $(\psi_{n,i}^j(t,x,0))_{1\le i,j\le d}$ is the identity matrix, it follows from (\ref{4.4}) and (\ref{3.16-1}) that there exists a constant $\rho_0 \in (0,1)$ such that for any $n\ge N_0$, $(t,x)\in  [0,{T_1}] \times \Rd$, and $r \in (-\rho_0,\rho_0)$, $|(\psi_{n,i}^j(t,x,r))_{1\le i,j\le d}|\ge \frac{1}{2}$.
%\begin{eqnarray}\label{3.15}  \nonumber
%\begin{split}
%|(\psi_{n,i}^j(t,x,r))_{1\le i,j\le d}|\ge \frac{1}{2}.
%\end{split}
%\end{eqnarray}
%Also by (\ref{3.30}) we have $|(\psi_{i}^j(t,x,r))_{1\le i,j\le d}|=\lim_{n \to \infty}|(\psi_{n,i}^j(t,x,r))_{1\le i,j\le d}| \ge \frac{1}{2}$.
Since $(\psi_{n,i}^j(t,x,r))_{1\le i,j\le d}$ is the Jacobian matrix of $y_n(t,\c,r)$, by a change of variable, we get (\ref{3.16}).
\end{proof}

\vskip 0.5cm

% Moreover $\psi_{i}^j(t,x,r),\psi_{n,i}^j(t,x,r) \in C([0,T_1]\times \Rd \times [0,\infty))$.
From now on, we fix $t_0\in [0,{T_1}]$ and $z_0 \in u(t_0,\partial D)$.
Set
$$\rho_1:=\frac{\d_3}{4} \wedge  \rho_0 , $$
$$H_{t_0,z}:=\{x\in \Rd:(x-z)\cdot\g(t_0,z)=0\},$$
%\begin{eqnarray}
%\rho_1&:=& \frac{\d_3}{4} \wedge  \rho_0 , \label{2.92} \\
%H_{t_0,z}&:=&\{x\in \Rd:(x-z)\cdot\g(t_0,z)=0\},  \label{2.30}
%%F_n(t,x,r)&:=&(y_n(t,x,r)-z)\c \g(t_0,z), \nonumber \\
%%F(t,x,r)&:=&(y(t,x,r)-z)\c \g(t_0,z).  \nonumber
%\end{eqnarray}
for some $z \in \Rd$. Recall $N_1$ is the constant defined in Lemma \ref{L3.2}. The next lemma states that in a neighborhood of $z_0$,  $y(t,x,r)$ hits the hyperplane $H_{t_0,z}$ at a unique point $r=\G^z(t,x)$, and so does $y_n(t,x,r)$ for $n \ge N_1$.
\begin{lem} \label{L3.3}
There exist constants $\eta_1\in(0,\eta_0 )$ and $\d_4\in(0,\frac{\d_3}{2})$ such that for any $n\ge N_1$, $(t,x)\in ((t_0-\eta_1)\vee 0,(t_0+\eta_1)\wedge T_1)\times B(z_0,\d_4)$ and $z\in B(z_0,\d_4)$, there exist unique $\Gamma^z(t,x),\G_n^z(t,x) \in (-\rho_1,\rho_1)$ such that $y(t,x,\G^z(t,x)),y_n(t,x,\G_n^z(t,x))\in H_{t_0,z}$. Moreover
\begin{eqnarray}
&&\G_n^z(\c , \c) \in C_b^{1,1}(((t_0-\eta_1)\vee 0,(t_0+\eta_1)\wedge T_1)\times B(z_0,\d_4)),   \label{4.10}\\
&&\Gamma^z(\c , \c)  \in C_b^{0,1}(((t_0-\eta_1)\vee 0,(t_0+\eta_1)\wedge T_1)\times B(z_0,\d_4)),  \label{4.11}
\end{eqnarray}
and
\begin{eqnarray}
&&\p_{x_i}\G_n^z(t,x)= \! - \! \left(   \g_n (t,y_n(t,x,\G_n^z(t,x)))  \! \c  \! \g (t_0,z)  \right)^{-1} \! \!  \sum_{1  \! \le \!  j \! \le \!  d} \! \psi_{n,   i}^j(t,x,  \! \G_n^z(t,x))\g^j(t_0,z), \label{3.49}\\
&&\p_t\G_n^z(t,x)=- \left( \g_n(t,y_n(t,x,\G_n^z(t,x)))\c \g(t_0,z) \right)^{-1} \Lambda_n(t,x,\G^z_n(t,x)) \c \g(t_0,z), \label{3.14} \\
&&\p_{x_i}\G^z(t,x)=-\left(  \g (t,y(t,x,\G^z(t,x))) \c \g(t_0,z)\right)^{-1} \!  \! \sum_{1 \le j\le d}\psi_i^j(t,x, \G^z(t,x))\g^j(t_0,z), \label{3.29}\\
&& \g_n(t,y_n(t,x,\G_n^z(t,x)))\c \g(t_0,z) \ge \cos \t_1,  \label{3.13} \\
&&\lim_{n \to \infty}\sup_{(t,x)\in ((t_0-\eta_2)\vee 0,(t_0+\eta_2)\wedge T_1)\times B(z_0,\d_4)}|\G^z_n(t,x)-\G^z(t,x)|=0. \label{3.17}
\end{eqnarray}

\end{lem}

\begin{proof}
For $n\ge N_1$, set $H_{n,t}(x,z,r):=(y_n(t,x,r)-z)\c \g(t_0,z)$. Then for any $x,z\in B(z_0,\frac{\d_3}{2})$,
$t\in ((t_0-\eta_0)\vee 0,(t_0+\eta_0)\wedge T_1)$ and $r \in (-\frac{\d_3}{2},\frac{\d_3}{2})$, we have $|y_n(t,x,r)-z_0|<\d_3$. Therefore by Lemma \ref{L3.2}, we have
\begin{eqnarray}
\p_r H_{n,t}(x,z,r)=\gamma_n(t, y_n(t, x, r)) \cdot \gamma (t_{0}, z )\ge \cos \t_1 >0. \label{2.93}
\end{eqnarray}
Note that  $H_{n,t}(z_0,z_0, 0)=0$,  hence applying the implicit function theorem to $H_{n,t}(\c, \c ,\c)$, there exist constants $\eta_1\in(0,\eta_0 )$ and $\d_4 \in(0,\frac{\d_3}{2})$ sufficiently small, such that for any $n\ge N_1$ and $t  \in \!  ((   t_0 \! - \! \eta_1 )    \vee    0,   (   t_0 \! + \! \eta_1  )   \wedge    T_1)$, there exists a unique $g_{n   ,  t}(\c,\c) \! \in  \! C^{1   ,  1}   (B(z_0, \d_4 )\times B(z_0,  \d_4 ))$, such that
$H_{n,  t}(x,z,   g_{n  ,  t}(x,   z)) \! = \! 0$
and $g_{n,t}(x,z) \in (-\rho_1, \rho_1)$ for $(x,z)\in B(z_0, \d_4 )\times B(z_0,  \d_4 )$. Denote $\G_n^z(t,x):=g_{n,t}(x,z)$. Then it is easy to see that (\ref{4.10}), (\ref{3.49}) and (\ref{3.14}) hold for any $n \ge N_1$ and $z \in B(z_0, \d_4 )$.
%and $G_n(t,x,r):=(t,x,F_n(t,x,r))$. Hence $G_n(\c,\c,\c)\in C^{1,1,1}$ and $D_{(t,x,r)}G_n(t,x,r)$, the Jacobian matrix of $G_n(\c,\c,\c)$, is defined as
%$$ \left[\begin{array}{ll}{I_{d+1},} & { (a_{i} )_{1\le i \leqslant d+1}} \\ {0,} & {\gamma_n(t, y_n(t, x, r)) \cdot \gamma_n (t_{0}, z )}\end{array}  \right],$$
%where $I_d$ is $(d+1) \times (d+1)$ unite matrix, $a_1:=\p_t y_n(t,x,r)$ and $a_i:=\partial_{x_{i-1}}F(t,x,r)=\sum_{1\le j \le d}\psi_{n,i-1}^j(t,x,r)\g^j(t_0,z)$ for $2\le i \le d$.
%By (\ref{4.4}) and (\ref{3.49}), $|\nabla_x \G_n^z(t,x)|$ is bounded in $((t_0-c_1)\vee 0,(t_0+c_1)\wedge T_1)\times B(z_0,c_2)$. Together with $\G^z_n(t,z)=0$, it follows that there exist
%$\eta_1\in (0, c_1)$ and $\d_4\in(0,  c_2)$ sufficiently small, such that for any $(t,x)\in ((t_0-\eta_1)\vee 0,(t_0+\eta_1)\wedge T_1)\times B(z_0,\d_4)$ and $z\in B(z_0,\d_4)$, we have $\G_n^z(t,x) \in (-\rho_1, \rho_1)$, where $\rho_1$ is defined in (\ref{2.92}).
Since $|\G_n^z(t,x)|< \rho_1< \frac{\d_3}{2}$, by (\ref{2.93}) we get (\ref{3.13}).

\vskip 0.3cm
 Applying the implicit function theorem to $H_t(x,z,r):=(y(t,x,r)-z)\c \g(t_0,z)$, we can choose common constants $\eta_1,\d_4>0$, such that for $(t,x)\in ((t_0-\eta_1)\vee 0,(t_0+\eta_1)\wedge T_1)\times B(z_0,\d_4)$ and $z\in B(z_0,\d_4)$, there exists a unique $ \G^z(t,x)\in (-\rho_1,\rho_1)$ such that $ y(t,x,\G^z(t,x))\in H_{t_0,z}$. Moreover (\ref{4.11})
 and (\ref{3.29}) hold.

\vskip 0.3cm

Next we prove (\ref{3.17}).
By (\ref{2.29}) we have
\begin{eqnarray}  \label{2.18}
\begin{split}
&\ \ \ \    \lim_{n \to \infty}\sup_{(t,x)\in ((t_0-\eta_1)\vee 0,(t_0+\eta_1)\wedge T_1)\times B(z_0,\d_4)}|H_{t}(x,z,\G^z_n(t,x)) -H_{t}(x,z,\G^z(t,x))| \\
&=\lim_{n \to \infty}\sup_{(t,x)\in ((t_0-\eta_1)\vee 0,(t_0+\eta_1)\wedge T_1)\times B(z_0,\d_4)}|H_{t}(x,z,\G^z_n(t,x)) -H_{n,t}(x,z,\G_n^z(t,x))|=0.
\end{split}
\end{eqnarray}

On the other hand, for $(t,x)\in ((t_0-\eta_1)\vee 0,(t_0+\eta_1)\wedge T_1)\times B(z_0,\d_4)$ and $r \in (-\rho_1,\rho_1)$, since $|y(t,x,r)-z_0| \le |r|+|x-z_0|<\d_3$, together with Lemma \ref{L3.2} we have
\begin{eqnarray} \nonumber
\begin{split}
&\ \ \ \    \sup_{(t,x)\in ((t_0-\eta_1)\vee 0,(t_0+\eta_1)\wedge T_1)\times B(z_0,\d_4) \atop r \in (-\rho_1,\rho_1), n\ge N_1}   |\partial_r H_{t}(x,z,r)|\\
&=\sup_{(t,x)\in ((t_0-\eta_1)\vee 0,(t_0+\eta_1)\wedge T_1)\times B(z_0,\d_4) \atop r \in (-\rho_1,\rho_1), n\ge N_1}   |\g(t,y(t,x,r))\c \g(t_0,z)| \ge \cos \t_1 .
\end{split}
\end{eqnarray}
Hence, taking into account  (\ref{2.18}),
$$|\G^z_n(t,x)-\G^z(t,x)|\leq \frac{1}{\cos \theta_1}|H_{t}(x,z,\G^z_n(t,x)) -H_{t}(x,z,\G^z(t,x)) |,$$
which yields (\ref{3.17}).
%
%\vskip 0.3cm
%
%By (\ref{3.17}) and the continuity of $\G_n^z(t,x)$, we see that $\G^z(t,x)$ is continuous $w.r.t.$ $(t,x)$, which further yields $\Gamma^z(t,x)  \in C_b^{0,1}(((t_0-\eta_1)\vee 0,(t_0+\eta_1)\wedge T_1)\times B(z_0,\d_4))$.
\end{proof}

\vskip 0.4cm

% Define $z(\d):=y(t_0,z_0,\frac{\d}{2})$.
Define
\begin{eqnarray} \label{2.19}
\begin{split}
C(z,\d):=\bigcup_{c\in \R}B(z-c\g(t_0,z),2\d \tan \t_1 )\bigcap B(z,\d).
\end{split}
\end{eqnarray}
We have the following relationship.

\begin{lem} \label{L2.7}
For any $\d\in(0,\frac{\d_4}{2}]$, $t\in((t_0-\eta_1)\vee 0,(t_0+\eta_1)\wedge T_1)$ and $z:=y(t_0,z_0,\frac{\d}{2})$, we have
\begin{eqnarray}\label{2-201}
\begin{split}
|y(t,x,\G^z(t,x))-z| < 3\d \tan \t_1 , \ \forall x\in C(z,\d),
\end{split}
\end{eqnarray}
and
\begin{eqnarray}
\{x\in B(z,\d): y(t,x,\G^z(t,x))\in B(z,\d \tan \t_1 ) \}\subset C(z,\d). \label{3.83}
% \{x\in B(z,\d):    y_n(t,x,\G^z_n(t,x))\in B(z,\d \tan \t_1 )\}\subset C(z,\d). \label{4.5}
\end{eqnarray}
\end{lem}

\begin{proof}
Fix $\d\in(0,\frac{\d_4}{2}]$, $t\in((t_0-\eta_1)\vee 0,(t_0+\eta_1)\wedge T_1)$, $z:=y(t_0,z_0,\frac{\d}{2})$ and $x\in B(z,\d)$. Define $P(\a):=(\a \c \g(t_0,z))\g(t_0,z)$ and $Q(\a):=\a-P(\a)$ for $\a \in \Rd$. It is easy to see that for any $r \in (0,\G^z(t,x)]$, we have
\begin{eqnarray}  \nonumber
\begin{split}
 |y(t,x,r)-z_0| & \le|y(t,x,r)-x|+|x-z|+|z-z_0| \\
 &\le |r|+\d+\frac{\d}{2}<\d_3.
\end{split}
\end{eqnarray}
Therefore by Lemma \ref{L3.2} and the fact that
\begin{eqnarray}
P(y(t,x,\G^z(t,x))-z)=0, \label{4.6}
\end{eqnarray}
we have
\begin{eqnarray} \label{2-4}  \nonumber
\begin{split}
|Q(\g(t,y(t,x,r)))|^2&=|\g(t,y(t,x,r))-(\g(t,y(t,x,r))\c \g(t_0,z))\g(t_0,z)|^2\\
&= |\g(t,y(t,x,r))|^2- (\g(t,y(t,x,r))\c \g(t_0,z))^2\\
&\le 1- \cos^2 \t_1 \\
&= \cos^2 \t_1  \tan^2 \t_1 \\
&\le (\g(t,y(t,x,r))\c \g(t_0,z))^2 \tan^2 \t_1 \\
&= |P(\g(t,y(t,x,r)))|^2 \tan^2 \t_1 ,\\
\end{split}
\end{eqnarray}
and
\begin{eqnarray} \label{2-5}  \nonumber
\begin{split}
|P(y(t,x,\G^z(t,x))-x)|&= |P(z-x)+P(y(t,x,\G^z(t,x))-z)|\\
&= |P(z-x)| \\
&\le  |z-x| < \d.
\end{split}
\end{eqnarray}
Hence it follows that
\begin{eqnarray} \label{2.26}
\begin{split}
|Q(y(t,x,\G^z(t,x))-x)|&\le |\int_{0}^{\G^z(t,x)}|Q(\g(t,y(t,x,r)))|dr|\\
&\le |\int_{0}^{\G^z(t,x)} |P(\g(t,y(t,x,r)))|\tan \t_1  dr| \\
&=  |\int_{0}^{\G^z(t,x)} \g(t,y(t,x,r))  \c \g(t_0,z) dr| \tan \t_1 \\
&=  |( \int_{0}^{\G^z(t,x)} \g(t,y(t,x,r))  dr \c \g(t_0,z))| \tan \t_1 \\
&=  |P(y(t,x,\G^z(t,x))-x)| \tan \t_1 \\
&< \d \tan \t_1 .
\end{split}
\end{eqnarray}

 For $x \in C(z,\d)$, we easily  see that $|Q(z-x)|< 2\d \tan \t_1 $. Hence by (\ref{4.6}) and (\ref{2.26}) we get that
\begin{eqnarray}  \nonumber
\begin{split}
|y(t,x,\G^z(t,x))-z|&= |Q(y(t,x,\G^z(t,x))-z)|\\
&\le |Q(y(t,x,\G^z(t,x))-x)|+|Q(z-x)|\\
&< \d \tan \t_1+2\d \tan \t_1  =3\d \tan \t_1 .
\end{split}
\end{eqnarray}
which is (\ref{2-201}).
\vskip 0.3cm
Next we prove (\ref{3.83}).
If $x\in B(z,\d) \setminus C(z,\d)$, then we have
$$|Q(z-x)|\ge 2\d \tan \t_1  .$$
Hence by  (\ref{4.6}) and (\ref{2.26}) we get that
\begin{eqnarray}  \nonumber
\begin{split}
|y(t,x,\G^z(t,x))-z|&= |Q(y(t,x,\G^z(t,x))-z)|\\
&\ge |Q(z-x)|-|Q(y(t,x,\G^z(t,x))-x)|\\
&> 2\d \tan \t_1  -\d \tan \t_1 =\d \tan \t_1 ,
\end{split}
\end{eqnarray}
which implies (\ref{3.83}).
%\begin{eqnarray}\label{2.2}
%\begin{split}\{x\in B(z,\d): y(t,x,\G^z(t,x))\in B(z,\d \tan \t_1 ) \}\subset \bigcup_{c\in \R}B(z-c\g(t_0,z),2\d \tan \t_1 ).
%\end{split}
%\end{eqnarray}
%Similarly, we also have
%\begin{eqnarray}\label{2.3}
%\begin{split}
%\{x\in B(z,\d): y_n(t,x,\G_n^z(t,x))\in B(z,\d \tan \t_1 ) \}\subset \bigcup_{c\in \R}B(z-c\g(t_0,z),2\d \tan \t_1 ).
%\end{split}
%\end{eqnarray}
%Combining (\ref{2.2}) and (\ref{2.3}), we get (\ref{3.83}).
\end{proof}

\vskip 0.4cm

%Note that $\g_n$ converges to $\g$ in $W^{1,2}_{2d+2}(\tilde {D}_{\e})$, where $\tilde {D}_{\e}$ is the strict interior of $\tilde D$ having a positive distance with %$\tilde D$, and $y_n(t,x,r) \notin u(t,D)$ even for $x \in u(t,D)$.
The following Lemma plays an important role in the proofs of Proposition \ref{P4.1.1} and Proposition \ref{P4.2.1}, which establish the convergence of $\G^z_n(t,x)$ and $y_n(t,x,\G^z_n(t,x))$ in some Sobolev spaces and provide further regularities of $y(t,x,\G^z(t,x))$. Recall that
$$D(t,c)=\{x:d(x,u(t,D)^c)> c \},$$
$\t_0$ and $N_0(\e)$ were defined in Proposition \ref{P2.2} and Lemma \ref{L3.2.2} respectively.

\begin{lem}  \label{L3.4}
There exist constants $\d_5\in(0, \d_4 )$ and $\eta_2\in (0,\eta_1)$ such that for any $\e>0$, there exists an integer $N_1(\e)>N_1 \vee N_0(\e)$ satisfying that for $z:=y(t_0,z_0,\frac{\d_5}{2})$, $n,m\ge N_1(\e)$, $t\in((t_0-\eta_2)\vee 0,(t_0+\eta_2)\wedge T_1)$ and $x\in C(z,\d_5)\bigcap D(t,\e)$, if $\G^z(t,x)\neq 0$, then we have
\begin{eqnarray}
y(t,x,r)\in D(t,(\e \sin \frac{\t_0}{2}) \wedge  (\frac{\d_5}{16}\sin \frac{\t_0}{2})) \ \mbox{for} \ r\in (0, \G^z(t,x)], \label{3.46}
\end{eqnarray}
and
\begin{eqnarray}
y_n(t,x,r)\in D(t,(\frac{\e}{2} \sin \frac{\t_0}{2}) \wedge  (\frac{\d_5}{32}\sin \frac{\t_0}{2})),  \label{3.48}
\end{eqnarray}
for $r\in (0,  \G_n^z(t,x) ] \cup (\G^z_n(t,x),\G^z_m(t,x)]\cup (\G^z_m(t,x),\G^z(t,x)]$.
%If $\G^z(t,x)= 0$, then for any $x' \in B(x,\frac{\d_5}{32}\sin \frac{\t_0}{2})$, we have
%\begin{eqnarray}
% y(t,x',r)\in D(t,\frac{\d_5}{64}\sin \frac{\t_0}{2}) \ \mbox{for} \ r\in (-\frac{\d_5}{64}\sin \frac{\t_0}{2}, \frac{\d_5}{64}\sin \frac{\t_0}{2}), \label{2.7}   \\
%y_n(t,x',r)\in D(t,\frac{\d_5}{64}\sin \frac{\t_0}{2}) \ \mbox{for} \ r\in (-\frac{\d_5}{64}\sin \frac{\t_0}{2}, \frac{\d_5}{64}\sin \frac{\t_0}{2}). \label{3.78}
%\end{eqnarray}
\end{lem}

\begin{proof}
Fix $\d_5:= \frac{\d_4 \sin \t_0}{16} $ and $z:=y(t_0,z_0,\frac{\d_5}{2})$.

First we show that for
%$t\in((t_0-\eta_1)\vee 0,(t_0+\eta_1)\wedge T_1)$ and $\G^z(t_0,x)\le 0$, then $\G^z(t,x)\le 0$.
%
%Suppose there exists a  $t'\in((t_0-\eta_1)\vee 0,(t_0+\eta_1)\wedge T_1)$ such that $\G^z(t',x)>0$. By the definition of $\G^z$, we have
%\begin{eqnarray} \nonumber
%\begin{split}
%(y(t',x,\G^z(t',x))-y(t_0,x,\G^z(t_0,x)))\c \g(t_0,z)=0.
%\end{split}
%\end{eqnarray}
%However, since $\G^z(t_0,x)\le 0$, by Lemma \ref{L3.2},
%\begin{eqnarray} \nonumber
%\begin{split}
%&\ \ \ \ (y(t',x,\G^z(t',x))-y(t_0,x,\G^z(t_0,x)))\c \g(t_0,z) \\
%&=(y(t',x,\G^z(t',x))-x)\c \g(t_0,z)+(x-y(t_0,x,\G^z(t_0,x)))\c \g(t_0,z)\\
%&=\int_0^{\G^z(t',x)}\g(t',y(t',x,r)) \c \g(t_0,z )dr+\int_{\G^z(t_0,x)}^{0}\g(t_0,y(t_0,x,r)) \c \g(t_0,z )dr\\
%&\ge \G^z(t',x) \cos \t_1  - \G^z(t_0,x) \cos \t_1  >0,
%\end{split}
%\end{eqnarray}
%which is a contradiction. Hence for any $t \in((t_0-\eta_1)\vee 0,(t_0+\eta_1)\wedge T_1)$, we have $\G^z(t,x)\le 0$.
%
%\vskip 0.4cm
%Next we show that if
$t\in((t_0-\eta_1)\vee 0,(t_0+\eta_1)\wedge T_1)$ and $x\in C(z,\d_5)$, if $\G^z(t,x)\le 0$ and $r\in [\G^z(t,x),0]$, then $d(y(t,x,r), u(t_0,D)^c)> \frac{\d_5}{8}\sin \frac{\t_0}{2}$.
\vskip 0.3cm

By (\ref{2-201}), we get
\begin{eqnarray} \label{2.25}
\begin{split}
|y(t,x,\G^z(t,x))-z_0|\le |y(t,x,\G^z(t,x))-z|+|z-z_0|\le 3\d_5 \tan \t_1  + |z-z_0|.
\end{split}
\end{eqnarray}
On the other hand, by Lemma \ref{L3.2} and the fact that  $\cos \t_1  > \frac{1}{4}$ we have
\begin{eqnarray} \label{2.27}
\begin{split}
|z-z_0|^2&=\int_0^{\frac{\d_5}{2}}\g(t_0,y(t_0,z_0,r)) \c (z-z_0)dr\\
&=\int_0^{\frac{\d_5}{2}}dr \int_0^{\frac{\d_5}{2}} \g(t_0,y(t_0,z_0,r)) \c \g(t_0,y(t_0,z_0,\tau))d\tau\\
&>  \int_0^{\frac{\d_5}{2}}dr \int_0^{\frac{\d_5}{2}} \frac{1}{4} d\tau =  \frac{\d_5^2}{16},
\end{split}
\end{eqnarray}
i.e.,  $\frac{\d_5}{|z-z_0|}< 4$. By Lemma \ref{L3.2} and (\ref{2.25}) we have
\begin{eqnarray} \label{2-16-1}
\begin{split}
&\ \ \ \ (y(t,x,\G^z(t,x))-z_0)\c \g(t_0,z_0)\\
&=(z-z_0)\c \g(t_0,z_0)+(y(t,x,\G^z(t,x))-z)\c \g(t_0,z_0)\\
&=\int_0^{\frac{\d_5}{2}}\g(t_0,y(t_0,z_0,\tau)) \c \g(t_0,z_0)d\tau+(y(t,x,\G^z(t,x))-z)\c (\g(t_0,z_0)-\g(t_0,z))\\
&\ge  \frac{\d_5}{2} \cos \t_1  -|y(t,x,\G^z(t,x))-z| |\g(t_0,z_0)-\g(t_0,z)|\\
&\ge |z-z_0| \cos \t_1  -|y(t,x,\G^z(t,x))-z| (2-2\cos \t_1 )^{\frac{1}{2}} \\
&\ge |z-z_0| (\cos \t_1  -(2-2\cos \t_1 )^{\frac{1}{2}} )-|y(t,x,\G^z(t,x))-z_0| (2-2\cos \t_1 )^{\frac{1}{2}} \\
&\ge (\frac{|z-z_0| (\cos \t_1  -(2-2\cos \t_1 )^{\frac{1}{2}} )}{|z-z_0|+3\d_5 \tan \t_1 }-(2-2\cos \t_1 )^{\frac{1}{2}} )|y(t,x,\G^z(t,x))-z_0| \\
&>  (\frac{\cos \t_1  -(2-2\cos \t_1 )^{\frac{1}{2}} }{1+12 \tan \t_1 }-(2-2\cos \t_1 )^{\frac{1}{2}} )|y(t,x,\G^z(t,x))-z_0|,
\end{split}
\end{eqnarray}
where the fact  $\frac{\d_5}{|z-z_0|}< 4$ has been used in the last inequality.
Note that for $\tau \in [\G^z(t,x),0]$,
\begin{eqnarray}
\begin{split}
|y(t,x,\tau)-z_0|\le |y(t,x,\tau)-x|+|x-z |+|z-z_0|< \d_3, \label{4.7}
\end{split}
\end{eqnarray}
 which implies that $\g(t,y(t,x,\tau))\c \g(t_0,z_0)\ge   \cos \t_1  $ by Lemma \ref{L3.2}. Together with (\ref{2.5}) and (\ref{2-16-1}) we get  $ \g(t,y(t,x,\tau)) \c(y(t,x,\G^z(t,x))-z_0) \ge 0$. Hence
\begin{eqnarray} \label{2.28}
\begin{split}
&\ \ \ \ |y(t,x,r)-z_0|^2\\
&=|y(t,x,r)-y(t,x,\G^z(t,x))|^2+|y(t,x,\G^z(t,x))-z_0|^2\\
&\ \ \ \ +2(y(t,x,r)-y(t,x,\G^z(t,x))) \c (y(t,x,\G^z(t,x))-z_0)\\
&\ge |y(t,x,\G^z(t,x))-z_0|^2+2\int_{\G^z(t,x)}^{r} \g(t,y(t,x,\tau)) \c (y(t,x,\G^z(t,x))-z_0)d\tau\\
&\ge  |y(t,x,\G^z(t,x))-z_0|^2. \\
\end{split}
\end{eqnarray}
Combining (\ref{2-201}), (\ref{2.27}), (\ref{2.28}) and the fact that $\t_1< \arctan \frac{1}{24}$, we deduce that
\begin{eqnarray}\label{2.8}
\begin{split}
|y(t,x,r)-z_0|& \ge |y(t,x,\G^z(t,x))-z_0| \\
& \ge |z-z_0|-|y(t,x,\G^z(t,x))-z| \\
& > \frac{\d_5}{4}-3\d_5 \tan \t_1 > \frac{\d_5}{8}.
\end{split}
\end{eqnarray}

By Lemma \ref{L3.2}, (\ref{2-16-1}) and (\ref{4.7}), we have
\begin{eqnarray} \label{2-16-2}
\begin{split}
&\ \ \ \ (y(t,x,r)-z_0)\c \g(t_0,z_0)\\
&=(y(t,x,r)-y(t,x,\G^z(t,x)))\c \g(t_0,z_0)+(y(t,x,\G^z(t,x))-z_0)\c \g(t_0,z_0)\\
&> \int_{\G^z(t,x)}^{r} \g(t,y(t,x,\tau))\c \g(t_0,z_0)d\tau  + (\frac{\cos \t_1  -(2-2\cos \t_1 )^{\frac{1}{2}} }{1+12 \tan \t_1 }-(2-2\cos \t_1 )^{\frac{1}{2}} ) \\
&\ \ \ \ \times |y(t,x,\G^z(t,x))-z_0|  \\
&\ge (r-\G^z(t,x))\cos \t_1     +    (\frac{\cos \t_1   -   (2  -  2\cos \t_1 )^{\frac{1}{2}} }{1+12 \tan \t_1 }   -   (2-2\cos \t_1 )^{\frac{1}{2}} )  \\
&\ \ \ \ \times |y(t,x,\G^z(t,x))-z_0|  \\
&\ge |y(t,x,r)-y(t,x,\G^z(t,x))|\cos \t_1  + (\frac{\cos \t_1  -(2-2\cos \t_1 )^{\frac{1}{2}} }{1+12 \tan \t_1 }-(2-2\cos \t_1 )^{\frac{1}{2}} )  \\
&\ \ \ \ \times |y(t,x,\G^z(t,x))-z_0|  \\
&> |y(t,x,r)-z_0| \cos \frac{\t_0}{2},
\end{split}
\end{eqnarray}
where $\t_1<\frac{\t_0}{2}$ and (\ref{2.6}) have been used for the last inequality. Now, (\ref{4.7}) and (\ref{2-16-2}) implies that
$$y(t,x,r) \in C(z_0,\g(t_0,z_0), \frac{\t_0}{2},\d_2)\subset u(t_0,D).$$
Combining this with  (\ref{2.8}) we obtain that
\begin{eqnarray} \label{4.8}
\begin{split}
d(y(t,x,r), u(t_0,D)^c)\ge d(y(t,x,r), \p C(z_0,\g(t_0,z_0),\t_0,\d_2))>  \frac{\d_5}{8} \sin \frac{\t_0}{2}.
\end{split}
\end{eqnarray}

\vskip 0.3cm

Note that $\tilde {D}_{\frac{\d_5}{16} \sin \frac{\t_0}{2}}=\{(t,y):t\in (0,T_1),\ y \in D(t, \frac{\d_5}{16} \sin \frac{\t_0}{2})\}$ is an open set in $\Rdd$ by Lemma \ref{L3.2.2}, and $\{ t_0 \} \times \bar A \subset \tilde {D}_{\frac{\d_5}{16} \sin \frac{\t_0}{2}}$ by (\ref{4.8}), where
$$A:=\{y(t,x,r): \ t\in((t_0-\eta_1)\vee 0,(t_0+\eta_1)\wedge T_1), \ x\in C(z,\d_5),\ \G^z(t,x)\le 0,   \ r\in [\G^z(t,x),0]\}. $$
Hence there exists a $\eta_2 \in (0,\eta_1 \wedge (\frac{\d_4\sin \t_0}{8M_0})^{1/a_0})$ such that
$$ ((t_0-\eta_2)\vee 0,(t_0+\eta_2)\wedge T_1) \times A \subset \tilde {D}_{\frac{\d_5}{16} \sin \frac{\t_0}{2}}.$$
Obviously, for $t\in((t_0-\eta_2)\vee 0,(t_0+\eta_2)\wedge T_1)$, $x\in C(z,\d_5)$, if $\G^z(t,x)\le 0$ and $r\in [\G^z(t,x),0]$, then
\begin{eqnarray}\label{2.12}
\begin{split}
d(y(t,x,r), u(t,D)^c)> \frac{\d_5}{16}\sin \frac{\t_0}{2}.
\end{split}
\end{eqnarray}
%In particular, we have $z=y(t,z,0)\in u(t,D)$.
% and $\sup_{\d_5 \in(0,\d_4)}d(y(t_0,z_0,\frac{\d}{2}), \p C(z_0,\g(t,z_0),\t_0,\d_2))>0$.

%\vskip 0.3cm
%
%For $t\in((t_0-\eta_2)\vee 0,(t_0+\eta_2)\wedge T_1)$, $x\in C(z,\d_5)$, if $\G^z(t,x)= 0$, then for any $x' \in B(x,\frac{\d_5}{32}\sin \frac{\t_0}{2})$ and $r\in (-\frac{\d_5}{64}\sin \frac{\t_0}{2}, \frac{\d_5}{64}\sin \frac{\t_0}{2})$, by (\ref{2.12}) we have $d(y(t,x',r),u(t,D)^c)\ge d(x,u(t,D)^c)-d(x,x')-d(y(t,x',r),x')>\frac{\d_5}{16}\sin \frac{\t_0}{2}-\frac{\d_5}{32}\sin \frac{\t_0}{2}-\frac{\d_5}{64}\sin \frac{\t_0}{2}=\frac{\d_5}{64}\sin \frac{\t_0}{2}$, which proves (\ref{2.7}).

\vskip 0.3cm

Next we will prove (\ref{3.46}). By (\ref{2.12}) we just need to show that for $\e>0$, \\
 $t\in((t_0-\eta_2)\vee 0,(t_0+\eta_2)\wedge T_1)$ and $x\in C(z,\d_5)\bigcap D(t,\e)$, if $\G^z(t,x)>0$ and $r \in (0, \G^z(t,x)]$, then $d(y(t,x,r) , u(t,D)^c)> \e \sin \frac{\t_0}{2} $.
\vskip 0.3cm
Set $z_0':=u(t,u^{-1}(t_0,z_0))\in u(t, \p D)$ and $F(a):=x+a\g(t,z_0')$ for $a \in \R$.
% Without lose of generality, we assume $\G^z(t,x)>0$.
Obviously, there exists a constant $a_1 <0$ such that $F(a_1) \in \p \cup_{r>0}C(z_0',-\g(t,z_0'),\t_0,r))$. Noting $\eta_2<(\frac{\d_4\sin \t_0}{8M_0})^{1/a_0}$ and $\d_5=\frac{\d_4 \sin \t_0}{16}$, by (\ref{3.1}) we have
\begin{eqnarray} \label{2.9}
\begin{split}
|z_0-z_0'|&= |u(t_0,u^{-1}(t_0,z_0))-u(t,u^{-1}(t_0,z_0))|\\
&\le M_0|t_0-t|^{\a_0}< \frac{\d_4 \sin \t_0}{8},
\end{split}
\end{eqnarray}
and
\begin{eqnarray}  \nonumber
\begin{split}
|x-z_0'|&\le  |x-z|+|z-z_0|+|z_0-z_0'|\\
&< 2\d_5 +\frac{\d_4 \sin \t_0}{8}=\frac{\d_4 \sin \t_0}{4}.
\end{split}
\end{eqnarray}
Therefore
\begin{eqnarray} \nonumber
\begin{split}
|F(a_1)-z_0'|=\frac{d(z_0',\{F(a):a\in \R\})}{\sin \t_0} \le \frac{|x-z_0'|}{\sin \t_0} <\frac{\d_4  }{4},
\end{split}
\end{eqnarray}
which implies that $F(a_1) \in \bar C(z_0',-\g(t,z_0'),\t_0,\d_2) \subset u(t,D)^c$.
Since
$$F(a_1)\in u(t,D) ^c \cap B(z_0',\frac{\d_4  }{4})$$
and
$$F(0)=x\in u(t,D) \cap B(z_0', \frac{\d_4  }{4}),$$
there eixsts a constant $a_2\in (a_1,0)$ such that $F(a_2)\in u(t,\p D) \cap B(z_0',\frac{\d_4  }{4})$.

For any $\tau \in (0, \G^z(t,x)]$, since $\tau <\rho_1 \le \frac{\d_3}{4}$ and $\d_4 < \frac{\d_3  }{2}$, we have
\begin{eqnarray} \label{2.10}
\begin{split}
|y(t,x,\tau)-F(a_2)|&\le |y(t,x,\tau)-x|+|x-F(a_2)| \\
&\le \tau+|x-F(a_2)|\\
&\le \tau+|x-F(a_1)| \\
&\le  \tau+|x-z_0'|+|F(a_1)-z_0'|\\
&< \tau+\frac{\d_4\sin \t_0}{4}+\frac{\d_4  }{4}< \frac{\d_3}{2},
\end{split}
\end{eqnarray}
which implies that
\begin{eqnarray}  \label{4.9}
\begin{split}
|y(t,x,\tau)-z_0| &\le|y(t,x,\tau)-F(a_2)|+|F(a_2)-z_0'|+|z_0'-z_0|  \\
&<\frac{\d_3}{2}+\frac{\d_4  }{4}+\frac{\d_4 \sin \t_0}{8}<\d_3.
\end{split}
\end{eqnarray}
Together with Lemma \ref{L3.2}, (\ref{2.9}) and the fact that $x \in D(t,\e)$, we deduce that
\begin{eqnarray} \label{2.11}
\begin{split}
&\ \ \ \ |y(t,x,r)-F(a_2)|^2\\
&=|y(t,x,r)-x|^2+|x-F(a_2)| ^2+2(-a_2)\int_0^r\g(t,y(t,x,\tau))\c \g(t,z_0') d\tau\\
&\ge  |x-F(a_2)| ^2\ge d(x, u(t,D)^c)^2> {\e}^2.
\end{split}
\end{eqnarray}
On the other hand, note that $|F(a_2)-z_0|\le |F(a_2)-z_0'|+|z_0'-z_0|<  \frac{\d_4  }{4}+ \frac{\d_4 \sin \t_0}{8}<\d_3$. Combining this with  Lemma \ref{L3.2}, (\ref{2.9}) and (\ref{4.9}), we have
\begin{eqnarray}  \nonumber
\begin{split}
&\ \ \ \ (y(t,x,r)-F(a_2))\c \g(t,F(a_2))\\
&= (y(t,x,r)-x)\c \g(t,F(a_2))+(x-F(a_2))\c \g(t,F(a_2))\\
&=\int_0^r \g(t,y(t,x,\tau)) \c \g(t,F(a_2))  d\tau+(-a_2) \g(t,z_0')  \c \g(t,F(a_2)) \\
&\ge r \cos \t_1  +(-a_2)\cos \t_1 \\
&\ge  \cos \t_1  |y(t,x,r)-x|+\cos \t_1  |x-F(a_2)| \\
&\ge \cos \t_1   |y(t,x,r)-F(a_2)|> \cos \frac{\t_0}{2} |y(t,x,r)-F(a_2)|.\\
\end{split}
\end{eqnarray}
Together with (\ref{2.10}), we have $y(t,x,r) \in  C(F(a_2),\g(t,F(a_2)), \frac{\t_0}{2},\d_2) \subset u(t,D)$. Hence combining this with (\ref{2.11}) we obtain that
\begin{eqnarray}  \nonumber
\begin{split}
d(y(t,x,r) , u(t, D)^c) &\ge d(y(t,x,r) , \partial C(F(a_2),\g(t,F(a_2)),\t_0,\d_2)) \\
&\ge|y(t,x,r) -F(a_2)| \sin \frac{\t_0}{2} > \e \sin \frac{\t_0}{2} .
\end{split}
\end{eqnarray}

\vskip 0.3cm

Finally we prove (\ref{3.48}). By (\ref{2.29}), there exists an integer $\tilde N(\e)>N_1 \vee N_0(\e)$ such that for $n>\tilde N(\e)$,
\begin{eqnarray} \nonumber
\begin{split}
|y_n(t,x,r)-y(t,x,r)| < (\frac{\e}{4} \sin \frac{\t_0}{2}) \wedge  (\frac{\d_5}{64}\sin \frac{\t_0}{2})  \ \mbox{for} \ r\in (0,  \G^z(t,x) ].
\end{split}
\end{eqnarray}
Using this and (\ref{3.46}), we see that
\begin{eqnarray}
\begin{split}
y_n(t,x,r) \in D(t,(\frac{3\e}{4} \sin \frac{\t_0}{2}) \wedge  (\frac{3\d_5}{64}\sin \frac{\t_0}{2}) )  \ \mbox{for} \ r\in (0,  \G^z(t,x) ]. \label{3.47} \\
\end{split}
\end{eqnarray}
By (\ref{3.17}), (\ref{3.47}) and noting
$$\sup_{(t,x)\in [0,T_1]\times \Rd,\atop \tau \in(-\rho_1,\rho_1),n\ge N_0}|\p_r y_n(t,x,\tau)|<\infty,$$
for any $\e>0$, there exists an integer $N_1(\e)>\tilde N(\e)$ such that for $n,m>N_1(\e)$ and \\
$r \in (\G^z (t,x),\G^z_n(t,x)]\bigcup (\G^z_n (t,x),\G_m^z(t,x)]$, we have
$$|y_n(t,x,r)-y_n(t,x,\G^z(t,x))|< (\frac{\e}{4} \sin \frac{\t_0}{2}) \wedge  (\frac{\d_5}{64}\sin \frac{\t_0}{2}),$$
and $d(y_n(t,x,\G^z(t,x)),u(t,D)^c)> (\frac{3\e}{4} \sin \frac{\t_0}{2}) \wedge  (\frac{3\d_5}{64}\sin \frac{\t_0}{2}) $. Hence
\begin{eqnarray}  \nonumber
\begin{split}
d(y_n(t,x,r),u(t,D)^c)& \ge  d(y_n(t,x,\G^z_n(t,x)),u(t,D)^c)-|y_n(t,x,r)-y_n(t,x,\G^z_n(t,x))|\\
&> (\frac{\e}{2} \sin \frac{\t_0}{2}) \wedge  (\frac{\d_5}{32}\sin \frac{\t_0}{2}).
\end{split}
\end{eqnarray}
Together with (\ref{3.47}), we obtain (\ref{3.48}).
\end{proof}

\vskip 0.4cm

From now on, we fix $z:=y(t_0,z_0,\frac{\d_5}{2})$.
Recall that $\psi_i^j$ and $\Lambda_n$ were defined in (\ref{3.7}) and (\ref{3.10}) respectively, $D(t,\e)$ and $C(z,\d_5)$ were defined in (\ref{2.19-1}) and (\ref{2.19}) respectively. Set
$$\O_{\e}:=\{(t,x):  t\in ((t_0-\eta_2)\vee 0,(t_0+\eta_2)\wedge T_1), \ x\in C(z,\d_5) \bigcap D(t,\e)\}. $$
\vskip 0.4cm

The regularity of $\G^z(t,x)$ and the convergence of $\G_n^z(t,x)$ are stated in the following two propositions. The proofs of theses results are quite lengthy. They are put in the appendix.

\begin{prp}  \label{P4.1.1}
For any  $1\le i,j\le d$,
\begin{eqnarray} \label{4.48}
\psi_{i}^j(\c , \c , \G^z( \c , \c ))\in W_{2d+2}^{0,1}((((t_0-\eta_2)\vee 0,(t_0+\eta_2)\wedge T_1)\times C(z,\d_5))\bigcap \tilde D),
\end{eqnarray}
and
\begin{eqnarray} \label{4.49}
\Gamma^z ( \c , \c ) \in W_{2d+2}^{0,2}(((t_0-\eta_2)\vee 0,(t_0+\eta_2)\wedge T_1)\times C(z,\d_5)\bigcap \tilde D).
\end{eqnarray}
\end{prp}

\begin{prp}  \label{P4.2.1}
Let $N_1(\e)$ be given in Lemma \ref{L3.4}, then
\begin{eqnarray} \label{4.50}
\begin{split}
\sup_{\e>0} \sup_{n\ge N_1(\e)} \| \Lambda_n(\cdot,\cdot,\G_n^z(\cdot,\cdot)) \|_{L^{2d+2}(\O_{\e})}<\infty.
\end{split}
\end{eqnarray}
Moreover, for any $\e>0$, we have
\begin{eqnarray}
\lim_{n,m \to \infty} \| \Lambda_n (\cdot, \cdot, \G_{n}^{z}(\cdot, \cdot) )  -\Lambda_m (\cdot, \cdot, \G_{m}^{z}(\cdot, \cdot) ) \|_{L^{2d+2}(\O_{\e})}=0, \label{4.51} \\
\lim_{n,m\to \infty} \|\p_t\G_n^z(\cdot,\cdot)-\p_t\G_m^z(\cdot,\cdot)\|_{L^{2d+2}(\O_{\e})}=0. \label{4.52}
\end{eqnarray}
\end{prp}

% The proposition is split into three lemmas whose proofs are given in the appendix.

\begin{remark}
From (\ref{3.14}), (\ref{3.13}), (\ref{4.49}), (\ref{4.50}) and (\ref{4.52}), we see that
$$\Gamma^z ( \c , \c )\in W_{2d+2}^{1,2}(((t_0-\eta_2)\vee 0,(t_0+\eta_2)\wedge T_1)\times C(z,\d_5)\bigcap \tilde D). $$
\end{remark}

\section{Construction of test functions}
In this section, we will construct a family of auxiliary functions which will be used to prove
the pathwise uniqueness of the solutions of reflecting stochastic differential equations. Recall that $\t_1$ was defined in Lemma \ref{L3.2}. Let $t_0$, $z_0$, $\d_5$, $\eta_2$ and $z:=y(t_0,z_0,\frac{\d_5}{2})$ be defined as in Section 4.

\begin{lem} \label{L3.7}
Let $u_0 \in C_0^2(B(z,\d_5 \tan \t_1  ) \bigcap H_{t_0,z})$ be nonnegative with $u_0(z)=1$. Define $h(t,x):=u_0(y(t,x,\G^z(t,x)))$. Then \\
$(\mathrm{i})$. $h(t_0,z_0)=1$, \\
$(\mathrm{ii})$. $B(z,\d_5) \bigcap \mathrm{supp} \ h(t,\c) \subset C(z,\d_5)$ for  $t \in ((t_0-\eta_2)\vee 0,(t_0+\eta_2)\wedge T_1)$, \\
$(\mathrm{iii})$. $h$ belongs to the following space:
$$C_b^{0,1}(((t_0-\eta_2)\vee 0,(t_0+\eta_2)\wedge T_1)\times B(z,\d_5)) \bigcap W_{2d+2}^{1,2}(((t_0-\eta_2)\vee 0,(t_0+\eta_2)\wedge T_1)\times B(z,\d_5)\bigcap \tilde D).$$
%\begin{eqnarray} \nonumber
%\begin{split}
%&(\mathrm{i}). \ h(t_0,z_0)=1, \\
%&(\mathrm{ii}). \ h(t,x) \ \mbox{belongs to the following space:} \\
%& C_b^{0,1}(((t_0-\eta_2)\vee 0,(t_0+\eta_2)\wedge T_1)\times B(z,\d_5)) \bigcap W_{2d+2}^{1,2}(((t_0-\eta_2)\vee 0,(t_0+\eta_2)\wedge T_1)\times B(z,\d_5)\bigcap \tilde D),\\
%&(\mathrm{iii}). \ B(z,\d_5) \bigcap \mathrm{supp} \ h(t,\c) \subset C(z,\d_5) \ \mbox{for} \ t \in ((t_0-\eta_2)\vee 0,(t_0+\eta_2)\wedge T_1).
%\end{split}
%\end{eqnarray}
\end{lem}

\begin{proof}
By Lemma \ref{L2.7} , the choice of $u_0$ and the definition of $\G^z(t,x)$, we see that
$$h(t_0,z_0)=h(t_0,z)=u_0(z)=1,$$
and
\begin{eqnarray} \label{2.32}
\begin{split}
B(z,\d_5) \bigcap \mathrm{supp} \ h(t,\c) \subset C(z,\d_5), \ \forall t \in ((t_0-\eta_2)\vee 0,(t_0+\eta_2)\wedge T_1),
\end{split}
\end{eqnarray}
hence (i) and (ii) are proved.

\vskip 0.3cm

By Lemma \ref{L3.3} and the fact that $\g \in C^{0,1}_b([0,{T_1}] \times \Rd)$, we have for $1 \le i \le d$,
 $$h \in C_b^{0,1}(((t_0-\eta_2)\vee 0,(t_0+\eta_2)\wedge T_1)\times B(z,\d_5)),$$
 and
\begin{eqnarray} \nonumber
\begin{split}
\p_{x_i}h(t,x)=\sum_{1 \le j \le d}\p_{y_j}u_0(y(t,x,\G^z(t,x)))[\psi_i^j(t,x,\G^z(t,x))+\g^j(t,y(t,x,\G^z(t,x)))\p_{x_i}\G^z(t,x)].
\end{split}
\end{eqnarray}
From  Proposition \ref{P4.1.1} and (\ref{2.32}), it follows that
 $$h \in W_{2d+2}^{0,2}(((t_0-\eta_2)\vee 0,(t_0+\eta_2)\wedge T_1)\times B(z,\d_5)\bigcap \tilde D).$$
  % So we need only to show $h(t,x)\in W_{2d+2}^{1,0}(((t_0-\eta_2)\vee 0,(t_0+\eta_2)\wedge T_1)\times B(z,\d_5)\bigcap \tilde D)$.

\vskip 0.3cm

Set $h_n(t,x):=u_0(y_n(t,x,\G_n^z(t,x)))$. Then $h_n(t,x)$ converges to $h(t,x)$ uniformly on $((t_0-\eta_2)\vee 0,(t_0+\eta_2)\wedge T_1)\times B(z,\d_5)$ and
\begin{eqnarray} \label{2-19-2}
\begin{split}
\p_{t}h_n(t,x)&=\nabla_y u_0(y_n(t,x,\G_n^z(t,x))) \c \Lambda_n(t,x,\G_n^z(t,x)) \\
&\ \ \ \ + \nabla_y u_0(y_n(t,x,\G_n^z(t,x))) \c \g_n(t,y_n(t,x,\G_n^z(t,x))) \p_t \G_n^z(t,x).\\
\end{split}
\end{eqnarray}
%Noting that $\partial_{t} F_{n}(t, x, r)=\Lambda_{n}(t, x, r)\c \g(t_0, z)$, by (\ref{4.50}) it follows that
% $$\sup_{\e>0}\sup_{n\ge N_1(\e)}\|\partial_{t} F_{n} (t, x, \G_{n}^{z}(t, x) )\|_{L^{2d+2}(\O_{\e})}<\infty. $$
By (\ref{3.14}), (\ref{3.13}) and (\ref{4.50}), we have
\begin{eqnarray} \label{2-19-3}
\begin{split}
\sup_{\e>0}\sup_{n\ge N_1(\e)}\|\p_t\G_n^z\|_{L^{2d+2}(\O_{\e})}<\infty.
\end{split}
\end{eqnarray}
Therefore by (\ref{4.50}), (\ref{2-19-2}) and (\ref{2-19-3}) we have $\sup\limits_{\e>0} \sup\limits_{n\ge N_1(\e)}\|\p_{t}h_n  \|_{L^{2d+2}(\O_{\e})}<\infty$. On the other hand, by  (\ref{4.51}) and (\ref{4.52}),  we have for any $\e>0$, $\lim\limits_{n,m \to \infty}\|\p_{t}h_n-\p_{t}h_m\|_{L^{2d+2}(\O_{\e})}=0$. Hence $h  \in W_{2d+2}^{1,0}(((t_0-\eta_2)\vee 0,(t_0+\eta_2)\wedge T_1)\times B(z,\d_5)\bigcap \tilde D)$. The proof of (iii) is complete by combining the above statements about $h$ together.
\end{proof}

\vskip 0.4cm

Now, we start to construct the first important class of  test functions. The construction is inspired by \cite{Dupuis3} and \cite{Lundstrom}.

\begin{prp} \label{P3.5}
There exists a nonnegative function $H  \in C^{0,1}_b([0,T_1]\times \Rd) \bigcap W^{1,2}_{2d+2}(\tilde D)$ such that for any  $t \in [0,{T_1}]$ and $x \in u(t,\partial D)$
\begin{eqnarray} \label{3.84}
\begin{split}
 \nabla_{x}H(t,x) \c \g(t,x) \ge 1.
\end{split}
\end{eqnarray}

\end{prp}

\begin{proof}
By Lemma \ref{L3.7}, we know that for any given $t_0\in [0,{T_1}]$ and $z_0 \in u(t_0,\partial D)$, there exists a nonnegative function $h(t,x):=u_0(y(t,x,\G^z(t,x)))$ with $h(t_0,z_0)=1$ and $h $ belongs to the following space:
$$C_b^{0,1}(((t_0-\eta_2)\vee 0,(t_0+\eta_2)\wedge T_1)\times B(z,\d_5))  \bigcap W_{2d+2}^{1,2}(((t_0-\eta_2)\vee 0,(t_0+\eta_2)\wedge T_1)\times B(z,\d_5)\bigcap \tilde D),$$
where $\d_5, \eta_2$ are dependent of $(t_0,z_0)$, and $z:=y(t_0,z_0,\frac{\d_5}{2})$.
Using the method of characteristics, we know that $h(t,x)$ is the solution to the following Cauchy problem:
\begin{align} \label{2.36}
 \left\{
\begin{aligned}
& \nabla_x h(t,x) \c \g(t,x)=0, \\
& h(t,\c)|_{H_{t_0, z}}=u_0.
\end{aligned}
 \right.
\end{align}

\vskip 0.3cm By (\ref{2.17}) and (\ref{2-3}), there exists a constant $\kappa \in (\frac{1}{2},1)$ such that
\begin{eqnarray} \label{2-1}
\begin{split}
\frac{ (\kappa^2-4 \tan^2 \t_1 )^{\frac{1}{2}} \cos \t_1  - 2 \tan \t_1   -\frac{1}{2}}{(\frac{5}{4}+ 4  \tan^2 \t_1  +2 \tan \t_1 - (\kappa^2-4 \tan^2 \t_1 )^{\frac{1}{2}} \cos \t_1 )^{\frac{1}{2}} } >\cos \t_0,
\end{split}
\end{eqnarray}
and
\begin{eqnarray} \label{2-2}
\begin{split}
\frac{ (\kappa^2-4 \tan^2 \t_1 )^{\frac{1}{2}}  \cos \t_1  -2 \tan \t_1  +\frac{1}{2} \cos \t_1 }{(\frac{9}{4}+4  \tan^2 \t_1  +2 \tan \t_1 )^{\frac{1}{2}}} >\cos \t_0.
\end{split}
\end{eqnarray}
Now we show that
%$d(z,\bigcup_{c\in \R}B(z-c\g(t_0,z),2\d \tan \t_1 ) \bigcap C(z_0,-n(x_0),\t_0,\d_2))$ for any $x\in \mathrm{supp} u(t,\c)\bigcap B(z,\d) \bigcap C(u^{T_1}(t,x),-n(x),\t_0,\d_2)$
%for any $t\in ((t_0-\eta_2)\vee 0,(t_0+\eta_2)\wedge T_1)$, we have
\begin{eqnarray} \label{2.22}
\begin{split}
u(t_0,\p D)  \bigcap (\bar C(z,\d_5) \backslash B(z,\kappa \d_5))= \emptyset,
\end{split}
\end{eqnarray}
where $C(z,\d_5)$ was defined in (\ref{2.19}) and $\bar C(z,\d_5)$ is the closure of $C(z,\d_5)$.

Let $x\in \bar C(z,\d_5) \backslash B(z,\kappa \d_5)$, $\sigma:=(x-z)\c \g(t_0,z) $ and $\beta:=x-z-\sigma \g(t_0,z) $. Then it is easy to see that % $\beta \c \g(t_0,z)=0$,
\begin{eqnarray} \label{2.21}
\begin{split}
|x-z_0|\le|x-z|+|z-z_0|\le \d_5+\frac{\d_5}{2}< \d_2,
\end{split}
\end{eqnarray}
and
\begin{eqnarray} \label{5.11}
\begin{split}
|x-z_0|\ge|x-z|- |z-z_0|\ge \kappa \d_5-\frac{\d_5}{2}>0.
\end{split}
\end{eqnarray}
Since $x \in \bar C(z,\d_5)$ we have
\begin{eqnarray} \label{2.34}
\begin{split}
|\beta|\le 2 \d_5 \tan \t_1 .
\end{split}
\end{eqnarray}
Hence $\d_5^2\ge |x-z|^2 \ge |\sigma|^2=|x-z|^2-|\beta|^2 \ge \kappa^2 \d_5^2- 4 \d_5^2 \tan ^2\t_1 $, which means that
\begin{eqnarray} \label{2.33}
\begin{split}
\frac{|\sigma|}{\d_5} \in [(\kappa^2-4 \tan^2 \t_1 )^{\frac{1}{2}} ,1].
\end{split}
\end{eqnarray}

If $\sigma \le 0$, then by (\ref{2.34}),
\begin{eqnarray}  \nonumber
\begin{split}
&\ \ \ \ |x-z_0|^2 \\
&=|x-z|^2+|z-z_0|^2+2  \int_0^{\frac{\d_5}{2}}(x-z)\c\g(t_0,y(t_0,z_0,\tau) d\tau\\
&\le \sigma^2+|\beta|^2 +|z-z_0|^2+2  \int_0^{\frac{\d_5}{2}} |\beta| |\g(t_0,y(t_0,z_0,\tau)| d\tau +2\sigma  \int_0^{\frac{\d_5}{2}}\g(t_0,z) \c\g(t_0,y(t_0,z_0,\tau) d\tau\\
&\le \sigma^2+ 4 \d_5^2 \tan^2 \t_1 + \frac{\d_5^2}{4}+ 2 \d_5^2 \tan \t_1  - |\sigma|  \d_5 \cos \t_1 , %=\d^2(\frac{5}{4}+2 \tan \t_1 -(\frac{9}{16}-4 \tan^2 \t_1 )^{\frac{1}{2}} \cos \t_1 ).
\end{split}
\end{eqnarray}
 which implies that
\begin{eqnarray} \label{2-19-4}  \nonumber
\begin{split}
\frac{ |x-z_0|^2}{\d_5^2}  \le \frac{\sigma^2}{\d_5^2} + 4 \tan^2 \t_1  +\frac{1}{4}+ 2   \tan \t_1 -  \frac{|\sigma| }{\d_5} \cos \t_1 . \\
% &\le 1 -   \cos \t_1 + 4 \tan^2 \t_1  +\frac{1}{4}+ 2   \tan \t_1 . \\
\end{split}
\end{eqnarray}
%Note that $ \d_5 /\sigma \in (1,(\kappa^2-4 \tan^2 \t_1 )^{-1/2})$ by (\ref{2.33}) and $(\kappa^2-4 \tan^2 \t_1 )^{-1/2}<2$ by (\ref{2-1}).
Combining this with (\ref{2-1}), (\ref{2.34}) and (\ref{2.33}), we have
\begin{eqnarray} \label{2.23}
\begin{split}
&\ \ \ \ (x-z_0)\c (-\g(t_0,z_0))\\
&= -\sigma \g(t_0,z) \c \g(t_0,z_0)-\beta \c \g(t_0,z_0)-(z-z_0) \c \g(t_0,z_0)\\
&\ge |\sigma| \g(t_0,z) \c \g(t_0,z_0)  -|\beta|-|z-z_0| \\
&\ge (\frac{|\sigma| }{\d_5}  \cos \t_1  - 2  \tan \t_1  -\frac{1}{2})\d_5 \\
&\ge \frac{\frac{|\sigma| }{\d_5}  \cos \t_1  - 2  \tan \t_1  -\frac{1}{2}}{(\frac{\sigma^2}{\d_5^2} + 4 \tan^2 \t_1  +\frac{1}{4}+ 2   \tan \t_1 -  \frac{|\sigma| }{\d_5} \cos \t_1 )^{\frac{1}{2}} } |x-z_0|\\
&\ge \frac{ (\kappa^2-4 \tan^2 \t_1 )^{\frac{1}{2}} \cos \t_1  - 2 \tan \t_1   -\frac{1}{2}}{(\frac{5}{4}+ 4  \tan^2 \t_1  +2 \tan \t_1 - (\kappa^2-4 \tan^2 \t_1 )^{\frac{1}{2}} \cos \t_1 )^{\frac{1}{2}} } |x-z_0| \\
& >|x-z_0| \cos \t_0.
\end{split}
\end{eqnarray}
(\ref{2.21}), (\ref{5.11}) and (\ref{2.23}) show that  $x \in C(z_0,-\g(t_0,z_0),\d_2,\t_0)\subset u(t_0, \bar D)^c$, which in particular implies $x \not\in u(t_0,\p D)$.

If $\sigma > 0$, then by (\ref{2.34})
\begin{eqnarray}  \nonumber
\begin{split}
&\ \ \ \ |x-z_0|^2\\
&=|x-z|^2+|z-z_0|^2+2  \int_0^{\frac{\d_5}{2}}(x-z)\c\g(t_0,y(t_0,z_0,\tau) d\tau\\
&\le \sigma^2+|\beta|^2+|z-z_0|^2+2  \int_0^{\frac{\d_5}{2}} |\beta| |\g(t_0,y(t_0,z_0,\tau)| d\tau + 2\sigma  \int_0^{\frac{\d_5}{2}}\g(t_0,z) \c\g(t_0,y(t_0,z_0,\tau) d\tau\\
&\le \sigma^2+ 4 \d_5^2 \tan^2 \t_1 + \frac{\d_5^2}{4}+ 2 \d_5^2 \tan \t_1 +  \sigma \d_5, %=\d_5^2(\frac{5}{4}+2 \tan \t_1 -(\frac{9}{16}-4 \tan^2 \t_1 )^{\frac{1}{2}} \cos \t_1 ).
\end{split}
\end{eqnarray}
 which implies that
\begin{eqnarray}  \nonumber
\begin{split}
\frac{ |x-z_0|^2}{\d_5^2}  \le \frac{\sigma^2}{\d_5^2} + 4 \tan^2 \t_1  +\frac{1}{4}+ 2   \tan \t_1 +  \frac{\sigma}{\d_5}  . \\
% &\le 1 -   \cos \t_1 + 4 \tan^2 \t_1  +\frac{1}{4}+ 2   \tan \t_1 . \\
\end{split}
\end{eqnarray}
Combining this with (\ref{2-2}), (\ref{2.34}) and (\ref{2.33}), we have
\begin{eqnarray}  \nonumber
\begin{split}
&\ \ \ \ (x-z_0)\c \g(t_0,z_0)\\
&= \sigma \g(t_0,z) \c \g(t_0,z_0)+\beta \c \g(t_0,z_0)+(z-z_0) \c \g(t_0,z_0)\\
&\ge \sigma \g(t_0,z) \c \g(t_0,z_0)  -|\beta|+\int_0^{\frac{\d_5}{2}} \g(t_0,y(t_0,z_0,\tau))\c \g(t_0,z_0) d \tau\\
&\ge ( \frac{\sigma}{\d_5} \cos \t_1  - 2  \tan \t_1  +\frac{1}{2} \cos \t_1 )\d_5 \\
&\ge \frac{ \frac{\sigma}{\d_5}  \cos \t_1  -2 \tan \t_1  +\frac{1}{2} \cos \t_1  }{(\frac{\sigma^2}{\d_5^2} + 4 \tan^2 \t_1  +\frac{1}{4}+ 2   \tan \t_1 +  \frac{\sigma}{\d_5}  )^{\frac{1}{2}} } |x-z_0|\\
& \ge \frac{ (\kappa^2-4 \tan^2 \t_1 )^{\frac{1}{2}}  \cos \t_1  -2 \tan \t_1  +\frac{1}{2} \cos \t_1 }{(\frac{9}{4}+4  \tan^2 \t_1  +2 \tan \t_1 )^{\frac{1}{2}}}|x-z_0|\\
& >|x-z_0| \cos \t_0.
\end{split}
\end{eqnarray}
Together with (\ref{2.21}) and (\ref{5.11}) we see that  $x \in C(z_0, \g(t_0,z_0),\d_2,\t_0)\subset u(t_0,  D) $, which  again implies $x \not\in u(t_0,\p D)$. Hence we obtain (\ref{2.22}).

\vskip 0.3cm

By (\ref{3.1}) and (\ref{2.22}), there exists a $\eta_3\in (0,\eta_2)$ such that for
$t\in ((t_0-\eta_3)\vee 0,(t_0+\eta_3)\wedge T_1)$, we have $u(t,\p D)  \bigcap (\bar C(z,\d_5) \backslash B(z,\kappa \d_5))= \emptyset$. Together with Lemma \ref{L3.7} we obtain that
\begin{eqnarray} \label{2.35}
\begin{split}
u(t,\p D) \bigcap B(z,\d_5) \bigcap (\mathrm{supp} \ h(t,\c) \backslash B(z,\kappa \d_5))= \emptyset. \\
\end{split}
\end{eqnarray}

Now take a nonnegative function $\chi_1 \in C_0^{1,2}(((t_0-\eta_3)\vee 0,(t_0+\eta_3)\wedge T_1) \times B(z,\d_5))$ such that $\chi_1 =1$ on $((t_0-\frac{\eta_3}{2})\vee 0,(t_0+\frac{\eta_3}{2})\wedge T_1) \times B(z,\kappa \d_5)$, and choose a constant $M>0$ large enough, such that $\chi_2(x):=(x-z_0) \c \g(t_0,z_0)+M$ is nonnegative on $B(z,\d_5)$. Set $h_{t_0,z_0}(t,x):=h(t,x) \chi_1(t,x) \chi_2(x)$. We see that $h_{t_0,z_0}$ belongs to the following space:
 $$ C_0^{0,1}(((t_0-\eta_3)\vee 0,(t_0+\eta_3)\wedge T_1)\times B(z,\d_5))  \bigcap W_{2d+2}^{1,2}(((t_0-\eta_3)\vee 0,(t_0+\eta_3)\wedge T_1)\times B(z,\d_5)\bigcap \tilde D).$$ Note that by (\ref{2.35}), $\chi_1=1$ on the neighborhood  of
 $$\{(t,x):t \in ((t_0-\frac{\eta_3}{2})\vee 0,(t_0+\frac{\eta_3}{2})\wedge T_1), \ x \in u(t, \p D) \cap B(z,\d_5)\cap \mathrm{supp} \ h(t,\c)\}.$$
 Using the above fact, Lemma \ref{L3.2} and (\ref{2.36}) we have
\begin{eqnarray}  \nonumber
\begin{split}
\nabla_x h_{t_0,z_0}(t_0,z_0) \c \g(t_0,z_0)&=\chi_2(z_0) \nabla_x h(t_0,z_0) \c \g(t_0,z_0)+h(t_0,z_0) \nabla_x \chi_2(z_0) \c \g(t_0,z_0)\\
&=1,
\end{split}
\end{eqnarray}
and
\begin{eqnarray}  \nonumber
\begin{split}
\nabla_x h_{t_0,z_0}(t,x) \c \g(t,x)&=\chi_2(x) \nabla_x h(t,x) \c \g(t,x)+h(t,x) \nabla_x \chi_2(x) \c \g(t,x)\\
&\ge 0,
\end{split}
\end{eqnarray}
for $t \in ((t_0-\frac{\eta_3}{2})\vee 0,(t_0+\frac{\eta_3}{2})\wedge T_1)$ and $x \in u(t, \p D) \cap B(z,\d_5)$. Now, by a standard compactness argument we can construct  a nonnegative function $H   \in C^{0,1}_b([0,T_1]\times \Rd) \bigcap W^{1,2}_{2d+2}(\tilde D)$ such that (\ref{3.84}) holds for any  $t \in [0,{T_1}]$ and $x \in u(t,\partial D)$.
%We refer the reader to \cite{Lundstrom} for details.
\end{proof}

\vskip 0.4cm

Following exactly the argument of Lemma 4.4 in \cite{Dupuis1}, we have the next result.

\begin{lem}
% For any $\iota \in(0,1]$,

There exist a function $g\in C^{1}(\R^{2d})\bigcap C^{2}((\Rd \setminus \{0\}) \times \Rd)$ and positive constants $M_4,M_5$, satisfying that for  any $\rho,\xi \in \Rd$ with $|\xi|  \le 1$ the following conditions hold:
\begin{eqnarray}
& (\mathrm{i}).&   g(0,\xi)=0, \label{5.15}\\
& (\mathrm{ii}).&   g(\rho,\xi)\ge M_4 |\rho|^{2}, \label{2.37} \\
& (\mathrm{iii}).&    \nabla_{\rho}g(\rho,\xi)\c \xi \ge 0,  \  \mbox{for} \ \rho \c \xi \ge -\cos \t_0 |\rho|  \  \mbox{and} \   |\xi|=1,  \label{3.40}\\
& (\mathrm{iv}). &   \nabla_{\rho}g(\rho,\xi)\c \xi \le 0, \  \mbox{for} \ \rho \c \xi \le \cos \t_0 |\rho| \  \mbox{and} \  |\xi|=1,  \label{3.55}\\
 & (\mathrm{v}).&   |\nabla_{\rho}g(\rho,\xi)| \le  M_5 |\rho|,  \   |\nabla_{\xi}g(\rho,\xi)| \le  M_5 |\rho|^{2}, \ \mbox{for} \ |\rho| \neq 0 ,  \label{3.56}\\
& (\mathrm{vi}). &    |\p_{\rho_i}\p_{\rho_j} g(\rho,\xi) | \le M_5, \  |\p_{\xi_i} \p_{\rho_j}g(\rho,\xi) | \le M_5 |\rho|, \nonumber \\
& &   | \p_{\xi_i}\p_{\xi_j}g(\rho,\xi) |  \le M_5|\rho|^{2}, \  \mbox{for} \ |\rho| \neq 0, \label{3.58}
\end{eqnarray}
where $\t_0$ was defined in Proposition \ref{P2.2}.
\end{lem}

\vskip 0.4cm

Take $\sigma \in C^{2}(\R)$ such that $\sigma(t)=1$ for $t \le \frac{1}{2}$, $\sigma(t)=t$ for $t \ge 2$ and $\sigma'(t)\ge 0, \sigma(t)\ge t$ for $t \in \R$. It is easy to see that $\omega(\rho,\xi):=\sigma(g(\rho,\xi)) \in C^2(\R^{2d})$. Now we introduce the second important class of  test functions. For $\varepsilon>0$,  define
\begin{eqnarray} \label{5.20}
\begin{split}
f_{\e}(t,x,y):=\e \o(\frac{u(t,x)-u(t,y)}{\e},n(x)).
\end{split}
\end{eqnarray}
The following result holds.

\begin{prp}
There exist positive constants $M_6$ and $M_7$, which are independent of $\e$, such that  for $t\in [0,T_1]$ and $\rho,\xi \in \Rd$ with $|\xi|\le 1$,
\begin{eqnarray}
& (\mathrm{i}).&|\nabla_{\rho}\o(\rho,\xi)|\le M_6|\rho|, \ |\nabla_{\xi}\o(\rho,\xi)|\le M_6|\rho|^2,  \label{3.61}\\
& (\mathrm{ii}).&|\p_{\rho_i} \p_{\rho_j}\o(\rho,\xi)|\le M_6, \ |\p_{\xi_i} \p_{\rho_j}\o(\rho,\xi)|\le M_6|\rho|, |\p_{\xi_i} \p_{\xi_j} \o(\rho,\xi)|\le M_6|\rho|^2, \label{3.38}\\
& (\mathrm{iii}).&M_7\frac{|x-y|^2}{\e}\le f_{\e}(t,x,y) \le \e+\frac{M_6|x-y|^2}{\e} ,\  \mbox{for} \  x,y \in \bar{D} ,  \label{3.36}\\
& (\mathrm{iv}). &    \nabla_x f_{\e}(t,x,y) \c n(x)\le M_6 \frac{|x-y|^2}{\e}, \  \mbox{for} \  x\in \p D \  \mbox{and} \  y\in \bar{D}, \label{3.53} \\
& (\mathrm{v}). &    \nabla_y f_{\e}(t,x,y) \c n(y)\le M_6 \frac{|x-y|^2}{\e}, \  \mbox{for} \  x\in  \bar{D} \  \mbox{and} \  y\in \p D. \label{3.54}
\end{eqnarray}

\end{prp}

\begin{proof}
Let $\rho,\xi \in \Rd$ with $|\xi|\le 1$, then by (\ref{3.56}) and the boundedness  of $\sigma'(t)$,
it is easy to see that $|\nabla_{\rho}\o(\rho,\xi)|\le M_6|\rho|$ and $|\nabla_{\xi} \o(\rho,\xi)|\le M_6|\rho|^2$
for some positive constant $M_6$.
%\begin{eqnarray}
%\begin{split}
%|\p_{\xi_i} \o(\rho,\xi)|= |\sigma'(g(\rho,\xi))\p_{\xi_i}g(\rho,\xi )|\le M_6|\rho|^2,\\
%\end{split}
%\end{eqnarray}
%and
%\begin{eqnarray}
%\begin{split}
%|D_{\rho}\o(\rho,\xi)|\le M_6|\rho|.
%\end{split}
%\end{eqnarray}
\vskip 0.3cm
By (\ref{2.37}), (\ref{3.56}) and (\ref{3.58}), we have
\begin{eqnarray}  \nonumber
\begin{split}
|\p_{\rho_i} \p_{\rho_j}\o(\rho,\xi)|&=|\sigma''(g(\rho,\xi))\p_{\rho_i}g(\rho.\xi )\p_{\rho_j}g(\rho,\xi )+\sigma'(g(\rho,\xi ))\p_{\rho_i}\p_{\rho_j}g(\rho,\xi)|\\
&\les |\rho|^2I_{g(\rho,\xi)\le 2}(\rho,\xi)+1\\
&\le |\rho|^2I_{M_4|\rho|^2 \le 2}(\rho,\xi)+1\\
&\le M_6,
\end{split}
\end{eqnarray}
Similarly, we also have $|\p_{\xi_i} \p_{\rho_j}\o(\rho,\xi)| \le M_6 |\rho|$ and $|\p_{\xi_i} \p_{\xi_j} \o(\rho,\xi)|\le M_6|\rho|^2$.
%\begin{eqnarray} \nonumber
%\begin{split}
%|\p_{\xi_i} \p_{\rho_j}\o(\rho,\xi)|&=|\sigma''(g(\rho,\xi))\p_{\xi_i}g(\rho.\xi )\p_{\rho_j}g(\rho,\xi )+\sigma'(g(\rho,\xi ))\p_{\xi_i}\p_{\rho_j}g(\rho,\xi)|\\
%&\les |\rho|^3I_{g(\rho,\xi)\le 2}(\rho,\xi)+|\rho|\\
%&\le |\rho|^3I_{M_4|\rho|^2 \le 2}(\rho,\xi)+|\rho|\\
%&\le M_6 |\rho|,\\
%\end{split}
%\end{eqnarray}

\vskip 0.3cm

Now we show (\ref{3.36}). Let $x,y \in \bar{D}$.
By (\ref{3.2}), (\ref{2.37}) and the fact that $\sigma(t)\ge t$, we have
\begin{eqnarray} \nonumber
\begin{split}
f_{\e}(t,x,y)\ge \e (g(\frac{u(t,x)-u(t,y)}{\e},n(x)))\ge \e M_4  |\frac{u(t,x)-u(t,y)}{\e}|^2 \ge M_7\frac{|x-y|^2}{\e},
\end{split}
\end{eqnarray}
for some constant $M_7>0$. By (\ref{3.2}), (\ref{5.15}) and (\ref{3.56}) we have
% (The other part follows from Remark 3.3 in \cite{Dupuis3}.)
\begin{eqnarray} \nonumber
\begin{split}
&\ \ \ \ f_{\e}(t,x,y)\\
&= \e \sigma((g(0,n(x))))\\
&\ \ \ \ +  \e \int_0^1\sigma'(g(\frac{\lambda (u(t,x)-u(t,y))}{\e},n(x)))\nabla_{\rho}g(\frac{\lambda (u(t,x)-u(t,y))}{\e},n(x))\c \frac{u(t,x)-u(t,y)}{\e}d\lambda \\
&\le \e+\e \int_0^1 c |\frac{\lambda (u(t,x)-u(t,y))}{\e}| |\frac{u(t,x)-u(t,y)}{\e}|d\lambda  \le \e+\frac{M_6|x-y|^2}{\e}.\\
\end{split}
\end{eqnarray}

\vskip 0.3cm

Next we  show (\ref{3.53}). When $x\in \p D$ and $y \in \bar D$ satisfying that $|x-y|< \frac{\d_2}{M_2}$, by (\ref{3.2}) we have $|u(t,x)-u(t,y)|< \d_2$. Combining this with (\ref{3.59}) we deduce that
$$\frac{u(t,x)-u(t,y)}{\e} \c n(x) \le \cos \t_0 |\frac{u(t,x)-u(t,y)}{\e}|.$$
In view of  (\ref{3.2}), (\ref{3.55}), (\ref{3.56}), taking into consideration of  the facts $\nabla_x u^i(t,x)\c n(x)=n^i(x)$ and $\sigma'(t)\ge 0$, we have
\begin{eqnarray} \nonumber
\begin{split}
&\ \ \ \ \nabla_x f_{\e}(t,x,y) \c n(x)\\
&= \sigma'(g(\frac{u(t,x)-u(t,y)}{\e},n(x)))\sum_{1 \le i \le d}[\p_{\rho_i}g(\frac{u(t,x)-u(t,y)}{\e},n(x)) \nabla_x u^i(t,x)\c n(x)\\
&\ \ \ \ + \e \p_{\xi_i}g(\frac{u(t,x)-u(t,y)}{\e},n(x)) \nabla_x n^i(x)\c n(x)] \\
&\le \sigma'(g(\frac{u(t,x)-u(t,y)}{\e},n(x))) \nabla_{\rho}g(\frac{u(t,x)-u(t,y)}{\e},n(x)) \c  n(x)  \\
&\ \ \ \ + \sup_{t \in \R} \sigma'(t) \e \sum_{1 \le i \le d} |\p_{\xi_i}g(\frac{u(t,x)-u(t,y)}{\e},n(x)) \nabla_x n^i(x)\c n(x)| \\
&\le \sup_{t \in \R} \sigma'(t) \e M_5 |\frac{u(t,x)-u(t,y)}{\e}|^2 \sum_{1 \le i \le d}|\nabla_x n^i(x)| \le M_6 \frac{|x-y|^2}{\e}. \\
\end{split}
\end{eqnarray}

When $x\in \p D$ and $y \in \bar D$ satisfying that $|x-y| \! \ge  \! \frac{\d_2}{M2}$, noting that  $\nabla_x u^i(t,x)\c n(x)=n^i(x)$, by (\ref{3.2}) and (\ref{3.56}), we have
\begin{eqnarray} \nonumber
\begin{split}
&\ \ \ \ \nabla_x f_{\e}(t,x,y) \c n(x)\\
&= \sigma'(g(\frac{u(t,x)-u(t,y)}{\e},n(x)))\sum_{1 \le i \le d}[\p_{\rho_i}g(\frac{u(t,x)-u(t,y)}{\e},n(x)) \nabla_x u^i(t,x)\c n(x)\\
&\ \ \ \ + \e \p_{\xi_i}g(\frac{u(t,x)-u(t,y)}{\e},n(x)) \nabla_x n^i(x)\c n(x)] \\
&\le \sup_{t \in \R} \sigma'(t)[ |\nabla_{\rho}g(\frac{u(t,x)-u(t,y)}{\e},n(x))| +   \e \sum_{1 \le i \le d} | \p_{\xi_i}g(\frac{u(t,x)-u(t,y)}{\e},n(x))| | \nabla_x n^i(x)|] \\
&\les |\frac{u(t,x)-u(t,y)}{\e}| +  \e  |\frac{u(t,x)-u(t,y)}{\e}|^2 \\
&\le  M_2 \frac{|x-y|}{\e}+ M_2^2 \frac{|x-y|^2}{\e}\\
&\le \frac{M_2^2}{\d_2}  \frac{|x-y|^2}{\e} + M_2^2 \frac{|x-y|^2}{\e} \le M_6 \frac{|x-y|^2}{\e}. \\
\end{split}
\end{eqnarray}

\vskip 0.3cm

Finally we prove (\ref{3.54}).
When $x\in \bar D$ and $y \in \p D$ satisfying that $|x-y|< \frac{\d_2}{M_2}$, by (\ref{3.2}) we have $|u(t,x)-u(t,y)|< \d_2$. Combining this with (\ref{3.59}) gives
$$\frac{u(t,x)-u(t,y)}{\e} \c n(y) \ge -\cos \t_0 |\frac{u(t,x)-u(t,y)}{\e}|.$$
 By (\ref{3.2}), (\ref{3.40}), (\ref{3.58}), and the facts that $\nabla_x u^i(t,y)\c n(y)=n^i(y)$ and $\sigma'(t)\ge 0$, we have
\begin{eqnarray} \nonumber
\begin{split}
&\ \ \ \ \nabla_y f_{\e}(t,x,y) \c n(y)\\
&= -\sigma'(g(\frac{u(t,x)-u(t,y)}{\e},n(x)))\sum_{1 \le i \le d}\p_{\rho_i}g(\frac{u(t,x)-u(t,y)}{\e},n(x)) \nabla_x u^i(t,y)\c n(y) \\
&= \sigma'(g(\frac{u(t,x) \! - \! u(t,y)}{\e},n(x))) (\nabla_{\rho}g(\frac{u(t,x) \! - \! u(t,y)}{\e},n(y)) \! - \! \nabla_{\rho}g(\frac{u(t,x) \! - \! u(t,y)}{\e},n(x))) \c n(y) \\
&\ \ \ \ - \sigma'(g(\frac{u(t,x)-u(t,y)}{\e},n(x))) \nabla_{\rho}g(\frac{u(t,x)-u(t,y)}{\e},n(y)) \c n(y)\\
&\le \sup_{t \in \R} \sigma'(t) |\nabla_{\rho}g(\frac{u(t,x)-u(t,y)}{\e},n(y))-\nabla_{\rho}g(\frac{u(t,x)-u(t,y)}{\e},n(x))| \le M_6 \frac{|x-y|^2}{\e}.
\end{split}
\end{eqnarray}

When $x\in \bar D$ and $y \in \p D$ satisfying that $|x-y|\ge \frac{\d_2}{M2}$, by (\ref{3.2}), (\ref{3.56}) and using the fact that $\nabla_x u^i(t,y)\c n(y)=n^i(y)$ we have
\begin{eqnarray} \nonumber
\begin{split}
&\ \ \ \ \nabla_y f_{\e}(t,x,y) \c n(y)\\
&= -\sigma'(g(\frac{u(t,x)-u(t,y)}{\e},n(x)))\sum_{1 \le i \le d}\p_{\rho_i}g(\frac{u(t,x)-u(t,y)}{\e},n(x)) \nabla_x u^i(t,y)\c n(y) \\
&\le \sup_{t \in \R} \sigma'(t) |\nabla_{\rho}g(\frac{u(t,x)-u(t,y)}{\e},n(x))| \\
& \les \frac{|x-y|}{\e}\le \frac{M_2}{\d_2}\frac{|x-y|^2}{\e}.
\end{split}
\end{eqnarray}
\end{proof}

\vskip 0.4cm

Now we recall the stochastic Gronwall's inequality. See Theorem 4 in \cite{Scheutzow} or Lemma 2.8 in \cite{ZhangX}.

\begin{lem}\label{L3.11}
Let $\xi_t$ and $\eta_t$ be two nonnegative c\`{a}dl\`{a}g adapted processes, $A_t$ a continuous nondecreasing adapted process
with $A_0 = 0$, $M_t$ a local martingale with $M_0 = 0$. Suppose that
$$\xi_t \le \eta_t +\int_0^t\xi_sdA_s+M_t,\ \forall \ t>0.$$
Then for any $0 < q < p < 1$ and stopping time $\tau > 0$, we have
$$[E(\xi^*_{\tau})^q]^{1/q}\le (\frac{p}{p-q})^{1/q}(Ee^{pA_{\tau}/(1-p)})^{(1-p)/p}E(\eta^*_{\tau}),$$
where $\xi^*_{\tau}:=\sup_{s\in [0,\tau]}\xi_s$ and $\eta^*_{\tau}:=\sup_{s\in [0,\tau]}\eta_s$.

\end{lem}

\vskip 0.4cm

Using the Krylov's estimate established in Lemma 5.1 in \cite{Krylov}, and following the same arguments as in the proof of Lemma 4.1 in \cite{Gyongy}, we have the following estimates:

\begin{lem} \label{L3.12}
Assume $(X_t,L_t)$ and $(\tilde X_t,\tilde L_t)$ are solutions to the reflecting
SDEs (\ref{1.1}) with $E[|L|_T]<\infty$ and $E[|\tilde L|_T]<\infty$ for  $T>0$. %and $b\in L^{d+1}([0,T]\times D)$ for $T>0$.
Then there exists a positive constant $M_8$ depending only on $T$, $E[|L|_T]$, $E[|\tilde L|_T]$ and $\|b \|_{L^{d+1}((0,T)\times D)}$, such that for any $f \in L^{d+1}((0,T)\times D)$,
\begin{eqnarray} \label{3.85}
\begin{split}
E[\int_0^T|f(t,  X_t)|dt]\le M_8 \|f\|_{L^{d+1}((0,T) \times D)}.
\end{split}
\end{eqnarray}
 Moreover, there exists a positive constant $M_9$ depending only on $T$, $E[|L|_T]$, $E[|\tilde L|_T]$ and $\|b \|_{L^{d+1}((0,T)\times D)}$, such that for any $f \in L^{d+1}((0,T)\times \Rd)$ and $\a \in [0,1]$,
\begin{eqnarray} \label{3.86}
\begin{split}
E[\int_0^T|f(t,\a X_t+(1-\a)\tilde X_t)|dt]\le M_9 \|f\|_{L^{d+1}((0,T)\times \Rd)}.
\end{split}
\end{eqnarray}
\end{lem}

\section{Existence and uniqueness }

In this section, we will establish the existence and uniqueness of strong solutions to the reflecting SDEs (\ref{1.1}) with singular coefficients. The existence of a weak solution follows from the Girsanov theorem. The strong solution is obtained by proving the pathwise uniqueness of the solutions.
\vskip 0.3cm

When the drift $b$ vanishes, the solution of equation (\ref{1.1}) is the so called reflecting Brownian motion. The existence and uniqueness of  reflecting Brownian motion $(X_t,L_t)$ is now well known (see e.g. \cite{Hsu}). Then using the Girsanov transformation, we easily obtain the following result.

\begin{prp} \label{P4.1}
For any $x\in \bar D$, there exists a unique weak solution $(X_t,L_t)$ to the reflecting SDEs (\ref{1.1}) with $X_0=x$, Moreover, $E_x[|L|_T]<\infty$.

\end{prp}

\vskip 0.4cm
To obtain the existence and uniqueness of strong solutions of the reflecting SDEs, according to the Yamada-Watanabe theorem it is sufficient to prove the pathwise uniqueness of the equation (\ref{1.1}). The rest of this section is devoted to this goal.
\vskip 0.4cm
Using the Krylov's estimate in Lemma \ref{L3.12}, we have the following generalized It\^o's formula:

\begin{lem} \label{L4.1}
Let $F  \in W^{1,2}_{q}((0,T)\times D)$ for some $T>0$ and $q>d+2$. Let $X_t$ be a solution to the reflecting SDEs (\ref{1.1}). Then we have for any $0 \le t\le T$,
\begin{eqnarray} \label{4.2}
\begin{split}
& \ \ \ \ F (t,X_t)\\
&=F (0,X_0)+\int_0^t\p_t F (s,X_s)ds+\sum_{1 \le i \le d}\int_0^t\p_{x_i} F (s,X_s)dW^i_s\\
& \ \ \ \ +\int_0^t \nabla_{x} F (s,X_s) \c b(s,X_s)ds+\int_0^t \nabla_{x} F (s,X_s)\c n(X_s)d|L|_s+\frac{1}{2}\ \int_0^t \Delta_{x}  F (s,X_s )ds.
\end{split}
\end{eqnarray}
\end{lem}

\begin{proof}
Since $F \in W^{1,2}_{q}((0,T)\times D)$ and the boundary of $D$ is smooth, there exists a sequence of functions $\{F_n \}_{n\ge 1}  \subset   C_b^{1,2}  ([0,T]\times \Rd)$ such that $F_n $ converges to $F $ in $W^{1,2}_{q}  ((0,T)\times D)$. By the It\^o's formula, we have
\begin{eqnarray}\label{4.3}
\begin{split}
& \ \ \ \ F_n(t,X_t)\\
&=F_n(0,X_0)+\int_0^t\p_t F_n (s,X_s)ds+\sum_{1 \le i \le d}\int_0^t\p_{x_i} F_n (s,X_s)dW^i_s\\
& \ \ \ \ +\int_0^t \nabla_{x} F_n (s,X_s) \c b(s,X_s)ds+\int_0^t \nabla_{x} F_n (s,X_s)\c n(X_s)d|L|_s+\frac{1}{2} \int_0^t  \Delta_{x}  F_n (s,X_s )ds.\\
\end{split}
\end{eqnarray}
Note that $F_n(t,x)$ and $ \nabla_{x} F_n(t,x)$ converges to $F(t,x)$ and $ \nabla_{x} F(t,x)$ uniformly on $[0,T]\times \bar D$ respectively by Sobolev inequality. Combining this with Lemma \ref{L3.12}, letting $n\to \infty$ in (\ref{4.3}), we get (\ref{4.2}).
\end{proof}

\vskip 0.4cm

Recall that the constant $T_1$ was defined in Proposition \ref{P2.2}, and the functions $H$ and $f_{\e}$ were defined in Proposition \ref{P3.5} and (\ref{5.20}) respectively.
For $0\le t \le T_1$, set
\begin{eqnarray} \nonumber
\begin{split}
F_{\e}(t,x,y):=Z_t f_{\e}(t,x,y):=e^{-\lambda[H(t,u(t,x))+H(t,u(t,y))]}f_{\e}(t,x,y),
\end{split}
\end{eqnarray}
where $\e$ and $\l$ are some positive constants, and
%Then $F_{\e}(t,x,y)\in W_{2d+2}^{1,2,2}((0,T_1)\times D\times D)$  $F_{\e}(t,x,y)$.
%
%Set
\begin{eqnarray} \nonumber
\begin{split}
& \ \ \ \ \ M_t\\
&:=- \l \int_0^t Z_s f_{\e}(s,X_s,\tilde X_s) [\sum_{1 \le i \le d}(\p_{x_i} H)(s,u(s,X_s))  \nabla_{x} u^i(s,X_s) \c dW_s\\
&\ \ \ \ \ \ \ \ + \sum_{1 \le i \le d}(\p_{x_i} H)(s,u(s,\tilde X_s))  \nabla_{x} u^i(s,\tilde X_s) \c dW_s ]\\
& \ \ \ \ + \int_0^t Z_s \sum_{1 \le i \le d}\p_{\rho_i}\o (\frac{u(s,X_s)-u(s,\tilde X_s)}{\e},n(X_s))(\nabla_{x}u^i(s,X_s)-\nabla_{x}u^i(s,\tilde X_s)) \c dW_s \\
& \ \ \ \ + \e \int_0^t Z_s \sum_{1 \le i \le d}\p_{\xi_i}\o (\frac{u(s,X_s)-u(s,\tilde X_s)}{\e},n(X_s))\nabla_{x}n^i(X_s) \c dW_s,
\end{split}
\end{eqnarray}
\begin{eqnarray} \nonumber
\begin{split}
& \ \ \ \ \ A^1_t\\
&:=- \l \int_0^t Z_s f_{\e}(s,X_s,\tilde X_s) [(\nabla_{x} H)(s,u(s,X_s)) \c n(X_s)d|L|_s+(\nabla_{x} H)(s,u(s,\tilde X_s))\c n(\tilde X_s)d|\tilde L|_s ]\\
& \ \ \ \ +  \int_0^t Z_s \nabla_{\rho}\o (\frac{u(s,X_s)-u(s,\tilde X_s)}{\e},n(X_s))\c (n(X_s)d|L|_s-n(\tilde X_s)d|\tilde L|_s) \\
& \ \ \ \ + \e \int_0^t Z_s \sum_{1 \le i \le d}\p_{\xi_i}\o (\frac{u(s,X_s)-u(s,\tilde X_s)}{\e},n(X_s)) \nabla_x n^i(X_s)\c n(X_s)d|L|_s,
\end{split}
\end{eqnarray}
\begin{eqnarray} \nonumber
\begin{split}
& \ \ \ \ \ A^2_t\\
&:=- \l \int_0^t Z_s f_{\e}(s,X_s,\tilde X_s) [(\p_s H)(s,u(s,X_s))ds+(\p_s H)(s,u(s,\tilde X_s))ds]\\
& \ \ \ \ - \frac{\l }{2} \int_0^t Z_s f_{\e}(s,X_s,\tilde X_s)\sum_{1 \le i,j \le d} (\p_{x_j} \p_{x_i} H)(s,u(s,X_s))  \nabla_{x}u^i(s,X_s)\c \nabla_{x}u^j(s,X_s)ds\\
& \ \ \ \ - \frac{\l }{2} \int_0^t Z_s f_{\e}(s,X_s,\tilde X_s)\sum_{1 \le i,j \le d} (\p_{x_j} \p_{x_i} H)(s,u(s,\tilde X_s))  \nabla_{x}u^i(s,\tilde X_s)\c \nabla_{x}u^j(s,\tilde X_s)ds\\
%& \ \ \ \ + \frac{\l^2}{2} \int_0^t Z_s f_{\e}(s,X_s,\tilde X_s)  \sum_{1 \le i,j \le d} (\p_{x_j} \p_{x_i} H)(s,u(s,X_s))  \nabla_{x}u^i(s,X_s)\c \nabla_{x}u^j(s,X_s)ds\\
%& \ \ \ \ + \frac{\l^2}{2} \int_0^t Z_s f_{\e}(s,X_s,\tilde X_s)  \sum_{1 \le i,j \le d} (\p_{x_j} \p_{x_i} H)(s,u(s,\tilde X_s))  \nabla_{x}u^i(s,\tilde X_s)\c \nabla_{x}u^j(s,\tilde X_s)ds\\
& \ \ \ \ + \e \int_0^t Z_s \sum_{1 \le i \le d}\p_{\xi_i}\o (\frac{u(s,X_s)-u(s,\tilde X_s)}{\e},n(X_s))[\nabla_x n^i(X_s) \c b(s,X_s)+\frac{1}{2}\Delta_x n^i(X_s)]ds\\
& \ \ \ \ + \frac{1}{2 \e} \int_0^t Z_s  \sum_{1 \le i,j \le d}\p_{\rho_j}\p_{\rho_i}\o (\frac{u(s,X_s)-u(s,\tilde X_s)}{\e},n(X_s)) \\
& \ \ \ \  \ \ \ \ \times (\nabla_x u^i(s,X_s)-\nabla_x u^i(s,\tilde X_s)) \c (\nabla_x u^j(s,X_s)-\nabla_x u^j(s,\tilde X_s))ds \\
& \ \ \ \ +  \int_0^t Z_s \sum_{1 \le i,j \le d}\p_{\xi_j}\p_{\rho_i}\o (\frac{u(s,X_s)-u(s,\tilde X_s)}{\e},n(X_s)) \\
& \ \ \ \  \ \ \ \ \times (\nabla_x u^i(s,X_s)-\nabla_x u^i(s,\tilde X_s)) \c \nabla_x n^j(X_s)ds \\
& \ \ \ \ + \frac{\e}{2}  \int_0^t Z_s  \sum_{1 \le i,j \le d}\p_{\xi_j}\p_{\xi_i}\o (\frac{u(s,X_s)-u(s,\tilde X_s)}{\e},n(X_s))\nabla_x n^i(X_s) \c \nabla_x n^j(X_s)ds \\
& \ \ \ \ + \! \frac{\l^2}{2} \!  \!  \int_0^t \!  \! Z_s f_{\e}(s,X_s,\tilde X_s)  \! \!  \sum_{1 \le i,j \le d}  \!  \! (\p_{x_i}  \! H)(s,u(s,X_s)) (\p_{x_j}  \! H)(s,u(s,X_s)) \nabla_{x}  \! u^i(s,X_s)  \! \c  \! \nabla_{x}  \! u^j(s,X_s)ds \\
& \ \ \ \ + \! \l^2  \!  \! \int_0^t  \!  \! Z_s f_{\e}(s,X_s,\tilde X_s) \!  \!  \sum_{1 \le i,j \le d}  \!  \! (\p_{x_i}  \! H)(s,u(s,X_s)) (\p_{x_j}  \! H)(s,u(s,\tilde X_s)) \nabla_{x}  \! u^i(s,X_s)  \! \c  \! \nabla_{x}  \! u^j(s,\tilde X_s)ds \\
& \ \ \ \ + \! \frac{\l^2}{2}  \!  \! \int_0^t  \!  \! Z_s f_{\e}(s,X_s,\tilde X_s)  \!  \! \sum_{1 \le i,j \le d}  \!  \! (\p_{x_i}  \! H)(s,u(s,\tilde X_s)) (\p_{x_j}  \! H)(s,u(s,\tilde X_s))  \nabla_{x}  \! u^i(s,\tilde X_s)  \! \c  \! \nabla_{x}  \! u^j(s,\tilde X_s)ds \\
& \ \ \ \ -\l \int_0^t Z_s  \sum_{1 \le i,j  \le d}(\p_{x_i} H)(s,u(s,X_s)) \p_{\rho_j}\o (\frac{u(s,X_s)-u(s,\tilde X_s)}{\e},n(X_s))\\
& \ \ \ \  \ \ \ \ \times  \nabla_{x} u^i(s,X_s) \c (\nabla_x u^j(s,X_s)-\nabla_x u^j(s,\tilde X_s)) ds \\
& \ \ \ \ - \! \l \e   \!  \! \int_0^t  \!  \! Z_s \sum_{1 \le i,j \le d} (\p_{x_i}  \! H)(s,u(s,X_s)) \p_{\xi_j} \! \o (\frac{u(s,X_s) \! - \! u(s,\tilde X_s)}{\e},n(X_s)) \nabla_{x}  u^i(s,X_s)  \! \c \!  \nabla_x  n^j(X_s) ds\\
& \ \ \ \ -\l \int_0^t Z_s  \sum_{1 \le i,j  \le d}(\p_{x_i} H)(s,u(s,\tilde X_s)) \p_{\rho_j}\o (\frac{u(s,X_s)-u(s,\tilde X_s)}{\e},n(X_s))\\
& \ \ \ \  \ \ \ \ \times \nabla_{x} u^i(s,\tilde X_s) \c (\nabla_x u^j(s,X_s)-\nabla_x u^j(s,\tilde X_s)) ds \\
& \ \ \ \ - \! \l \e  \!  \!  \int_0^t \!  \!  Z_s  \!  \! \sum_{1 \le i,j \le d} \!  \!  (\p_{x_i} H)(s,u(s,\tilde X_s)) \p_{\xi_j}\o (\frac{u(s,X_s) \! - \! u(s,\tilde X_s)}{\e},n(X_s)) \nabla_{x} u^i(s,\tilde X_s)  \! \c  \! \nabla_x n^j(X_s) ds.
\end{split}
\end{eqnarray}

\newpage

\begin{thm} \label{T4.1}
Assume $(X_t,L_t)$ and $(\tilde X_t,\tilde L_t)$ are two  solutions to the reflecting SDEs (\ref{1.1}).
Then we have for any $0 \le t\le T_1$,
%\begin{eqnarray} \label{4.1}
%\begin{split}
%& \ \ \ \ F_{\e}(t,X_t,\tilde X_t)\\
%&=F_{\e}(0,X_0,Y_0)+\int_0^tZ_s[\p_t f_{\e}(s,X_s,\tilde X_s)-\l(\p_s H)(s,u(s,X_s))-\l (\p_s H)(s,u(s,\tilde X_s))]ds\\
%& \ \ \ \ +\sum_{1 \le i \le d}[\int_0^tZ_s\p_{x_i} f_{\e}(s,X_s,\tilde X_s)dX^i_s+\int_0^tZ_s\p_{y_i} F_{\e}(s,X_s,\tilde X_s)dY^i_s]\\
%& \ \ \ \ -\l\sum_{1 \le i \le d}[\int_0^tZ_sf_{\e}(s,X_s,\tilde X_s) [(\p_{x_i} H)(s,u(s,X_s))  \nabla_{x} u^i(s,X_s) \c %dW_s-(\p_{x_i} H)(s,u(s,\tilde X_s))  \nabla_{x} u^i(s,\tilde X_s) \c dW_s\\
%& \ \ \ \ +\frac{1}{2}\sum_{1 \le i,j \le d}\int_0^t[\p_{x_j}\p_{x_i} F_{\e}(s,X_s,\tilde X_s)+2\p_{y_j}\p_{x_i} %F_{\e}(s,X_s,\tilde X_s)+\p_{y_j}\p_{y_i} F_{\e}(s,X_s,\tilde X_s)]ds.
%\end{split}
%\end{eqnarray}

\begin{eqnarray} \label{4.1}
\begin{split}
& \ \ \ \ F_{\e}(t,X_t,\tilde X_t)=F_{\e}(0,X_0,\tilde X_0)+M_t+A^1_t+A^2_t.\\
\end{split}
\end{eqnarray}

\end{thm}

\begin{proof}
Assume $(X_t,L_t)$ and $(\tilde X_t,\tilde L_t)$ are two solutions to reflecting SDEs (\ref{1.1}). Applying Lemma \ref{L4.1}, we obtain
\begin{eqnarray} \nonumber
\begin{split}
u(t,X_t)=u(0,X_0)+\int_0^t \nabla_xu(s,X_s) \c dW_s+\int_0^tn(X_s)d|L|_s, \\
u(t,\tilde X_t)=u(0,\tilde X_0)+\int_0^t \nabla_xu(s,\tilde X_s) \c dW_s+\int_0^tn(\tilde X_s)d|\tilde L|_s. \\
\end{split}
\end{eqnarray}

Since $\omega  \in C^2(\R^{2d})$, using the It\^o's formula %to $\omega(\rho,\xi)$ with $u(t,X_t)-u(t,\tilde X_t)$ and $n(X_s)$,
we have

\begin{eqnarray} \label{3.50}
\begin{split}
& \ \ \ \ f_{\e}(t,X_t,\tilde X_t)\\
&=f_{\e}(0,X_0,\tilde X_0) \! +  \!   \int_0^t  \!  \! \sum_{1 \le i \le d} \!  \! \p_{\rho_i}\o (\frac{u(s,X_s) \! - \! u(s,\tilde X_s)}{\e},n(X_s)) (\nabla_x u^i(s,X_s) \! - \! \nabla_x u^i(s,\tilde X_s)) \! \c  \! dW_s \\
& \ \ \ \ +  \int_0^t \sum_{1 \le i \le d}\p_{\rho_i}\o (\frac{u(s,X_s)-u(s,\tilde X_s)}{\e},n(X_s)) (n^i(X_s)d|L|_s-n^i(\tilde X_s)d|\tilde L|_s) \\
& \ \ \ \ + \e \!  \int_0^t  \! \sum_{1 \le i  \le d} \! \p_{\xi_i}\o (\frac{u(s,X_s) \! - \! u(s,\tilde X_s)}{\e},n(X_s))(\nabla_x n^i(X_s) \c dW_s \! + \!  \nabla_x n^i(X_s)  \! \c  \! b(s,X_s)ds)\\
& \ \ \ \ + \e  \! \int_0^t  \! \sum_{1 \le i  \le d} \! \p_{\xi_i}\o (\frac{u(s,X_s)-u(s,\tilde X_s)}{\e},n(X_s)) (\nabla_x n^i(X_s) \c n(X_s)d|L|_s+\frac{1}{2}\Delta_x n^i(X_s)ds) \\
& \ \ \ \ + \frac{1}{2\e} \int_0^t \sum_{1 \le i,j \le d}\p_{\rho_j}\p_{\rho_i}\o (\frac{u(s,X_s)-u(s,\tilde X_s)}{\e},n(X_s)) \\
&\ \ \ \ \ \ \ \ \times (\nabla_x u^i(s,X_s)-\nabla_x u^i(s,\tilde X_s)) \c (\nabla_x u^j(s,X_s)-\nabla_x u^j(s,\tilde X_s))ds \\
& \ \ \ \ + \!  \int_0^t  \! \sum_{1 \le i,j \le d} \! \p_{\xi_j}\p_{\rho_i}\o (\frac{u(s,X_s) \! - \! u(s,\tilde X_s)}{\e},n(X_s)) (\nabla_x u^i(s,X_s) \! - \! \nabla_x u^i(s,\tilde X_s))  \! \c \!  \nabla_x n^j(X_s)ds \\
& \ \ \ \ + \frac{\e}{2} \int_0^t \sum_{1 \le i,j \le d}\p_{\xi_j}\p_{\xi_i}\o (\frac{u(s,X_s)-u(s,\tilde X_s)}{\e},n(X_s))\nabla_x n^i(X_s) \c \nabla_x n^j(X_s)ds. \\
\end{split}
\end{eqnarray}

On the other hand, since $\tilde H(t,x):=H(t,u(t,x))\in W^{1,2}_{2d+2}((0,T_1) \times D)$ by Lemma \ref{L3.5},  we can apply (\ref{1}) and Lemma \ref{L4.1}  to get
\begin{eqnarray} \label{3.51}
\begin{split}
& \ \ \ \ H(t,u(t,X_t))\\
&=H(0,u(0,X_0))+ \int_0^t \p_s \tilde H(s, X_s)ds+ \int_0^t \nabla_x \tilde H(s,X_s) \c dX_s+ \frac{1}{2} \int_0^t \Delta_x \tilde H(s, X_s)ds\\
&=H(0,u(0,X_0))+ \int_0^t [(\p_s H)(s,u(s,X_s))+\sum_{1 \le i \le d} (\p_{x_i} H) (s,u(s,X_s)) \p_s u^i(s,X_s)]ds\\
& \ \ \ \ + \int_0^t \sum_{1 \le i \le d} (\p_{x_i} H)(s,u(s,X_s)) [ \nabla_{x} u^i(s,X_s) \c dW_s+\nabla_{x} u^i(s,X_s) \c b(s,X_s)ds\\
& \ \ \ \ \ \ \ \ + \nabla_{x} u^i(s,X_s) \c n(X_s)d|L|_s]\\
& \ \ \ \ + \frac{1}{2} \int_0^t [\sum_{1 \le i,j \le d} (\p_{x_j} \p_{x_i} H)(s,u(s,X_s))  \nabla_{x}u^i(s,X_s)\c \nabla_{x}u^j(s,X_s)\\
& \ \ \ \ \ \ \ \ +\sum_{1 \le i \le d} (\p_{x_i} H)(s,u(s,X_s))  \Delta_{x}u^i(s,X_s)]ds\\
&=H(0,u(0,X_0))+ \int_0^t  \! (\p_s H)(s,u(s,X_s))ds+ \int_0^t  \! \sum_{1 \le i \le d} \! (\p_{x_i} H)(s,u(s,X_s)) [ \nabla_{x} u^i(s,X_s)  \! \c  \! dW_s\\
& \ \ \ \ + n^i(X_s)d|L|_s]+ \frac{1}{2} \int_0^t \sum_{1 \le i,j \le d} (\p_{x_j} \p_{x_i} H)(s,u(s,X_s))  \nabla_{x}u^i(s,X_s)\c \nabla_{x}u^j(s,X_s)ds,\\
\end{split}
\end{eqnarray}
and similarly,
\begin{eqnarray} \label{3.52}
\begin{split}
& \ \ \ \ H(t,u(t,\tilde X_t))\\
&=H(0,u(0,\tilde X_0))+ \int_0^t \!  (\p_s H)(s,u(s,\tilde X_s))ds+ \int_0^t \!  \sum_{1 \le i \le d} \! (\p_{x_i} H)(s,u(s,\tilde X_s)) [ \nabla_{x} u^i(s,\tilde X_s)  \! \c  \! dW_s\\
& \ \ \ \ + n^i(\tilde X_s)d|\tilde L|_s]+ \frac{1}{2} \int_0^t \sum_{1 \le i,j \le d} (\p_{x_j} \p_{x_i} H)(s,u(s,\tilde X_s))  \nabla_{x}u^i(s,\tilde X_s)\c \nabla_{x}u^j(s,\tilde X_s)ds.\\
\end{split}
\end{eqnarray}

By (\ref{3.50}), (\ref{3.51}), (\ref{3.52}) and the integration by parts, we easily deduce (\ref{4.1}).

\end{proof}

Now we are ready to prove the main result of  the paper:

\begin{thm}
For any $x \in \bar D$, the reflecting SDEs (\ref{1.1}) has a unique strong solution $(X_t,L_t)$ with $X_0=x$.
\end{thm}

\begin{proof}
By Proposition \ref{P4.1}, we know that the reflecting SDEs (\ref{1.1}) has a unique weak solution. Hence
by the Yamada-Watanabe theorem, it is sufficient to prove the pathwise uniqueness of the reflecting SDEs (\ref{1.1}).
% that if $(X_t,L_t)$ and $(\tilde X_t,\tilde L_t)$ are two strong solutions to the reflecting SDEs (\ref{1.1}) with $X_0=\tilde X_0=x$, then $X_t=\tilde X_t$ for all $t>0$, $P$-$a.e.$.

Assume $(X_t,L_t)$ and $(\tilde X_t,\tilde L_t)$ are two solutions to reflecting SDEs (\ref{1.1}) with $X_0=\tilde X_0=x$. Let $A^1_t$, $A_t^2$ be defined as in (\ref{4.1}).
Then by Proposition \ref{P3.5}, (\ref{3.36}), (\ref{3.53}) and (\ref{3.54}), we have for $t \in [0, T_1]$,
\begin{eqnarray} \label{2-20-2}
\begin{split}
& \ \ \ \ A^1_t\\
&=- \l \int_0^t  \! Z_s f_{\e}(s,X_s,\tilde X_s) [\nabla_{x}H(s,u(s,X_s)) \c n(X_s)d|L|_s \! + \!  \nabla_{x} H(s,u(s,\tilde X_s))\c n(\tilde X_s)d|\tilde L|_s ]\\
& \ \ \ \ + \int_0^t Z_s \nabla_xf_{\e}(s,X_s,\tilde X_s) \c n(X_s) d|L|_s+\int_0^t Z_s \nabla_yf_{\e}(s,X_s,\tilde X_s) \c n(\tilde X_s) d|\tilde L|_s \\
&\le \int_0^t Z_s (M_6\frac{|X_s-\tilde X_s|^2}{\e}- \l M_7\frac{|X_s-\tilde X_s|^2}{\e})  d|L|_s\\
& \ \ \ \ + \int_0^t Z_s  (M_6\frac{|X_s-\tilde X_s|^2}{\e}- \l M_7\frac{|X_s-\tilde X_s|^2}{\e}) d|\tilde L|_s.\\
\end{split}
\end{eqnarray}
Hence we can take $\l:=\frac{M_6}{M_7}$ so that $A^1_t \le 0$ for any $\e>0$.

For the term $A^2_t$, note that $n   \in    C^2_0( \Rd), \ H  \in    C^{0,1}_b([0,T_1]\times \Rd)$,
$u   \in    C^{0,1}_b([0,T_1]\times G')$. By (\ref{3.2}), (\ref{3.61}), (\ref{3.38}), (\ref{3.36}) and the H\"{o}lder inequality, we have for $t \in [0, T_1]$,
\begin{eqnarray} \label{2-20-3}
\begin{split}
& \ \ \ \ A^2_t\\
&\les   \l \int_0^t  (\e+\frac{|X_s-\tilde X_s|^2}{\e}) [|(\p_s H)(s,u(s,X_s))|+|(\p_s H)(s,u(s,\tilde X_s))|]ds\\
& \ \ \ \ + \l \sum_{1 \le i,j \le d}\int_0^t  (\e+\frac{|X_s-\tilde X_s|^2}{\e}) [|(\p_{x_j} \p_{x_i} H)(s,u(s,X_s))|+|(\p_{x_j} \p_{x_i} H)(s,u(s,\tilde X_s))|]ds\\
& \ \ \ \ + \e \int_0^t  |\frac{u(s,X_s)-u(s,\tilde X_s)}{\e}|^2 [|b(s,X_s)|+\frac{1}{2}|\Delta_x n^i(X_s)|]ds\\
& \ \ \ \ + \frac{1}{ \e} \int_0^t   \|\nabla_x u(s,X_s)-\nabla_x u(s,\tilde X_s)\|^2ds \\
& \ \ \ \ +  \int_0^t  |\frac{u(s,X_s)-u(s,\tilde X_s)}{\e}| \|\nabla_x u(s,X_s)-\nabla_x u(s,\tilde X_s)\|ds \\
& \ \ \ \ + \e(1+\l )  \int_0^t |\frac{u(s,X_s)-u(s,\tilde X_s)}{\e}|^2ds + \l^2 \int_0^t  (\e+\frac{|X_s-\tilde X_s|^2}{\e}) ds\\
&\ \ \ \ +\l \int_0^t   |\frac{u(s,X_s)-u(s,\tilde X_s)}{\e}| \|\nabla_x u(s,X_s)-\nabla_x u(s,\tilde X_s)\|  ds \\
&\les   \e \int_0^t [|(\p_s H)(s,u(s,X_s))|+|(\p_s H)(s,u(s,\tilde X_s))|+1]ds\\
&\ \ \ \ + \e  \sum_{1 \le i,j \le d} \int_0^t [|(\p_{x_j} \p_{x_i} H)(s,u(s,X_s))|+|(\p_{x_j} \p_{x_i} H)(s,u(s,\tilde X_s))|]ds\\
&\ \ \ \ + \frac{1}{\e}\int_0^t  |X_s-\tilde X_s|^2 [|(\p_s H)(s,u(s,X_s))|+|(\p_s H)(s,u(s,\tilde X_s))|+| b(s,X_s)|+1] ds \\
& \ \ \ \  +\frac{1}{\e}\sum_{1 \le i,j \le d} \int_0^t |X_s-\tilde X_s|^2 [|(\p_{x_j} \p_{x_i} H)(s,u(s,X_s))|+|(\p_{x_j} \p_{x_i} H)(s,u(s,\tilde X_s))|]ds\\
& \ \ \ \ +\frac{1}{ \e} \int_0^t   \|\nabla_x u(s,X_s)-\nabla_x u(s,\tilde X_s)\|^2ds.\\
\end{split}
\end{eqnarray}
Since $\p D$ is smooth, there exist a function $v \in W_{2d+2}^{1,2}((0,T_1)\times \Rd)$ and a sequence of functions $\{v_n\}_{n \ge 1} \subset C^{1,2}_b((0,T_1)\times \Rd)$ such that $v(t,x)=u(t,x)$ on $(0,T_1)\times D$ and $v_n$ converges to $v$ in $W_{2d+2}^{1,2}((0,T_1)\times \Rd)$. Therefore, $\nabla_x v_n$  converges to $\nabla_x v$ uniformly on $(0,T_1)\times \Rd$ by Sobolev inequality. Hence by (\ref{3.86}) we have
\begin{eqnarray} \label{2-20-4}
\begin{split}
&\ \ \ \ \int_0^t   \|\nabla_x u(s,X_s)-\nabla_x u(s,\tilde X_s)\|^2ds \\
&=\lim_{n \to \infty} \int_0^t   \|\nabla_x v_n(s,X_s)-\nabla_x v_n (s,\tilde X_s)\|^2ds \\
&\le \lim_{n \to \infty} \sum_{1 \le i  \le d} \int_0^t|\int_0^1 \nabla_{x} \p_{x_i}v_n(s,\a X_s+(1-\a)\tilde X_s)\c (X_s-\tilde X_s)d\a|^2ds\\
&\le  \sum_{1 \le i,j \le d}  \int_0^t |X_s-\tilde X_s|^2 \int_0^1 |\p_{x_j}  \p_{x_i}v(s,\a X_s+(1-\a)\tilde X_s)|^2d\a ds,
\end{split}
\end{eqnarray}
and $\int_0^t \int_0^1 |\nabla_x \p_{x_i}v(s,\a X_s+(1-\a)\tilde X_s)|^2d\a ds<\infty$, $P$-$a.e.$.
Combing this with (\ref{3.36}), (\ref{2-20-2}),  (\ref{2-20-3}),  (\ref{2-20-4})  and Theorem \ref{T4.1}, we have
\begin{eqnarray} \label{5.10}
\begin{split}
& \ \ \ \ \frac{1}{\e}|X_s-\tilde X_s|^2 \\
&\les  F_{\e}(t,X_t,\tilde X_t) \\
&\les \e+  M_t+ \e \int_0^t [|(\p_s H)(s,u(s,X_s))|+|(\p_s H)(s,u(s,\tilde X_s))|+1]ds\\
&\ \ \ \ + \e  \sum_{1 \le i,j \le d} \int_0^t [|(\p_{x_j} \p_{x_i} H)(s,u(s,X_s))|+|(\p_{x_j} \p_{x_i} H)(s,u(s,\tilde X_s))|]ds\\
&\ \ \ \ + \frac{1}{\e}\int_0^t  |X_s-\tilde X_s|^2dC_t ,
\end{split}
\end{eqnarray}
where
\begin{eqnarray} \nonumber
\begin{split}
C_t:=&\int_0^t   [ 1+| b(s,X_s)|+|(\p_s H)(s,u(s,X_s))|+|(\p_s H)(s,u(s,\tilde X_s))| \\
&+\sum_{1 \le i,j \le d}( |(\p_{x_j} \p_{x_i} H)(s,u(s,X_s))|+|(\p_{x_j} \p_{x_i} H)(s,u(s,\tilde X_s))| \\
&+\int_0^1 |\p_{x_j} \p_{x_i}v(s,\a X_s+(1-\a)\tilde X_s)|^2d\a)   ] ds.
\end{split}
\end{eqnarray}
Set $\tau_R:=\inf \{t\ge 0:C_t\ge R\} \wedge T_1$. Applying Lemma \ref{L3.11} to (\ref{5.10}) with $p=2q=\frac{1}{2}$, we have
\begin{eqnarray} \label{6.11}
\begin{split}
& \ \ \ \ [E(\sup_{s\in [0,\tau_R]}|X_s-\tilde X_s|^{1/4})]^{4}\\
&\les 16\e^2 {Ee^{ C_{\tau_R}}}\left (1+E\int_0^{T_1} [|(\p_s H)(s,u(s,X_s))|+|(\p_s H)(s,u(s,\tilde X_s))|]ds\right.\\
&\ \ \ \ + \left.\sum_{1 \le i,j \le d} E\int_0^{T_1} [|(\p_{x_j} \p_{x_i} H)(s,u(s,X_s))|+|(\p_{x_j} \p_{x_i} H)(s,u(s,\tilde X_s))|]ds\right)\\
&\le 16\e^2  e^{R}\left(1+E\int_0^{T_1} [|(\p_s H)(s,u(s,X_s))|+|(\p_s H)(s,u(s,\tilde X_s))|]ds\right. \\
&\ \ \ \ \left.+ \sum_{1 \le i,j \le d} E\int_0^{T_1} [|(\p_{x_j} \p_{x_i} H)(s,u(s,X_s))|+|(\p_{x_j} \p_{x_i} H)(s,u(s,\tilde X_s))|]ds\right).
\end{split}
\end{eqnarray}

Note that $(\p_s H)(s,u(s,x)),(\p_{x_j} \p_{x_i} H)(s,u(s,x))\in L^{2d+2}([0,T_1]\times D)$. Hence by (\ref{3.85}),
\begin{eqnarray} \nonumber
\begin{split}
&E\int_0^{T_1} [|(\p_s H)(s,u(s,X_s))|+|(\p_s H)(s,u(s,\tilde X_s))|]ds \\
&+ \sum_{1 \le i,j \le d} E\int_0^{T_1} [|(\p_{x_j} \p_{x_i} H)(s,u(s,X_s))|+|(\p_{x_j} \p_{x_i} H)(s,u(s,\tilde X_s))|]ds<\infty.
\end{split}
\end{eqnarray}
Letting $\e \to 0$ and $R \to \infty$ in (\ref{6.11}), we get $E(\sup_{s\in [0,T_1]}|X_s-\tilde X_s|^{1/4})=0$, which implies $X_t=\tilde X_t$ for all $0\le t\le T_1$, $P$-$a.e.$. Since $T_1$ is independent of the initial value $x$,  using a standard procedure, we can conclude $X_t=\tilde X_t$ for all $t>0$, $P$-$a.e.$. This completes the proof of the theorem.
\end{proof}

% By Lemma 6.32 in \cite{Gilbarg}, to prove (\ref{5.32}), we need only to show that there exists a constant $c_1>0$, depending only on $D$, $d$, $\a_0$, the function $\max_{1 \le i \le d}M_{\mu_i}^{1-\a_0}(\cdot)$ and $M_{\nu}^{1-\a_0}(\d_1)$, such that

\vskip 0.4cm

\section{Appendix}
In this part, we provide the proofs of Proposition \ref{P4.1.1} and \ref{P4.2.1}.
Fix $z:=y(t_0,z_0,\frac{\d_5}{2})$.
Recall that $D(t,\e)$ and $C(z,\d)$ were respectively defined in (\ref{2.19-1}) and (\ref{2.19}), and
$$\O_{\e}=\{(t,x):  t\in ((t_0-\eta_2)\vee 0,(t_0+\eta_2)\wedge T_1), \ x\in C(z,\d_5) \bigcap D(t,\e)\}.$$
Set $\varrho(\e):=(\frac{\e}{2} \sin \frac{\t_0}{2}) \wedge  (\frac{\d_5}{32}\sin \frac{\t_0}{2})$ and
$$\O_{\e}':=\{(t,x):  t\in ((t_0-\eta_2)\vee 0,(t_0+\eta_2)\wedge T_1), \ x\in D(t,\varrho(\e))\}.$$

Before giving the proofs of Proposition \ref{P4.1.1} and \ref{P4.2.1}, we need a simple Lemma.

\begin{lem} \label{L7.1}
For any constant $p>0$ and function $f \in L^p(\tilde D)$, we have
\begin{eqnarray}  \label{7.1}
\begin{split}
% \!  \! \lim_{n_1,n_2 \to \infty}  \! \! \| \! \int_{- \! \rho_0}^{\rho_0}  \!  | \! f(t,y_{n_1}(t,x,r)  \! ) \! - \! f(t,y_{n_2}(t,x,r) \! ) \! |   I_{u(t,D)} \! (y_{n_1}(t, x, r) \! )  I_{u(t,D)} \! (y_{n_2}(t, x, r) \! ) dr  \!  \|_{ \! L^p \! (\tilde D)} \! = \! 0 \! .
& \lim\limits_{n_1,n_2 \to \infty} \| \int_{- \rho_0}^{\rho_0}  |   f(t,y_{n_1}(t,x,r)    )   -  f(t,y_{n_2}(t,x,r) ) |   I_{u(t,D)} (y_{n_1}(t, x, r)  ) \\
 &  \ \ \ \ \ \ \ \  \ \ \ \ \ \ \ \ \  \ \times I_{u(t,D)}  (y_{n_2}(t, x, r) ) dr   \|_{  L^p  (\tilde D)} = 0  .
 \end{split}
\end{eqnarray}
\end{lem}

\begin{proof}
For any $\e>0$, there exists a function $\tilde f \in C_b((0,T_1)\times \Rd)$ such that $\|  f - \tilde f   \|_{L^p(\tilde D)}^p<\e$. Then by (\ref{2.29}) and
Lemma \ref{L4.2} we have for $n_1,n_2 \ge N_0$,
\begin{eqnarray}  \nonumber
\begin{split}
&\lim_{n_1,n_2 \to \infty} \|\int_{-\rho_0}^{\rho_0} |f(t,y_{n_1}(t,x,r) )-f(t,y_{n_2}(t,x,r))|   I_{u(t,D)}(y_{n_1}(t, x, r))  I_{u(t,D)}(y_{n_2}(t, x, r)) dr \|^p_{L^p(\tilde D)} \\
& \les \lim_{n_1,n_2 \to \infty} \int_{-\rho_0}^{\rho_0} \int_{\tilde D} |f(t,y_{n_1}(t,x,r) )-\tilde f(t,y_{n_1}(t,x,r))|^p   I_{u(t,D)}(y_{n_1}(t, x, r))    dxdt dr \\
& \ \ \ \ + \lim_{n_1,n_2 \to \infty} \int_{-\rho_0}^{\rho_0} \int_{\tilde D} |f(t,y_{n_2}(t,x,r) )-\tilde f(t,y_{n_2}(t,x,r))|^p   I_{u(t,D)}(y_{n_2}(t, x, r))    dxdt dr \\
& \ \ \ \ + \lim_{n_1,n_2 \to \infty} \int_{-\rho_0}^{\rho_0} \int_{\tilde D} |\tilde f(t,y_{n_1}(t,x,r) )-\tilde f(t,y_{n_2}(t,x,r))|^p    dxdt dr \\
& \le 4   \int_{-\rho_0}^{\rho_0}   \int_{\tilde D}    |f(t,x)-\tilde f(t,x)|^p   dxdt dr \le 8 \rho_0 \e.
\end{split}
\end{eqnarray}
Since $\e$ is arbitrary, (\ref{7.1}) follows.
\end{proof}

\vskip 0.5cm

\textbf{Proof of Proposition \ref{P4.1.1}}
\vskip 0.3cm
By Proposition \ref{P3.3} and (\ref{3.29}), to show (\ref{4.49}), we need only to show (\ref{4.48}), $i.e.$ for any  $1\le i,j\le d$,
\begin{eqnarray} \label{2.13}
\psi_{i}^j ( \c ,\c, \G^z(\c,\c))\in W_{2d+2}^{0,1}((((t_0-\eta_2)\vee 0,(t_0+\eta_2)\wedge T_1)\times C(z,\d_5))\bigcap \tilde D).
\end{eqnarray}

For any given $\e>0$, recall that $N_1(\e)$ is given in Lemma \ref{L3.4}.
Firstly, we show that for $n\ge N_1(\e)$, $\psi_{n,i}^j(t,x, \G^z(t,x))\in C_b^{0,1}(\O_{\e})$.

%set $\Pi_1:=\{x\in D(t,\e)\bigcap C(z,\d_5):\G^z(t,x)\neq 0\}$ and $\Pi_2:=\{x\in D(t,\e)\bigcap C(z,\d_5):\G^z(t,x)= 0\}$. Since $\G^z(t,x)$ is continuous,
%$\Pi_1$ is an open set in $\Rd$. By (\ref{3.48}), for $x\in \Pi_1$ and $r\in (0, \G^z(t,x))$, we have $y_n(t,x,r)\in D(t,\varrho(\e))$. Together with Proposition \ref{P3.3}, we have $\psi_{n, i}^{i}\left(t, \c, r)\right)$ and $\partial_{y_{k}} \g_{n}^{i}(t, y_n(t, \c, r))$ belong to $C_b^{1}(\Pi_1)$.  Therefore $\psi_{n,i}^i(t,x, \G^z(t,x))\in C_b^{0,1}(\O_{\e})$ and for $1\le m \le d$,
When $(t,x)\in \O_{\e}$ with $\G^z(t,x)\neq 0$, by (\ref{3.48}), we have $y_n(t,x,r)\in D(t,\varrho(\e))$ for $r\in (0, \G^z(t,x)]$. Hence together with Proposition \ref{P3.3} and Lemma \ref{L3.3}, we can see that $\psi_{n,i}^j(t,x, \G^z(t,x))\in C_b^{0,1}(\O_{\e})$ and for $1\le m \le d$,
\begin{eqnarray} \label{3.33}
\begin{split}
\frac{d}{d x_{m}} \psi_{n, i}^{j}\left(t, x, \Gamma^{z}(t, x)\right)&= \int_{0}^{\Gamma^{z}(t, x)} \sum_{1 \leqslant k, l \leqslant d}\p_{y_{l}} \partial_{y_{k}}  \gamma_{n}^{j}(t, y_n(t, x, r))\psi_{n, m}^{l}(t, x, r) \psi_{n, i}^{k}(t, x, r) d r\\
&\ \ \ \ + \int_{0}^{\Gamma^{z}(t, x)}\sum_{1 \leq k \leq d} \partial_{y_{k}} \g_{n}^{j}(t, y_n(t, x, r)) \partial_{x_{m}} \psi_{n, i}^{k}(t, x, r) d r\\
&\ \ \ \ +\sum_{ 1 \le k \leq d} \partial_{y_{k}} \gamma_{n}^{j}\left(t, y_n\left(t, x, \Gamma^{z}(t, x)\right)\right)\psi_{n,i}^{k}\left(t, x, \Gamma^{z}\left(t, x\right)\right)\partial_{x_{m}} \Gamma^{z}(t, x),\\
\end{split}
\end{eqnarray}
where $\frac{d}{d x_{m}} \psi_{n, i}^{j}\left(t, x, \Gamma^{z}(t, x)\right)$ stands for the partial derivative of $\psi_{n, i}^{j}\left(t, x, \Gamma^{z}(t, x)\right)$ $w.r.t.$ $x_m$.

%When $x\in \Pi_2$,
%% by (\ref{3.78}) and the continuity of $ \Gamma^{z}(t, \c)$, we have $y_n(t, x+y,r) \in u(t,D)$ for $r \in (0,\Gamma^{z}(t, x+y))$. Hence by Proposition \ref{P3.3},
%by Proposition \ref{P3.3} and Lemma \ref{L3.3}, the continuity of $ \Gamma^{z}(t, \c)$, we have for unite vector $y \in \Rd$ and $h \in \R$,
%\begin{eqnarray}
%\begin{split}
%&\ \ \ \ \frac{\psi_{n, i}^{i}\left(t, x+hy, \Gamma^{z}(t, x+hy)\right)-\psi_{n, i}^{i}\left(t, x, \Gamma^{z}(t, x)\right)}{h}\\
%&=\frac{\psi_{n, i}^{i}\left(t, x+hy, \Gamma^{z}(t, x+hy)\right)-\psi_{n, i}^{i}\left(t, x+hy, \Gamma^{z}(t, x)\right)}{h}\\
%&=\sum_{1\le k \le d}\frac{1}{h}\int_{\Gamma^{z}(t, x)}^{\Gamma^{z}(t, x+hy)} \partial_{y_k}\g_n^i(t, y_n(t, x+hy,r))\psi_{n,i}^k(t,x+hy, r)dr \to \partial_{y_i}\g_n^i(t, x) \nabla_x \Gamma^{z}(t, x) \c y.
%\end{split}
%\end{eqnarray}
%Hence $\psi_{n,i}^i(t,x, \G^z(t,x))\in C_b^{0,1}(\O_{\e})$ and (\ref{3.33}) holds for all $(t,x)\in \O_{\e}$.

\vskip 0.3cm

Now we show that
\begin{eqnarray}\label{3.80}
\sup_{\e >0} \sup_{n\ge N_1(\e)} \|\int_0^{\G^z(t,x)}|\nabla_x \psi_{n,i}^j(t,x, r)|dr\|_{L^{2d+2}(\O_{\e})}< \infty.
\end{eqnarray}

For $n \ge N_1(\e)$ and $(t,x)\in \O_{\e}$, by Lemma \ref{L4.2} and (\ref{3.48}) we have for $1 \le m,k \le d$,
\begin{eqnarray} \nonumber
\begin{split}
&\ \ \ \  \|\int_0^{\G^z(t,x)}|\p_{y_m} \p_{y_k}\g_n^j(t,y_n(t,x,r))|dr\|_{L^{2d+2}(\O_{\e})}\\
&\les  (\int_{(t_0-\eta_2)\vee 0}^{(t_0+\eta_2)\wedge T_1}\int_{\Rd} | \int_0^{\G^z(t,x)} |\p_{y_m} \p_{y_k}\g_n^j(t,y_n(t,x,r))|^{2d+2}  I_{D(t,\varrho(\e))}(y_n(t,x,r)) dr | dxdt)^{\frac{1}{2d+2}}\\
&\le    (\int_{-\rho_1}^{\rho_1}\int_{(t_0-\eta_2)\vee 0}^{(t_0+\eta_2)\wedge T_1}\int_{\Rd}|\p_{y_m} \p_{y_k}\g_n^j(t,y_n(t,x,r))|^{2d+2}  I_{D(t,\varrho(\e))}(y_n(t,x,r)) dxdtdr)^{\frac{1}{2d+2}}\\
&\les  (\int_{-\rho_1}^{\rho_1}\int_{(t_0-\eta_2)\vee 0}^{(t_0+\eta_2)\wedge T_1}\int_{\Rd}|\p_{y_m} \p_{y_k}\g_n^j(t,x)|^{2d+2} I_{D(t,\varrho(\e))}(x)dxdtdr)^{\frac{1}{2d+2}} \\
&\les   \| \g_n^j \|_{W_{2d+2}^{0,2}(\O_\e')}.\\
\end{split}
\end{eqnarray}
Together with (\ref{3.79}), we deduce that
\begin{eqnarray} \label{7.2}
\begin{split}
\sup_{\e >0} \sup_{n\ge N_1(\e)}  \|\int_0^{\G^z(t,x)}|\p_{y_m} \p_{y_k}\g_n^j(t,y_n(t,x,r))|dr\|_{L^{2d+2}(\O_{\e})}< \infty.
\end{split}
\end{eqnarray}
Applying the Gronwall's inequality to equation (\ref{4.53}), it is easy to see that
\begin{eqnarray} \nonumber
\begin{split}
&\ \ \ \ \sup_{\e >0} \sup_{n\ge N_1(\e)} \sum_{1\le j\le d}|\int_0^{\G^z(t,x)} \|\nabla_x \psi_{n,i}^j(t,x, r)|dr \|_{L^{2d+2}(\O_{\e})} \\
&\les \sup_{\e >0} \sup_{n\ge N_1(\e)}  \sum_{1\le j,m,k\le d} \|\int_0^{\G^z(t,x)}|\p_{y_m} \p_{y_k}\g_n^j(t,y_n(t,x,r))|dr \|_{L^{2d+2}(\O_{\e})} <\infty.
\end{split}
\end{eqnarray}

\vskip 0.3cm

Next we show that for any $\e>0$,
\begin{eqnarray}\label{3.81}
\sup_{n_1,n_2 \to \infty } \|\int_0^{\G^z(t,x)}|\nabla_x \psi_{n_1,i}^j(t,x, r)-\nabla_x \psi_{n_2,i}^j(t,x, r)|dr\|_{L^{2d+2}(\O_{\e})}=0.
\end{eqnarray}

In view of the boundness of $|\psi_{n,i}^j(t,x, r)|$ and $|\nabla_x\g_n^j(t,x)|$, Lemma \ref{L4.2} and (\ref{3.48}), using the Gronwall's inequality we have for $n_1,n_2\ge N_1(\e)$,
\begin{eqnarray} \nonumber
\begin{split}
&\ \ \ \ \sum_{1\le j \le d}  \|\int_0^{\G^z(t,x)}|\nabla_x \psi_{n_1,i}^j(t,x, r)-\nabla_x \psi_{n_2,i}^j(t,x, r)|dr\|_{L^{2d+2}(\O_{\e})}\\
&\les \sum_{1\le j,m,k\le d} \|\int_{-\rho_1}^{\rho_1}|\p_{y_m} \p_{y_k}\g_{n_1}^j(t,y_{n_1}(t,x,r))-\p_{y_m} \p_{y_k}\g^j(t,y_{n_1}(t,x,r)) |\\
&\ \ \ \ \ \ \ \ \ \ \ \ \ \ \ \ \times  I_{D(t,\varrho(\e))}(y_{n_1}(t,x,r))dr\|_{L^{2d+2}(\O_{\e})}\\
&\ \ \ \ + \sum_{1\le j,m,k\le d} \|\int_{-\rho_1}^{\rho_1}|\p_{y_m} \p_{y_k}\g^j(t,y_{n_1}(t,x,r))-\p_{y_m} \p_{y_k}\g^j(t,y_{n_2}(t,x,r)) |\\
&\ \ \ \ \ \ \ \ \ \ \ \ \ \ \ \ \times I_{D(t,\varrho(\e))}(y_{n_1}(t,x,r))I_{D(t,\varrho(\e))}(y_{n_2}(t,x,r))dr\|_{L^{2d+2}(\O_{\e})}\\
&\ \ \ \ + \sum_{1\le j,m,k\le d} \|\int_{-\rho_1}^{\rho_1}|\p_{y_m} \p_{y_k}\g^j(t,y_{n_2}(t,x,r)) |I_{D(t,\varrho(\e))}(y_{n_2}(t,x,r))dr\|_{L^{2d+2}(\O_{\e})}\\
&\ \ \ \ \ \ \ \ \ \ \ \ \ \ \ \ \times \sup_{(t,x)\in \O_{\e},\atop |r|< \rho_1,1\le k \le d}|\psi_{n_1,i}^k(t,x, r)-\psi_{n_2,i}^k(t,x, r)|\\
&\ \ \ \ + \sum_{1\le j,m,k\le d} \|\int_{-\rho_1}^{\rho_1}|\p_{y_m} \p_{y_k}\g_{n_2}^j(t,y_{n_2}(t,x,r))-\p_{y_m} \p_{y_k}\g^j(t,y_{n_2}(t,x,r)) | \\
&\ \ \ \ \ \ \ \ \ \ \ \ \ \ \ \ \times  I_{D(t,\varrho(\e))}(y_{n_2}(t,x,r))dr\|_{L^{2d+2}(\O_{\e})}\\
%&\ \ \ \ + \sum_{1\le j,m,k\le d} \|\int_{-\rho_1}^{\rho_1}|\p_{y_m} \p_{y_k}\g_{n_2}^j(t,y_{n_2}(t,x,r)) |I_{D(t,\varrho(\e))}(y_{n_2}(t,x,r))dr\|_{L^{2d+2}(\O_{\e})}\\
%&\ \ \ \ \ \ \ \ \ \ \ \ \ \ \ \ \times  \sup_{(t,x)\in \O_{\e},\atop |r|< \rho_1,1\le l,m \le d}|\psi_{n_2,m}^l(t,x, r)-\psi_{m}^l(t,x, r)|\\
&\ \ \ \ + \sum_{1\le k\le d} \|\int_0^{\G^z(t,x)}|\nabla_x \psi_{n_1,i}^k(t,x, r)|dr\|_{L^{2d+2}(\O_{\e})}\\
&\ \ \ \ \ \ \ \ \ \ \ \ \ \ \ \ \times   \sup_{(t,x)\in \O_{\e},\atop |r|< \rho_1,1\le j \le d}|\nabla_y \g_{n_1}^j(t,y_{n_1}(t,x,r))-\nabla_y \g_{n_2}^j(t,y_{n_2}(t,x,r))|\\
%&+ \sum_{1\le m,k \le d} \| \partial_{y_{k}} \gamma_{n_1}^{i}\left(t, y_{n_1}\left(t, x, \Gamma^{z}(t, x)\right)\right)\psi_{n_1,i}^{k}\left(t, x, \Gamma^{z}\left(t, x\right)\right)-\partial_{y_{k}} \gamma_{n_2}^{i}\left(t, y_{n_2}\left(t, x, \Gamma^{z}(t, x)\right)\right)\psi_{n_2,i}^{k}\left(t, x, \Gamma^{z}\left(t, x\right)\right)\|_{L^{2d+2}(\O_{\e})}\\
&\les \sum_{1\le j \le d} ( \|\g_{n_1}^j-\g^j\|_{W_{2d+2}^{0,2}(\O_{\e}')}+ \|\g_{n_2}^j-\g^j\|_{W_{2d+2}^{0,2}(\O_{\e}')})\\
&\ \ \ \ + \sum_{1\le j,m,k\le d} \|\int_{-\rho_1}^{\rho_1}|\p_{y_m} \p_{y_k}\g^j(t,y_{n_1}(t,x,r))-\p_{y_m} \p_{y_k}\g^j(t,y_{n_2}(t,x,r)) |\\
&\ \ \ \ \ \ \ \ \ \ \ \ \ \ \ \ \times I_{D(t,\varrho(\e))}(y_{n_1}(t,x,r))I_{D(t,\varrho(\e))}(y_{n_2}(t,x,r))dr\|_{L^{2d+2}(\O_{\e})}\\
&\ \ \ \ +  \sum_{1\le j \le d} \| \g^j  \|_{W_{2d+2}^{0,2}(\O_{\e}')}\sup_{(t,x)\in \O_{\e},\atop |r|< \rho_1,1\le k \le d}|\psi_{n_1,i}^k(t,x, r)-\psi_{n_2,i}^k(t,x, r)|\\
&\ \ \ \ +  \sum_{1\le k \le d} \|\int_0^{\G^z(t,x)}|\nabla_x \psi_{n_1,i}^k(t,x, r)|dr\|_{L^{2d+2}(\O_{\e})}\\
&\ \ \ \ \ \ \ \ \ \ \ \ \ \ \ \ \times   \sup_{(t,x)\in \O_{\e},\atop |r|< \rho_1,1\le j \le d}|\nabla_y \g_{n_1}^j(t,y_{n_1}(t,x,r))-\nabla_y \g_{n_2}^j(t,y_{n_2}(t,x,r))|,
\end{split}
\end{eqnarray}
Then by Proposition \ref{P3.3}, (\ref{2.29}), (\ref{3.30}), (\ref{3.80}) and Lemma \ref{L7.1}, we get (\ref{3.81}).

\vskip 0.3cm

Next we show that
\begin{eqnarray} \label {3.34}
\lim_{n_1,n_2\to \infty} \|\psi_{n_1,i}^j(t,x, \G^z(t,x))- \psi_{n_2,i}^j(t,x, \G^z(t,x))\|_{W_{2d+2}^{0,1}(\O_{\e})}=0.
\end{eqnarray}

By (\ref{4.4}), (\ref{3.48}), (\ref{3.33}), the boundness of $|\nabla_x \G^z(t,x)|$ and $|\nabla_x \g_n(t,x)|$, for $n_1,n_2\ge N_1(\e)$ and $(t,x)\in \O_{\e}$, we have for $1 \le m \le d$,
\begin{eqnarray} \nonumber
\begin{split}
&\ \ \ \ | \frac{d}{d x_{m}} \psi_{n_1, i}^{j}\left(t, x, \G^{z}(t, x)\right)-\frac{d}{d x_{m}} \psi_{{n_{2}, i}}^{j}\left(t, x, \Gamma^{z}(t, x)\right)|\\
&\le  | \! \! \int_{0}^{\Gamma^{z}(t, x)}  \! \! \!  \sum_{1 \leqslant k, l \leqslant d}  \! |\p_{y_{l}} \partial_{y_{k}}  \g_{n_{1}}^{j}(t, y_{n_1}(t, x, r)) \! -  \! \p_{y_{l}} \partial_{y_{k}}  \g_{n_{2}}^{j}(t, y_{n_1}(t, x, r)) | dr|  \! \! \sup _{(t,x)\in \O_{\e}, |r|< \rho_1, \atop n\ge N_1,1\le m,l \le d} \!  \left|\psi_{n, m}^{l}(t, x, r)\right|^2 \\
&\ \ \ \  +| \! \! \int_{0}^{\Gamma^{z}(t, x)}  \!  \!  \! \! \sum_{1 \leqslant k, l \leqslant d}  \!  \! |\p_{y_{l}} \partial_{y_{k}}   \! \g_{n_{2}}^{j}(t, y_{n_1}(t, x, r)) \! - \!  \p_{y_{l}} \partial_{y_{k}}   \! \g_{n_{2}}^{j}(t, y_{n_2}(t, x, r)) | dr |  \! \! \sup _{(t,x)\in \O_{\e}, |r|< \rho_1, \atop n\ge N_1,1\le m,l \le d} \!  \! \left|\psi_{n, m}^{l}(t, x, r)\right|^2 \\
&\ \ \ \  +| \int_{0}^{\Gamma^{z}(t, x)} \sum_{1 \leqslant k, l \leqslant d} |\p_{y_{l}} \partial_{y_{k}}  \g_{n_{2}}^{j}(t, y_{n_2}(t, x, r))| dr | \sup _{(t,x)\in \O_{\e}, |r|< \rho_1, \atop 1\le  m,l \le d} | \psi_{n_1, m}^{l}(t, x, r)-\psi_{n_2, m}^{l}(t, x, r)| \\
&\ \ \ \ \ \ \ \ \   \times \sup _{(t,x)\in \O_{\e}, |r|< \rho_1, \atop n\ge N_1,1\le m,k \le d}\left|\psi_{n, m}^{k}(t, x, r)\right|\\
&\ \ \ \  + | \int_{0}^{\Gamma^{z}(t, x)} \sum_{1\le k \le d} |\nabla_{x} \psi_{n_{1}, i}^{k}(t, x, r) |  d r | \sup _{(t, y)\in \tilde D} | \nabla_{y}\g_{n_{1}}^{j}(t, y)-\nabla_{y} \gamma_{n_{2}}^{j}(t, y) | \\
&\ \ \ \  + | \int_{0}^{\Gamma^{z}(t, x)} \sum_{1\le k \le d} |\nabla_{x} \psi_{n_{1}, i}^{k}(t, x, r)-\nabla_{x} \psi_{n_{2}, i}^{k}(t, x, r) |  d r | \sup_{(t,y)\in \tilde D,n\ge N_1} | \nabla_{y}\g_{n}^{j}(t, y)|\\
&\ \ \ \  +\sum_{ 1 \le k \leq d} |\partial_{y_{k}} \gamma_{n_1}^{j}\left(t, y_{n_1}\left(t, x, \Gamma^{z}(t, x)\right)\right)\psi_{n_1,i}^{k}\left(t, x, \Gamma^{z}\left(t, x\right)\right)\\
&\ \ \ \  \ \ \ \  \ \ \ \   \ \ \ -\partial_{y_{k}} \gamma_{n_2}^{j}\left(t, y_{n_2}\left(t, x, \Gamma^{z}(t, x)\right)\right)\psi_{n_2,i}^{k}\left(t, x, \Gamma^{z}\left(t, x\right)\right)| |\partial_{x_{m}} \Gamma^{z}(t, x)|\\
&\les  \int_{-\rho_1}^{\rho_1} \sum_{1 \le k, l \le d} |\p_{y_{l}} \partial_{y_{k}}  \g_{n_{1}}^{j}(t, y_{n_1}(t, x, r))- \p_{y_{l}} \partial_{y_{k}}  \g_{n_{2}}^{j}(t, y_{n_1}(t, x, r)) | I_{D(t,\varrho(\e))}(y_{n_1}(t, x, r)) dr \\
&\ \ \ \  +\int_{-\rho_1}^{\rho_1} \sum_{1 \leqslant k, l \leqslant d} |\p_{y_{l}} \partial_{y_{k}}  \g_{n_{2}}^{j}(t, y_{n_1}(t, x, r))- \p_{y_{l}} \partial_{y_{k}}  \g_{n_{2}}^{j}(t, y_{n_2}(t, x, r)) | \\
&\ \ \ \ \ \ \ \ \     \times  I_{D(t,\varrho(\e))}(y_{n_1}(t, x, r))  I_{D(t,\varrho(\e))}(y_{n_2}(t, x, r)) dr \\
&\ \ \ \  +\int_{-\rho_1}^{\rho_1} \sum_{1 \leqslant k, l \leqslant d} |\p_{y_{l}} \partial_{y_{k}}  \g_{n_{2}}^{j}(t, y_{n_2}(t, x, r))| I_{D(t,\varrho(\e))}(y_{n_2}(t, x, r))  dr \\
&\ \ \ \ \ \ \ \ \    \times  \sup _{(t,x)\in \O_{\e}, |r|< \rho_1, \atop 1\le m,l \le d} | \psi_{n_1, m}^{l}(t, x, r)-\psi_{n_2, m}^{l}(t, x, r)|\\
&\ \ \ \  + \sum_{ 1 \le k \leq d} | \int_{0}^{\Gamma^{z}(t, x)} |\nabla_{x} \psi_{n_{1}, i}^{k}(t, x, r) |  d r | \sup _{(t,y)\in \tilde D} | \nabla_{y}\g_{n_{1}}^{j}(t, y)-\nabla_{y} \gamma_{n_{2}}^{j}(t, y) | \\
&\ \ \ \  + \sum_{ 1 \le k \leq d} | \int_{0}^{\Gamma^{z}(t, x)} |\nabla_{x} \psi_{n_{1}, i}^{k}(t, x, r)-\nabla_{x} \psi_{n_{2}, i}^{k}(t, x, r) |  d r | \\
&\ \ \ \  +\sum_{ 1 \le k \leq d} |\partial_{y_{k}} \gamma_{n_1}^{j}\left(t, y_{n_1}\left(t, x, \Gamma^{z}(t, x)\right)\right)\psi_{n_1,i}^{k}\left(t, x, \Gamma^{z}\left(t, x\right)\right)\\
&\ \ \ \ \ \ \ \ \  -\partial_{y_{k}} \gamma_{n_2}^{j}\left(t, y_{n_2}\left(t, x, \Gamma^{z}(t, x)\right)\right)\psi_{n_2,i}^{k}\left(t, x, \Gamma^{z}\left(t, x\right)\right)|. \\
\end{split}
\end{eqnarray}
Combining this with Proposition \ref{P3.3}, Lemma \ref{L4.2}, Lemma \ref{L7.1}, (\ref{2.29}), (\ref{3.30}), (\ref{3.80}) and (\ref{3.81}), we obtain (\ref{3.34}).

\vskip 0.3cm

Finally we show (\ref{2.13}).
 From (\ref{4.4}), (\ref{3.33}), (\ref{3.80}) and (\ref{7.2}), we have
 $$\sup_{\e>0}\sup_{n\ge N_1(\e)}\|\psi_{n,i}^j(t,x, \G^z(t,x))\|_{W_{2d+2}^{0,1}(\O_{\e})}<\infty.$$
 Together with (\ref{3.30}) and (\ref{3.34}), we see that $\psi_{i}^j(t,x, \G^z(t,x))\in {W_{2d+2}^{0,1}(\O_{\e})}$ and
 $$\sup_{\e>0} \|\psi_{i}^j(t,x, \G^z(t,x))\|_{W_{2d+2}^{0,1}(\O_{\e})}<\infty.$$
Since $\bigcup_{\e>0}\O_{\e}=(((t_0-\eta_2)\vee 0,(t_0+\eta_2)\wedge T_1)\times C(z,\d_5))\bigcap \tilde D$, we have (\ref{2.13}).
\hfill $\blacksquare$

\vskip 0.6cm

\textbf{Proof of Proposition \ref{P4.2.1}}
% Obviously, from Proposition \ref{P3.3}, Lemma \ref{L3.7.1} and (\ref{3.29}), we have $\Gamma^z \in W_{2d+2}^{0,2}(((t_0-\eta_2)\vee 0,(t_0+\eta_2)\wedge T_1)\times C(z,\d_5)\bigcap \tilde D)$. In the remainder, we only need to show that $\Gamma^z \in W_{2d+2}^{1,0}(((t_0-\eta_2)\vee 0,(t_0+\eta_2)\wedge T_1)\times C(z,\d_5)\bigcap \tilde D)$.
\vskip 0.3cm
We first show that for any $\e>0$,
\begin{eqnarray} \label{3.45}
\lim_{n,m \to \infty}\|\int_{\Gamma_{n}^{z}(t, x)}^{\Gamma_{m}^{z}(t, x)} |(\partial_{t} \gamma_{n} ) (t, y_{n}(t, x, \tau) )| d \tau\|_{L^{2d+2}(\O_{\e})}=0.
\end{eqnarray}

By Lemma \ref{L4.2}, (\ref{3.48}) and the H\"{o}lder inequality, we have for any $M>0$ and $n\ge N_1(\e)$,
\begin{eqnarray} \nonumber
\begin{split}
& \ \ \ \ \|\int_{\Gamma_{n}^{z}(t, x)}^{\Gamma_{m}^{z}(t, x)} |(\partial_{t} \gamma_{n} ) (t, y_{n}(t, x, \tau) )| d \tau\|_{L^{2d+2}(\O_{\e})}\\
& \le \|\int_{\Gamma_{n}^{z}(t, x)}^{\Gamma_{m}^{z}(t, x)}| (\partial_{t} \gamma_{n} ) (t, y_{n}(t, x, \tau) )- (\partial_{t} \g ) (t, y_{n}(t, x, \tau) )| d \tau\|_{L^{2d+2}(\O_{\e})}\\
& \ \ \ \ +\|\int_{\Gamma_{n}^{z}(t, x)}^{\Gamma_{m}^{z}(t, x)}| (\partial_{t} \g ) (t, y_{n}(t, x, \tau) )| d \tau\|_{L^{2d+2}(\O_{\e})}\\
% & \le \|\int_{-\rho_1}^{\rho_1}| (\partial_{t} \gamma_{n} ) (t, y_{n}(t, x, \tau) )- (\partial_{t} \g ) (t, y_{n}(t, x, \tau) )|I_{(\Gamma_{n}^{z}(t, x),\Gamma_{m}^{z}(t, x))}(\tau) d \tau\|_{L^{2d+2}(\O_{\e})}\\
% & \ \ \ \ +\|\int_{\Gamma_{n}^{z}(t, x)}^{\Gamma_{m}^{z}(t, x)}| (\partial_{t} \g ) (t, y_{n}(t, x, \tau) )| d \tau\|_{L^{2d+2}(\O_{\e})}\\
& \le \|\int_{-\rho_1}^{\rho_1}| (\partial_{t} \gamma_{n} ) (t, y_{n}(t, x, \tau) )- (\partial_{t} \g ) (t, y_{n}(t, x, \tau) )|I_{D(t,\varrho(\e))}(y_{n}(t, x, \tau)) d \tau\|_{L^{2d+2}(\O_{\e})}\\
& \ \ \ \ +\|\int_{\Gamma_{n}^{z}(t, x)}^{\Gamma_{m}^{z}(t, x)}| (\partial_{t} \g ) (t, y_{n}(t, x, \tau) )| d \tau\|_{L^{2d+2}(\O_{\e})}\\
% & \lesssim  \|\partial_{t} \g-\partial_{t} \gamma_{n} \|_{L^{2d+2}(\O_{\e}')}+\|\int_{\Gamma_{n}^{z}(t, x)}^{\Gamma_{m}^{z}(t, x)}| (\partial_{t} \g ) (t, y_{n}(t, x, \tau) )| d \tau\|_{L^{2d+2}(\O_{\e})}\\
&\les  \|\partial_{t} \g_n  -  \partial_{t} \gamma  \|_{L^{2d+2}(\O_{\e}')} \\
& \ \ \ \ +  (\int_{(t_{0}-\eta_{2})\vee 0}^{(t_{0}+\eta_{2}) \wedge T_1}   \int_{C(z, \d_5) \bigcap D(t,\e)}   |   \int_{\Gamma_{n}^{z}(t, x)}^{\Gamma_{m}^{z}(t, x)}| (\partial_{t} \g ) (t, y_{n}(t, x, \tau) )|^{2d+2} d \tau | d x d t)^{\frac{1}{2d+2}} \\
&\le  \| \partial_{t} \g_n- \partial_{t} \gamma \|_{L^{2d+2}(\O_{\e}')}+M (\int_{(t_{0}-\eta_{2})\vee 0}^{(t_{0}+\eta_{2}) \wedge T_1} \int_{C(z, \d_5)} |\int_{\Gamma_{n}^{z}(t, x)}^{\Gamma_{m}^{z}(t, x)}d\tau | d x d t)^{\frac{1}{2d+2}} \\
&\ \ \ \ + (\int_{(t_{0}-\eta_{2})\vee 0}^{(t_{0}+\eta_{2}) \wedge T_1}    \int_{C(z, \d_5) \bigcap D(t,\e)}   \int_{-\rho_1}^{\rho_1}    |   (\partial_{t} \g ) (t, y_{n}(t, x, \tau) )I_{| (\partial_{t} \g ) (t, y_{n}(t, x, \tau) )|>M}   |^{2d+2} \\
&\ \ \ \ \ \ \ \ \times I_{D(t,\varrho(\e))}(y_{n}(t, x, \tau))d \tau d x d t )^{\frac{1}{2d+2}}  \\
&\lesssim  \|\partial_{t} \g_n-\partial_{t} \gamma \|_{L^{2d+2}(\O_{\e}')}+M (\int_{(t_{0}-\eta_{2})\vee 0}^{(t_{0}+\eta_{2}) \wedge T_1} \int_{C(z, \d_5)}|\Gamma_{n}^{z}(t, x)-\Gamma_{m}^{z}(t, x)| d x d t )^{2d+2} \\
&\ \ \ \ +(\int_{-\rho_1}^{\rho_1}\int_{(t_{0}-\eta_{2})\vee 0}^{(t_{0}+\eta_{2}) \wedge T_1} \int_{D(t,\varrho(\e))}| (\partial_{t} \g ) (t, x) I_{|(\p_t \g(t,x))|\ge M}|^{2d+2}  d x d t d \tau)^{\frac{1}{2d+2}} .
\end{split}
\end{eqnarray}
Combining this with Proposition \ref{P3.3} and (\ref{3.17}), we obtain that
\begin{eqnarray} \label{3.22}
\begin{split}
&\ \ \ \ \lim_{n,m \to \infty}\|\int_{\Gamma_{n}^{z}(t, x)}^{\Gamma_{m}^{z}(t, x)} |(\partial_{t} \gamma_{n} ) (t, y_{n}(t, x, \tau) )| d \tau\|_{L^{2d+2}(\O_{\e})} \\
&\lesssim (\int_{(t_{0}-\eta_{2})\vee 0}^{(t_{0}+\eta_{2}) \wedge T_1} \int_{D(t,\varrho(\e))}| (\partial_{t} \g ) (t, x)  I_{|(\p_t \g(t,x))|\ge M}|^{2d+2}  d x d t)^{\frac{1}{2d+2}} .
\end{split}
\end{eqnarray}
Since $\g \in W_{2d+2}^{1,2}(\tilde D)$, letting $M \to \infty$ in (\ref{3.22}), we get (\ref{3.45}).

\vskip 0.3cm

Next we show that for any $\e>0$,
\begin{eqnarray}  \label{2.31}
\begin{split}
\lim_{n,m \to \infty} \|   \int_0^{\Gamma_{m}^{z}(t, x)} | (\partial_{t} \g_n) (t, y_{n}(t, x, \tau)) \! -\! (\partial_{t} \g_m) (t, y_{m}(t, x, \tau))| d \tau \|_{L^{2d+2}(\! \O_{\e}\! )}\!=\!0.
\end{split}
\end{eqnarray}
For $(t,x)\in \O_{\e}$ and $n,m \ge N_1(\e)$, by (\ref{3.48}) we have
\begin{eqnarray} \nonumber
\begin{split}
& \ \ \ \  | \int_0^{\Gamma_{m}^{z}(t, x)} | (\partial_{t} \g_n) (t, y_{n}(t, x, \tau)) \! -\! (\partial_{t} \g_m) (t, y_{m}(t, x, \tau))| d \tau |\\
&\le |\int_0^{\Gamma_{m}^{z}(t, x)} | (\partial_{t} \g_{n}) (t, y_{n}(t, x, \tau)) -(\partial_{t} \g) (t, y_{n}(t, x, \tau))|d\tau  |\\
&\ \ \ \ + |\int_0^{\Gamma_{m}^{z}(t, x)} | (\partial_{t} \g_m) (t, y_{m}(t, x, \tau)) -(\partial_{t} \g) (t, y_{m}(t, x, \tau))| d\tau | \\
&\ \ \ \ + |\int_0^{\Gamma_{m}^{z}(t, x)} | (\partial_{t} \g) (t, y_{n}(t, x, \tau)) -(\partial_{t} \g) (t, y_{m}(t, x, \tau))| d\tau | \\
&\le \int_{-\rho_1}^{\rho_1}| (\partial_{t} \g_{n}) (t, y_{n}(t, x, \tau)) -(\partial_{t} \g) (t, y_{n}(t, x, \tau))|I_{D(t,\varrho(\e))}(y_{n}(t, x, \tau)) d\tau  \\
&\ \ \ \ + \int_{-\rho_1}^{\rho_1}| (\partial_{t} \g_{m}) (t, y_{m}(t, x, \tau)) -(\partial_{t} \g) (t, y_{m}(t, x, \tau))|I_{D(t,\varrho(\e))}(y_{m}(t, x, \tau)) d\tau  \\
&\ \ \ \ + \int_{-\rho_1}^{\rho_1}  | (\partial_{t} \g) (t, y_{n}(t, x, \tau)) -(\partial_{t} \g) (t, y_{m}(t, x, \tau))| I_{D(t,\varrho(\e))}(y_{n}(t, x, \tau)) I_{D(t,\varrho(\e))}(y_{m}(t, x, \tau)) d\tau, \\
\end{split}
\end{eqnarray}
which implies
\begin{eqnarray}  \nonumber
\begin{split}
& \ \ \ \ \|   \int_0^{\Gamma_{m}^{z}(t, x)} | (\partial_{t} \g_n) (t, y_{n}(t, x, \tau)) - (\partial_{t} \g_m) (t, y_{m}(t, x, \tau))| d \tau \|_{L^{2d+2}(  \O_{\e}  )}\\
&\les    (\int_{-\rho_1}^{\rho_1} \int_{(t_0-\eta_2)\vee 0}^{(t_0+\eta_2)\wedge T_1}  \int_{\Rd}| (\partial_{t} \g_{n} ) (t, y_{n}(t, x, \tau) )-(\partial_{t} \g) (t, y_{n}(t, x, \tau))| ^{2d+2} \\
&\ \ \ \ \ \ \ \ \times  I_{D(t,\varrho(\e))}(y_{n}(t, x, \tau))dxdtd\tau)^{\frac{1}{2d+2}} \\
&\ \ \ \ + (\int_{-\rho_1}^{\rho_1} \int_{(t_0-\eta_2)\vee 0}^{(t_0+\eta_2)\wedge T_1}  \int_{\Rd}| (\partial_{t} \g_{m} ) (t, y_{m}(t, x, \tau) )-(\partial_{t} \g) (t, y_{m}(t, x, \tau))| ^{2d+2} \\
&\ \ \ \ \ \ \ \ \times  I_{D(t,\varrho(\e))}(y_{m}(t, x, \tau))dxdtd\tau)^{\frac{1}{2d+2}}  \\
&\ \ \ \ + (\int_{-\rho_1}^{\rho_1} \int_{(t_0-\eta_2)\vee 0}^{(t_0+\eta_2)\wedge T_1}  \int_{\Rd}| (\partial_{t} \g) (t, y_{n}(t, x, \tau)) -(\partial_{t} \g) (t, y_{m}(t, x, \tau))|^{2d+2}\\
&\ \ \ \ \ \ \ \ \times   I_{D(t,\varrho(\e))}(y_{n}(t, x, \tau)) I_{D(t,\varrho(\e))}(y_{m}(t, x, \tau))dxdtd\tau )^{\frac{1}{2d+2}} \\
&\les  \|\partial_{t} \g_{n}  -\partial_{t} \g  \|_{L^{2d+2}(\O_{\e}')}+\|\partial_{t} \g_{m}  -\partial_{t} \g   \|_{L^{2d+2}(\O_{\e}')} \\
&\ \ \ \ + (\int_{-\rho_1}^{\rho_1} \int_{(t_0-\eta_2)\vee 0}^{(t_0+\eta_2)\wedge T_1}  \int_{\Rd}| (\partial_{t} \g) (t, y_{n}(t, x, \tau)) -(\partial_{t} \g) (t, y_{m}(t, x, \tau))|^{2d+2}\\
&\ \ \ \ \ \ \ \ \times   I_{D(t,\varrho(\e))}(y_{n}(t, x, \tau)) I_{D(t,\varrho(\e))}(y_{m}(t, x, \tau))dxdtd\tau )^{\frac{1}{2d+2}} \to 0, \\
\end{split}
\end{eqnarray}
as $n,m \to \infty$ by Proposition \ref{P3.3}, Lemma \ref{L4.2} and Lemma \ref{L7.1}.

\vskip 0.3cm

Now we show (\ref{4.50}), $i.e.$
\begin{eqnarray}\label{7.2.1}
\begin{split}
\sup_{\e>0} \sup_{n\ge N_1(\e)} \| \Lambda_n(t,x,\G_n^z(t,x)) \|_{L^{2d+2}(\O_{\e})}<\infty.
\end{split}
\end{eqnarray}
Applying the Gronwall's inequality to (\ref{3.10}), we get that for $(t,x)\in \O_{\e}$ and $n \ge N_1(\e)$,
\begin{eqnarray} \label{3.18}
\begin{split}
& \ \ \ \ \| \sup_{\tau \in (0,\Gamma_{n}^{z}(t, x)]\cup(\Gamma_{n}^{z}(t, x), \Gamma_{m}^{z}(t, x)]} |\Lambda_{n}(t, x, \tau)| \|_{L^{2d+2}(\O_{\e})}\\
&\les \| \sup_{\tau \in (0,\Gamma_{n}^{z}(t, x)]\cup(\Gamma_{n}^{z}(t, x), \Gamma_{m}^{z}(t, x)]} | \int_0^\tau| (\partial_{t} \g_{n} ) (t, y_{n}(t, x, \tau') )|  d\tau' | \|_{L^{2d+2}(\O_{\e})}\\
&\le \| \int_{-\rho_1}^{\rho_1}| (\partial_{t} \g_{n} ) (t, y_{n}(t, x, \tau') )| I_{D(t,\varrho(\e))}(y_{n}(t, x, \tau'))  d\tau'\|_{L^{2d+2}(\O_{\e})} \\
&\les (\int_{-\rho_1}^{\rho_1} \int_{(t_0-\eta_2)\vee 0}^{(t_0+\eta_2)\wedge T_1}  \int_{\Rd}| (\partial_{t} \g_{n} ) (t, y_{n}(t, x, \tau') )|^{2d+2} I_{D(t,\varrho(\e))}(y_{n}(t, x, \tau')) dxdtd\tau')^{\frac{1}{2d+2}} \\
&\les \| \partial_{t} \g_{n}   \|_{L^{2d+2}(\O'_{\e})} , \\
\end{split}
\end{eqnarray}
where the last inequality follows from Lemma \ref{L4.2}. Together with (\ref{3.79}), we have (\ref{7.2.1}).

\vskip 0.3cm

Next we show that for any $\e>0$,
\begin{eqnarray} \label{2-17-3}
\begin{split}
\lim_{n,m \to \infty} \| \Lambda_{n}(t, x, \Gamma_{m}^{z}(t, x))-\Lambda_m(t, x, \Gamma_{m}^{z}(t, x)) \|_{L^{2d+2}(\O_{\e})}=0.
\end{split}
\end{eqnarray}

Since for $ r \in \R$,
\begin{eqnarray} \nonumber
\begin{split}
& \ \ \ \ |\Lambda_{n}(t, x, r)-\Lambda_m(t, x, r)|\\
&\le | \int_0^r | (\partial_{t} \g_n) (t, y_{n}(t, x, \tau)) \! -\! (\partial_{t} \g_m) (t, y_{m}(t, x, \tau))| d \tau | \\
&\ \ \ \ + | \int_0^r |(\nabla_y \g_n(t,y_n (t, x, \tau))-\nabla_y \g(t,y_n(t,x,\tau))) \c \Lambda_{n}(t, x, \tau)|d\tau | \\
&\ \ \ \ + | \int_0^r|(\nabla_y \g_m(t,y_m (t, x, \tau))-\nabla_y \g(t,y_m(t,x,\tau))) \c \Lambda_{n}(t, x, \tau)|d\tau | \\
&\ \ \ \ + | \int_0^r |(\nabla_y \g(t,y_n(t, x, \tau))-\nabla_y \g(t,y_{m}(t, x, \tau))) \c \Lambda_{n}(t, x, \tau)|d\tau | \\
&\ \ \ \ +\sup_{(t,y) \in [0,T_1] \times \Rd}  |\nabla_y \g_m(t,y)| | \int_0^r|(\Lambda_{n}(t, x, \tau)-\Lambda_m(t, x, \tau))| d\tau |,
\end{split}
\end{eqnarray}
by the Gronwall's inequality and Lemma \ref{L3.4}, we see that for $(t,x)\in \O_{\e}$ and $n,m \ge N_1(\e)$,
\begin{eqnarray} \label{3.42}
\begin{split}
& \ \ \ \ |\Lambda_{n}(t, x, \Gamma_{m}^{z}(t, x))-\Lambda_m(t, x, \Gamma_{m}^{z}(t, x))|\\
&\les | \int_0^{\Gamma_{m}^{z}(t, x)} | (\partial_{t} \g_n) (t, y_{n}(t, x, \tau)) \! -\! (\partial_{t} \g_m) (t, y_{m}(t, x, \tau))| d \tau | \\
&\ \ \ \ +  |\int_0^{\Gamma_{m}^{z}(t, x)}|(\nabla_y \g_n(t,y_n (t, x, \tau))-\nabla_y \g(t,y_n(t,x,\tau))) \c \Lambda_{n}(t, x, \tau)|d\tau|\\
&\ \ \ \ + | \int_0^{\Gamma_{m}^{z}(t, x)}|(\nabla_y \g_m(t,y_m (t, x, \tau))-\nabla_y \g(t,y_m(t,x,\tau))) \c \Lambda_{n}(t, x, \tau)|d\tau|\\
&\ \ \ \ + | \int_0^{\Gamma_{m}^{z}(t, x)}|(\nabla_y \g(t,y_n(t, x, \tau))-\nabla_y \g(t,y_{m}(t, x, \tau))) \c \Lambda_{n}(t, x, \tau)|d\tau | \\
&\le | \int_0^{\Gamma_{m}^{z}(t, x)} | (\partial_{t} \g_n) (t, y_{n}(t, x, \tau)) \! -\! (\partial_{t} \g_m) (t, y_{m}(t, x, \tau))| d \tau | \\
&\ \ \ \ +  \rho_1 \sup_{(t,y) \in [0,T_1]\times \Rd}|\nabla_y \g_n(t,y)-\nabla_y \g(t,y)| \sup_{\tau \in (0, \Gamma_{m}^{z}(t, x)]} |\Lambda_{n}(t, x, \tau)|\\
&\ \ \ \ + \rho_1 \sup_{(t,y) \in [0,T_1]\times \Rd}|\nabla_y \g_m(t,y)-\nabla_y \g(t,y)| \sup_{\tau \in (0,  \Gamma_{m}^{z}(t, x)]} |\Lambda_{n}(t, x, \tau)| \\
&\ \ \ \ +   \rho_1  \sup_{(t,x) \in [0,T_1]\times \Rd , \atop  \tau \in (-\rho_1,\rho_1)}   | \nabla_y   \g(t,y_n(t, x, \tau))   -    \nabla_y    \g(t,y_{m}(t, x, \tau))|   \sup_{\tau \in (0, \Gamma_{m}^{z}(t, x)   ]}  |\Lambda_{n}(t, x, \tau)|.
\end{split}
\end{eqnarray}
Hence by Proposition \ref{P3.3}, (\ref{2.29}), (\ref{2.31}), (\ref{3.18}) and (\ref{3.42}),
\begin{eqnarray}  \label{3.41}  \nonumber
\begin{split}
& \ \ \ \ \|  \Lambda_{n}(t, x, \Gamma_{m}^{z}(t, x))-\Lambda_m(t, x, \Gamma_{m}^{z}(t, x)) \|_{L^{2d+2}(\O_{\e})}\\
&\les  \|     \int_0^{\Gamma_{m}^{z}(t, x)} | (\partial_{t} \g_n) (t, y_{n}(t, x, \tau)) - (\partial_{t} \g_m) (t, y_{m}(t, x, \tau))| d \tau  \|_{L^{2d+2}(  \O_{\e}  )}  \\
&\ \ \ \ +  \sup_{(t,y) \in [0,T_1]\times \Rd}|\nabla_y \g_n(t,y)-\nabla_y \g(t,y)| \| \partial_{t} \g_{n}   \|_{L^{2d+2}(\O'_{\e})} \\
&\ \ \ \ +  \sup_{(t,y) \in [0,T_1]\times \Rd}|\nabla_y \g_m(t,y)-\nabla_y \g(t,y)| \| \partial_{t} \g_{n}    \|_{L^{2d+2}(\O'_{\e})} \\
&\ \ \ \ +    \sup_{(t,x) \in [0,T_1]\times \Rd, \atop \tau \in (-\rho_1,\rho_1)}| \nabla_y \g(t,y_n(t, x, \tau))-\nabla_y \g(t,y_{m}(t, x, \tau))|  \| \partial_{t} \g_{n}    \|_{L^{2d+2}(\O'_{\e})} \\
%&\les   (  \int_{(t_0-\eta_2)\vee 0}^{(t_0+\eta_2)\wedge T_1} \int_{D(t,\varrho(\e))}| \partial_{t} \g_{n}  (t, x )-\partial_{t} \g (t, x )| ^{2d+2} dx dt)^{\frac{1}{2d+2}} \\
%&\ \ \ \ +  (  \int_{(t_0-\eta_2)\vee 0}^{(t_0+\eta_2)\wedge T_1} \int_{D(t,\varrho(\e))}|  \partial_{t} \g_{m}  (t, x )- \partial_{t} \g  (t, x )| ^{2d+2} dx dt)^{\frac{1}{2d+2}} \\
%&\ \ \ \ + (\int_{-\rho_1}^{\rho_1} \int_{(t_0-\eta_2)\vee 0}^{(t_0+\eta_2)\wedge T_1}  \int_{\Rd}| (\partial_{t} \g) (t, y_{n}(t, x, \tau)) -(\partial_{t} \g) (t, y_{m}(t, x, \tau))|\\
%&\ \ \ \ \ \ \ \ \times   I_{D(t,\varrho(\e))}(y_{n}(t, x, \tau)) I_{D(t,\varrho(\e))}(y_{m}(t, x, \tau))dxdtd\tau )^{\frac{1}{2d+2}}  \\
%&\ \ \ \ +  \sup_{(t,y) \in [0,T_1]\times \Rd}|\p_y \g_n(t,y)-\p_y \g(t,y)| (\int_{(t_0-\eta_2)\vee 0}^{(t_0+\eta_2)\wedge T_1}  \int_{D(t,\varrho(\e))}| \partial_{t} \g_{n}   (t, x )|^{2d+2} dxdt)^{\frac{1}{2d+2}}\\
%&\ \ \ \ +  \sup_{(t,y) \in [0,T_1]\times \Rd}|\p_y \g_m(t,y)-\p_y \g(t,y)|(\int_{(t_0-\eta_2)\vee 0}^{(t_0+\eta_2)\wedge T_1}  \int_{D(t,\varrho(\e))}|  \partial_{t} \g_{n}  (t, x )|^{2d+2} dxdt)^{\frac{1}{2d+2}}\\
%&\ \ \ \ +    \sup_{(t,x) \in [0,T_1]\times \Rd, \tau \in (-\rho_1,\rho_1)}| \p_y \g(t,y_n(t, x, \tau))-\p_y \g(t,y_{m}(t, x, \tau))|  \\
%&\ \ \ \ \ \ \ \ \times (\int_{(t_0-\eta_2)\vee 0}^{(t_0+\eta_2)\wedge T_1}  \int_{D(t,\varrho(\e))}| \partial_{t} \g_{n}  (t, x )|^{2d+2} dxdt)^{\frac{1}{2d+2}}\\
&\les \|   \int_0^{\Gamma_{m}^{z}(t, x)}  | (\partial_{t} \g_n) (t, y_{n}(t, x, \tau)) - (\partial_{t} \g_m) (t, y_{m}(t, x, \tau))| d \tau \|_{L^{2d+2}(  \O_{\e}  )} \\
&\ \ \ \ +  \sup_{(t,y) \in [0,T_1]\times \Rd}|\nabla_y \g_n(t,y)-\nabla_y \g(t,y)|+  \sup_{(t,y) \in [0,T_1]\times \Rd}|\nabla_y \g_m(t,y)-\nabla_y \g(t,y)|\\
&\ \ \ \ +    \sup_{(t,x) \in [0,T_1]\times \Rd, \atop \tau \in (-\rho_1,\rho_1)}| \nabla_y \g(t,y_n(t, x, \tau))-\nabla_y \g(t,y_{m}(t, x, \tau))| \to 0, \\
\end{split}
\end{eqnarray}
as $n,m \to \infty$.

\vskip 0.3cm

Now we show (\ref{4.51}), $i.e.$
\begin{eqnarray}
\lim_{n,m \to \infty} \| \Lambda_n (t, x, \G_{n}^{z}(t, x) )  -\Lambda_m (t, x, \G_{m}^{z}(t, x) ) \|_{L^{2d+2}(\O_{\e})}=0. \label{2-18-1}
\end{eqnarray}
By (\ref{2-17-3}), to prove (\ref{2-18-1}), we need only to prove that
\begin{eqnarray} \label{2-18-2}  \nonumber
\lim_{n,m \to \infty} \| \Lambda_n (t, x, \G_{n}^{z}(t, x) )  -\Lambda_n (t, x, \G_{m}^{z}(t, x) ) \|_{L^{2d+2}(\O_{\e})}=0.
\end{eqnarray}
By the definition of $\Lambda_n (t,x,r)$, (\ref{3.17}), (\ref{3.45}) and (\ref{3.18}), we have
\begin{eqnarray} \nonumber
\begin{split}
& \ \ \ \ \ \|  \Lambda_n (t, x, \G_{n}^{z}(t, x) )  -\Lambda_n (t, x, \G_{m}^{z}(t, x) ) \|_{L^{2d+2}(\O_{\e})} \\
&=\| \int_{\G_{n}^{z}(t, x)}^{\G_{m}^{z}(t, x)}  ( (\partial_{t} \g_{n} ) (t, y_{n}(t, x, r) )+\partial_{y} \g_{n} (t, y_{n}(t, x, r) )\Lambda_{n}(t, x, r) )dr  \|_{L^{2d+2}(\O_{\e})}\\
&\les \| \int_{\G_{n}^{z}(t, x)}^{\G_{m}^{z}(t, x)}  |(\partial_{t} \g_{n} ) (t, y_{n}(t, x, r) ) |dr  \|_{L^{2d+2}(\O_{\e})}  \\
& \ \ \ \ \ +\| \sup_{r \in  (\Gamma_{n}^{z}(t, x), \Gamma_{m}^{z}(t, x)] } |\Lambda_{n}(t, x, r) |  \|_{L^{2d+2}(\O_{\e})}  \sup_{(t,x)\in \O_{\e}} | \G_{n}^{z}(t, x)-\G_{m}^{z}(t, x) | \to 0,
\end{split}
\end{eqnarray}
as $n,m \to \infty$.

\vskip 0.3cm

Finally we show (\ref{4.52}), $i.e.$
\begin{eqnarray} \label{3.76}
\lim_{n,m\to \infty} \|\p_t\G_n^z(t,x)-\p_t\G_m^z(t,x)\|_{L^{2d+2}(\O_{\e})}=0.
\end{eqnarray}
%Recall $F_n$ is defined in (\ref{2.30}), we have $\partial_{r} F_{n}(t, x, r)=\g_{n}(t, y_n(t,x,r))\c \g(t_0, z)$ and
%\begin{eqnarray} \label{3.43}  \nonumber
%\begin{split}
%\partial_{t} F_{n}(t, x, r)=\Lambda_{n}(t, x, r)\c \g(t_0, z).
%% &=\sum_{1 \leqslant i \leqslant d} \int_{0}^{r}[ (\partial_{t} \gamma_{n} ) (t, y_{n}(t, x, \tau) )+\partial_{y} \g_{n} (t, y_{n}(t, x, \tau) ) \Lambda_{n}(t, x, \tau) ] d \tau \c \g(t_0, z).
%\end{split}
%\end{eqnarray}
By (\ref{3.14}) and (\ref{3.13}), we have for $(t,x)\in \O_{\e}$ and $n,m \ge N_1$,
\begin{eqnarray} \label{3.27}   \nonumber
\begin{split}
&\ \ \ \ |\p_t\G_n^z(t,x)-\p_t\G_m^z(t,x)|\\
&\le|\frac{\Lambda_n (t,x,\G^z_n(t,x)) - \Lambda_m(t,x,\G^z_m(t,x))}{\g_n(t,y_n(t,x,\G_n^z(t,x)))\c \g(t_0,z)}|\\
&\ \ \ \ +|(\frac{1}{\g_n(t,y_n(t,x,\G_n^z(t,x)))\c \g(t_0,z)}-\frac{1}{\g_m(t,y_m(t,x,\G_m^z(t,x)))\c \g(t_0,z)}) \Lambda_m(t,x,\G^z_m(t,x))|\\
&\le \cos ^{-1} \theta  |\Lambda_{n} (t, x, \Gamma_{n}^{z}(t, x) )-\Lambda_{m} (t, x, \Gamma_{m}^{z}(t, x) ) | \\
& \ \ \ \ + \cos ^{-2}\theta | \g_n(t,y_n(t,x,\G_n^z(t,x)))-\g_m(t,y_m(t,x,\G_m^z(t,x))) |   | \Lambda_{m}(t,x,\G^z_m(t,x))| \\
& \les  |\Lambda_{n} (t, x, \Gamma_{n}^{z}(t, x) )-\Lambda_{m} (t, x, \Gamma_{m}^{z}(t, x) ) | \\
& \ \ \ \ + |\Lambda_{m} (t,x,\G^z_m(t,x))| \sup_{(t,x)\in \O_{\e}}  | \g_n(t,y_n(t,x,\G^z_n(t,x))) - \g_m(t,y_m(t,x,\G^z_m(t,x)))|.
\end{split}
\end{eqnarray}
So together with (\ref{2-17-2}), (\ref{2.29}), (\ref{3.17}), (\ref{7.2.1}) and (\ref{2-18-1}), we obtain (\ref{3.76}).
\hfill $\blacksquare$

\vskip 0.4cm

\vskip 0.4cm

\noindent{\bf  Acknowledgement.}\   This work is partly
supported by National Natural Science Foundation of China (No. 11671372, No. 11971456,  No. 11721101).
%%%%%%%%%%%%%%%%%%%%%%%%%%%%%%%%%%%%%%%%%%%%%%%%%%%%%%%%%%%%%%%%%%%%%%%%%%%%%%%%%%%%%%%


\begin{thebibliography}{10}

%\bibitem{Bo}
%Bo L, Wang Y. On one-dimensional reflecting stochastic differential equations with non-Lipschitz coefficients[J]. Preprint, 2005.


% \bibitem{Briand} Briand P, Delyon B, Hu Y, Pardoux E, Stoica L. $L^{2d+2}$ solutions of backward stochastic differential equations. Stochastic Processes and their Applications, 2003, 108(1): 109-129.

% \bibitem{ChenSunZhang} Chen C Z, Sun W, Zhang J. Probabilistic representations of solutions of elliptic boundary value problem and non-symmetric semigroups. Journal of Differential Equations, 2016, 260(1): 26--55.
%\bibitem{Chen}
%Chen Z Q. On reflecting diffusion processes and Skorokhod decompositions. Probability theory and related fields, 1993, 94(3): 281-315.
% \bibitem{ChenSong} Chen Z Q, Song R. General gauge and conditional gauge theorems. Annals of probability, 2002: 1313-1339.


\bibitem{Dupuis1}
Dupuis P, Ishii H. On oblique derivative problems for fully nonlinear second-order elliptic partial differential equations on nonsmooth domains. Nonlinear Analysis: Theory, Methods \& Applications, 1990, 15(12): 1123-1138.

%\bibitem{Dupuis2}
%Dupuis P, Ishii H. On oblique derivative problems for fully nonlinear second-order elliptic PDE's on domains with corners. Hokkaido Mathematical Journal, 1991, 20(1): 135-164.

\bibitem{Dupuis3}
Dupuis P, Ishii H. SDEs with oblique reflection on nonsmooth domains. The Annals of Probability, 1993: 554-580.



\bibitem{Gilbarg}
Gilbarg D, Trudinger N S. Elliptic partial differential equations of second order. Springer, 2015.

\bibitem{Gyongy}
Gyöngy I, Martínez T. On stochastic differential equations with locally unbounded drift. Czechoslovak Mathematical Journal, 2001, 51(4): 763-783.

\bibitem{Hsu}
Hsu P. Reflecting Brownian motion, boundary local time, and the Neumann boundary value problem. Ph. D. dissertation, Stanford, 1984.

\bibitem{Lady}
Ladyženskaja O A, Solonnikov V A, Ural'ceva N N. Linear and quasi-linear equations of parabolic type. American Mathematical Soc., 1968.

\bibitem{Lieberman}
Lieberman G M. Second order parabolic differential equations. World scientific, 1996.
%\bibitem{Lieberman}
% Lieberman G M. Second order parabolic differential equations. World scientific, 1996.

\bibitem{Lions}
Lions P L, Sznitman A S. Stochastic differential equations with reflecting boundary conditions. Communications on Pure and Applied Mathematics, 1984, 37(4): 511-537.

\bibitem{Lundstrom}
Lundstr\"om N L P, \"Onskog T. Stochastic and partial differential equations on non-smooth time-dependent domains. Stochastic Processes and Their Applications, 2019, 129(4): 1097-1131.


\bibitem{Marin}
Marín-Rubio P, Real J. Some results on stochastic differential equations with reflecting boundary conditions. Journal of Theoretical Probability, 2004, 17(3): 705-716.


\bibitem{Krylov}
Krylov N V. On estimates of the maximum of a solution of a parabolic equation and estimates of the distribution of a semimartingale. Mathematics of the USSR-Sbornik, 1987, 58(1): 207.

\bibitem{Krylov2}
Krylov N V, Roeckner M. Strong solutions of stochastic equations with singular time dependent drift. Probability Theory and Related Fields, 2005, 131(2): 154-196.

\bibitem{Scheutzow}
Scheutzow M. A stochastic Gronwall's lemma. Infinite Dimensional Analysis, Quantum Probability and Related Topics, 2013, 16(02): 1350019.
% \bibitem{Ma} Ma Z M, Rockner M. Introduction to the theory of (non-symmetric) Dirichlet forms. Springer Science \& Business Media, 2012.

%\bibitem{Oshima}
%Oshima Y. Semi-Dirichlet forms and Markov processes. Walter de Gruyter, 2013.

%\bibitem{Peng2} Pardoux E, Peng S. Backward stochastic differential equations and quasilinear parabolic partial differential equations. In: Stochastic partial differential equations and their applications (Lect. Notes Control Inf. Sci. 176), 200-217, Springer, Berlin, 1992.

%\bibitem{Peng1} Peng S. Probabilistic interpretation for systems of quasilinear parabolic partial differential equations. Stochastics Stochastics Rep, 1991, 37(1-2): 61-74.

\bibitem{Tanaka}
Tanaka H. Stochastic differential equations with reflecting boundary condition in convex regions. Stochastic Processes: Selected Papers of Hiroshi Tanaka. 2002: 157-171


\bibitem{Veretennikov}
Veretennikov A J. On strong solutions and explicit formulas for solutions of stochastic integral equations. Sbornik: Mathematics, 1981, 39: 387-403.

%\bibitem{YangZhang} Yang S, Zhang T. Backward Stochastic Differential Equations and Dirichlet Problems of Semilinear Elliptic Operators with Singular Coefficients. Potential Analysis, 2017: 1-21.

%\bibitem{YangZhang} Yang X, Zhang T. Mixed boundary value problems of semilinear elliptic PDEs and BSDEs with singular coefficients. Stochastic Processes and their Applications, 2014, 124(7): 2442-2478.

\bibitem{ZhangX2}
Xie L, Zhang X. Sobolev differentiable flows of SDEs with local Sobolev and super-linear growth coefficients. The Annals of Probability, 2016, 44(6): 3661-3687.

\bibitem{ZhangX1}
Zhang X. Strong solutions of SDEs with singular drift and Sobolev diffusion coefficients. Stochastic Processes and Their Applications, 2005, 115(11): 1805-1818.

\bibitem{ZhangX}
Zhang X, Zhao G. Singular Brownian Diffusion Processes. Communications in Mathematics and Statistics, 2018, 6(4): 533-581.


\bibitem{Zhang}
Zhang T S. On the strong solutions of one-dimensional stochastic differential equations with reflecting boundary. Stochastic Processes and Their Applications, 1994, 50(1): 135-147.



\bibitem{Zvonkin}
Zvonkin A K. A transformation of the phase space of a diffusion process that removes the drift. Mathematics of the USSR-Sbornik, 1974, 22(1): 129.

% \bibitem{ZhangQ1}
% Zhang Q S. The boundary behavior of heat kernels of Dirichlet Laplacians. Journal of Differential Equations, 2002, 182(2): 416-430.


\end{thebibliography}
\end{document}